\documentclass[11pt]{amsart}
\usepackage{amsfonts}
\usepackage{amssymb}
\usepackage{amsmath}
\usepackage{amsthm}

\normalsize

\newtheorem{theorem}{Theorem}
\newtheorem{example}{Example}
\newtheorem{remark}{Remark}
\newtheorem{question}{Question}

\newtheorem{lemma}{Lemma}
\newtheorem{proposition}{Proposition}

\newtheorem{corollary}{Corollary}

\theoremstyle{definition}
\newtheorem{definition}{Definition}

\theoremstyle{definition}

\begin{document}

\title[]{Teichm\"uller Structures and Dual Geometric Gibbs Type Measure Theory for Continuous Potentials}

\author{Yunping Jiang}

\address{Department of Mathematics\\
Queens College of the City University of New York\\
Flushing, NY 11367-1597\\
and\\
Department of Mathematics\\
Graduate School of the City University of New York\\
365 Fifth Avenue, New York, NY 10016}
\email{yunping.jiang@qc.cuny.edu}

\thanks{2010 Mathematics Subject Classification. Primary 37D53;30F60, Secondary 37F10;30C62}

\keywords{circle endomorphism, symbolic space, dual symbolic
space, dual derivative, dual Gibbs measure, quasisymmetic
homeomorphism, symmetric homeomorphism}

\thanks{The research is partially supported by PSC-CUNY awards.}

\begin{abstract}
The Gibbs measure theory for smooth potentials is an old and
beautiful subject and has many important applications in modern
dynamical systems. For continuous potentials, it is impossible to
have such a theory in general. However, we develop a dual
geometric Gibbs type measure theory for certain continuous
potentials in this paper following some ideas and techniques from
Teichm\"uller theory for Riemann surfaces. Furthermore, we prove
that the space of those continuous potentials has a Teichm\"uller
structure. Moreover, this Teichm\"uller structure is a complete
structure and is the completion of the space of smooth potentials
under this Teichm\"uller structure. Thus our dual geometric Gibbs
type theory is the completion of the Gibbs measure theory for smooth
potentials from the dual geometric point of view.
\end{abstract}

\maketitle

\section{Introduction}

Starting from the celebrated work of Sinai~\cite{Sinai1,Sinai2}
and Ruelle~\cite{Ruelle1,Ruelle2}, a mathematical theory of Gibbs
states, an  important idea originally from physics, became an
important research topic in modern dynamical systems. Later,
Bowen~\cite{Bowen} brought Sinai and Ruelle's work into the study
of Axiom A dynamical systems. Their work finally led to a
definition of an SRB measure for a dynamical system. A very
important feature of a Gibbs measure (or an SRB measure) is that
it is an equilibrium state.

In the original study of Gibbs measures, a potential must be
smooth, which means it must be at least $C^{\alpha}$ for some
$0<\alpha\leq 1$. Later the smoothness condition was relaxed to
the summability condition in Walters' paper~\cite{Walters} (see
also~\cite{FanJiang}) but it is essentially the same as the smooth
case. For a long time, I have been interested in a study of a
Gibbs type theory for continuous potentials. But this is
impossible in general. However, we will show that it is possible
for a certain class of continuous potentials if we bring in some
ideas and techniques from Teichm\"uller theory and quasiconformal
mapping theory.

A basic idea in the Teichm\"uller theory is to use measurable
coordinates to view Riemann surfaces. That is, by fixing a Riemann
surface, all other Riemann surfaces homeomorphic to this Riemann
surface can be viewed from measurable coordinates on this Riemann
surface up to isotopy. A fundamental result in the study of
Teichm\"uller theory for Riemann surfaces is the measurable
Riemann mapping theorem as we describe below.

A measurable function $\mu$ on the Riemann sphere $\hat{\mathbb
C}$ is called a Beltrami coefficient if its $L^{\infty}$-norm
$k=\|\mu\|_{\infty} < 1$. The corresponding equation
$H_{\overline{z}}=\mu H_{z}$ is called the Beltrami equation. The
measurable Riemann mapping theorem says that the Beltrami equation
has a solution $H$ which is a quasiconformal homeomorphism of
$\hat{\mathbb C}$ whose quasiconformal dilatation is less than or
equal to $K=(1+k)/(1-k)$. It is called a $K$-quasiconformal
homeomorphism.

The study of the measurable Riemann mapping theorem has a long
history since Gauss considered in the 1820's the connection with
the problem of finding isothermal coordinates for a given surface.
As early as 1938, Morrey~\cite{Morrey} systematically studied
homeomorphic $L^{2}$-solutions of the Beltrami equation. But it
took almost twenty years until in 1957 Bers~\cite{Bers} observed
that these solutions are quasiconformal (refer to~\cite[p.
24]{Lehto}). Finally the existence of a solution to the Beltrami
equation under the most general possible circumstance, namely, for
measurable $\mu$ with $\|\mu\|_{\infty}<1$, was shown by
Bojarski~\cite{Bojarski} and by Ahlfors and
Bers~\cite{AhlforsBers}. In this generality the existence theorem
is sometimes called the measurable Riemann mapping theorem.

In this paper, we will borrow many ideas and techniques in the
Teichm\"uller theory and the quasiconformal mapping theory to
develop a Gibbs type measure theory for certain continuous
potentials. We will prove that the space of these continuous
potentials have Teichm\"uller structures. We will prove that for
such a continuous potential, there is a Gibbs type measure which
is an equilibrium state. Some properties about these Gibbs type measure are also studied.

We organize the paper as follows. In \S2, we define a uniformly
symmetric circle endomorphism and prove three examples. In particular,
for the third example, we mention and prove a more general version of the result
which we proved in~\cite{Jiang9}. This more general result (Theorem~\ref{KatokConj})
says that a $C^{1}$ circle endomorphism H\"older conjugate to a topologically expanding
circle endomorphism itself is expanding.

In \S3, we
review some classic results in dynamical systems which eventually
imply that there is only one topological model for the dynamics of
all circle endomorphisms of the same degree. In the same section,
we study the bounded nearby geometric property. The conclusion of
this property is that a conjugacy is quasisymmetric. This enables
us to define a Teichm\"uller structure on a space of circle
endomorphisms.

In \S4, we define the dual symbolic space and
geometrical models defined on it which we call dual derivatives.

In \S5, we define the Teichm\"uller space of smooth expanding
circle endomorphisms and the Teichm\"uller space of uniformly
symmetric circle endomorphisms. Furthermore, we prove that the
first Teichm\"uller space equals the space of all H\"older
continuous dual derivatives and the second Teichm\"uller space
equals the space of all continuous dual derivatives. Moreover, the
second one is the completion of the first one under the
Teichm\"uller metric. To prove this result and make this paper self-contained,
we state a special case (Theorem~\ref{onepointrigidity}) of the result about differentiable rigidity, which has being developed
in~\cite{Jiang1,Jiang3,Jiang4,Jiang5} for a more general situation. For the sake of the completeness of this paper,
we give a detailed proof. We also state the main result (Theorem~\ref{uaatouac}) in~\cite{GardinerJiangAAACCE} in \S6. That is, in \S6,
we first define an asymptotically conformal circle endomorphism and prove that a circle endomorphism is uniformly symmetric if and only if it is asymptotically conformal. After this, we prove the completion result in the end of \S6.

In \S7, we prove that the the Teichm\"uller space of uniformly
symmetric circle endomorphisms is contractible. In a remark in this section, we also state and give a outline of the proof about the contractility of the space of all $C^{1+\alpha}$ circle expanding circle endomorphisms and the Teichm\"uller space of smooth expanding
circle endomorphisms.

In \S8 and \S9, we define the linear model
for a uniformly symmetric circle endomorphism. We study the
relation between the linear model and the dual derivative.
We use this relation to set up a one-to-one correspondence between the Teichm\"uller space of
uniformly symmetric circle endomorphisms and the space of all continuous dual derivatives.
Furthermore, we give a characterization of a dual derivative.

In \S10, we define the maximum distance on the Teichm\"uller space of uniformly
symmetric circle endomorphisms and compare this maximum distance with the Teichm\"uler distance.

In \S11, we give a brief review of the Gibbs
measure theory for the smoothness case and define a dual invariant
measure. In the same section, we post several questions which we
study in this paper. In \S12, we give a review of the $g$-measure
theory. In \S13, we returned to the Gibbs measure theory for the
smoothness case but from the dual geometric point of view.

Finally, in \S14, we prove the existence of a dual geometric Gibbs
type measure for every continuous potential in the Teichm\"uller
space of uniformly symmetric circle endomorphisms. This measure
can be viewed as a coordinate structure such that the dynamical
system is smooth under this structure. Note that we start from a
uniformly symmetric circle endomorphism which may be very
singular. Most important, this measure is an equilibrium state.
This result could be served as the role of the Riemann mapping theorem
on the dual symbolic space.

In \S15, we study values of metric entropy for the Teichm\"uller space of uniformly
symmetric circle endomorphisms. The maximum value of the metric entropy is $log d$,
which is the topological entropy. We prove that the infimum of the metric entropy
for the Teichm\"uller space of uniformly symmetric circle endomorphisms is zero.

\vspace*{20pt}

\noindent {\bf Acknoledgement.} During this research, I have had
many conversations with Fred Gardiner, Aihua Fan, Guizhen Cui,
Jihua Ma, Anthony Quas, and Huyi Hu. I also learned many
techniques which I used in this paper from Dennis Sullivan during
his many lectures at the CUNY Graduate Center. I would like to
express my sincere thanks to everyone. This research is partially
supported by grants from NSF, PSC-CUNY, and Bai Ren Ji Hua from
the Chinese Academy of Sciences.

\section{Circle endomorphisms}

Let $T=\{ z\in {\mathbb C}\; |\; |z|=1\}$ be the unit circle in
the complex plane ${\mathbb C}$. Suppose
$$
f: T\to T
$$
is an orientation-preserving covering map of degree $d\geq 2$. We
call it in this paper a circle endomorphism. Suppose
$$
h: T\to T
$$
is an orientation-preserving homeomorphism. We call it in this
paper a circle homeomorphism.

For a circle endomorphism $f$, it has a fixed point. We will
assume throughout this paper that $f(1)=1$.

The universal cover of $T$ is the real line ${\mathbb R}$ with a
covering map
$$
\pi (x) = e^{2\pi i x}: {\mathbb R} \to T.
$$
Then every circle endomorphism $f$ can be lifted to an
orientation-preserving homeomorphism
$$
F: {\mathbb R}\to {\mathbb R}, \quad F(x+1)=F(x)+d, \quad \forall
x\in {\mathbb R}.
$$
We will assume throughout this paper that $F(0)=0$. Then there is
a one-to-one correspondence between $f$ and $F$. Therefore, we
also call such an $F$ a circle endomorphism.

Every orientation-preserving circle homeomorphism $h$ can be
lifted to an orientation-preserving homeomorphism
$$
H: {\mathbb R}\to {\mathbb R}, \quad H(x+1)=H(x)+1.
$$
We will assume throughout this paper that $0\leq H(0)<1$. Then
there is a one-to-one correspondence between $h$ and $H$.
Therefore, we also call such an $H$ a circle homeomorphism.

A circle endomorphism $f$ is $C^{k}$ for $k\geq 1$ if the
$k^{th}$-derivative $F^{(k)}$ exists and is continuous. And,
furthermore, it is called $C^{k+\alpha}$ for some $0< \alpha \leq
1$ if $F^{(k)}$ is $\alpha$-H\"older continuous, that is,
$$
\sup_{x\neq y\in {\mathbb R}} \frac{|F^{(k)} (x) -F^{(k)}
(y)|}{|x-y|^{\alpha}} =\sup_{x\neq y\in [0,1]} \frac{|F^{(k)} (x)
-F^{(k)} (y)|}{|x-y|^{\alpha}} <\infty.
$$
A $C^{1}$ circle endomorphism $f$ is called expanding if there are
constants $C>0$ and $\lambda >1$ such that
$$
(F^{n})'(x) \geq C\lambda^{n}, \quad n=1, 2, \cdots.
$$

A circle homeomorphism $h$ is called quasisymmetric if there is a
constant $M\geq 1$ such that
$$
M^{-1} \leq \frac{|H(x+t)-H(x)|}{|H(x)-H(x-t)|}\leq M, \quad
\forall x\in {\mathbb R}, \; \forall t>0.
$$
Furthermore, it is called symmetric if there is a bounded function
$\varepsilon (t)>0$ for $t>0$ such that $\varepsilon (t) \to
0^{+}$ as $t\to 0^{+}$ and such that
$$
\frac{1}{1+\varepsilon (t)} \leq
\frac{|H(x+t)-H(x)|}{|H(x)-H(x-t)|}\leq 1+\varepsilon (t), \quad
\forall x\in {\mathbb R}, \; \forall t>0.
$$

\vspace{10pt}
\begin{example}~\label{smoothimplysymmetry}
A $C^{1}$-diffeomorphism of $T$ is symmetric.
\end{example}

However, the class of symmetric homeomorphisms is larger than the
class of $C^{1}$-diffeomorphisms. For example, a symmetric
homeomorphism may not necessarily be absolutely continuous.

\vspace*{10pt}
\begin{definition}~\label{uniformlysymmetry}
A circle endomorphism $f$ is called uniformly symmetric if there
is a bounded function $\varepsilon (t)>0$ for $t>0$ such that
$\varepsilon (t) \to 0^{+}$ as $t\to 0^{+}$ and such that
$$
\frac{1}{1+\varepsilon (t)} \leq
\frac{|F^{-n}(x+t)-F^{-n}(x)|}{|F^{-n}(x)-F^{-n}(x-t)|}\leq
1+\varepsilon (t)
$$
for all $x\in {\mathbb R}$, all $t>0$, and all $n>0$.
\end{definition}

\vspace*{10pt}
\begin{example}~\label{holderexample}
A $C^{1+\alpha}$, for some $0<\alpha\leq 1$, circle expanding
endomorphism $f$ is uniformly symmetric. Furthermore,
$\varepsilon(t)\leq D t^{\alpha}$ for some constant $D>0$ and
$0\leq t\leq 1$.
\end{example}

\begin{proof}
Since $F(x+1) = F(x)+d$, then $F'(x+1)=F'(x)$ is a periodic
function. Since $F$ is $C^{1+\alpha}$, we have a constant
$C_{1}>0$ such that
$$
| F'(x)-F'(y) |\leq C_{1} |x-y|^{\alpha}, \quad \forall x,y\in
{\mathbb R}.
$$
Since $F$ is expanding, we have a constant $C_{2}>0$ and
$\lambda>1$ such that
$$
(F^{n})'(x) \geq C_{2}\lambda^{n}, \quad \forall x\in {\mathbb
R},\;\; n>0.
$$

For any $x, y\in {\mathbb R}$ and $n>0$, let $x_{k}=F^{-k}(x)$ and
$y_{k}=F^{-k}(y)$, $0\leq k\leq n$. Then
$$
\Big| \log \frac{(F^{-n})'(x)}{(F^{-n})'(y)} \Big|= \Big| \log
\frac{(F^{n})'(y_{n})}{(F^{n})'(x_{n})}\Big| \leq \sum_{k=1}^{n} |
\log F'(x_{k})-\log F'(y_{k})|
$$
$$
\leq \frac{1}{C_{2}\lambda} \sum_{k=1}^{n} | F'(x_{k})-F'(y_{k})|
\leq \frac{C_{1}}{C_{2}\lambda} \sum_{k=1}^{n}
|x_{k}-y_{y}|^{\alpha} \leq \frac{C_{1}}{C_{2}^{1+\alpha}\lambda}
\sum_{k=1}^{n} \lambda^{-\alpha k} |x-y|^{\alpha}.
$$
Let
$$
C = \frac{C_{1}\lambda^{\alpha}}{C_{2}^{1+\alpha}
(\lambda^{\alpha}-1)\lambda}.
$$
Then we have the following H\"older distortion property:
\begin{equation}~\label{smoothdistortion}
e^{-C |x-y|^{\alpha}} \leq \frac{(F^{-n})'(x)}{(F^{-n})'(y)}\leq
e^{C|x-y|^{\alpha}}, \quad \forall x, y \in {\mathbb R}, \;\;
\forall n>0.
\end{equation}
Furthermore, let
$$
\varepsilon (t) = \left\{ \begin{array}{ll} e^{Ct^{\alpha}}-1,
& 0<t\leq 1,\\
e^{C}-1, & t>1.
\end{array}\right.
$$
Then $\varepsilon (t)>0$ is a bounded function such that
$\varepsilon (t) \to 0$ as $t\to 0^{+}$ and such that
$$
\frac{1}{1+\varepsilon (t)} \leq
\frac{(F^{-n})'(\xi)}{(F^{-n})'(\eta)} =
\frac{|F^{-n}(x+t)-F^{-n}(x)|}{|F^{-n}(x)-F^{-n}(x-t)|}\leq
1+\varepsilon (t)
$$
for all $x\in {\mathbb R}$, all $t>0$, and all $n>0$,
where $\xi$ and $\eta$ are two numbers in $[0,1]$. Thus $F$ is
uniformly symmetric. Furthermore, one can see that
$\varepsilon(t)\leq D t^{\alpha}$ for some constant $D>0$ and
$0\leq t\leq 1$. We proved the example.
\end{proof}

\begin{remark}
The uniformly symmetric condition is a weaker condition than the
$C^{1+\alpha}$ expanding condition for some $0<\alpha\leq 1$. For
example, a uniformly symmetric circle endomorphism could be
totally singular, that is, it could map a set with positive
Lebesgue measure to a set with zero Lebesgue measure. But we will
see in the rest of the paper, many dynamical aspects, from the
daul geometric point of view, of a $C^{1+\alpha}$ expanding circle
endomorphism for some $0<\alpha\leq 1$ will be preserved by a
uniformly symmetric circle endomorphism.
\end{remark}

\vspace*{10pt}
Another example of a uniformly symmetric circle
endomorphism is a $C^{1}$ Dini expanding circle endomorphism as
follows. Suppose $f$ is a $C^{1}$ circle endomorphism. The
function
$$
\omega(t) =\sup_{|x-y|\leq t} |F'(x)-F'(y)|, \quad t>0,
$$
is called the modulus of continuity of $F'$. Then $f$ is called
$C^{1}$ Dini if
$$
\int_{0}^{1}\frac{\omega(t)}{t} dt <\infty.
$$

Suppose $f$ is a $C^{1}$ Dini expanding circle endomorphism. Let
$C>0$ and $\lambda >1$ be two constants such that
$$
(F^{n})'(x) \geq C\lambda^{n}, \quad x\in {\mathbb R},\;\; n\geq
1.
$$
Define
$$
\tilde{\omega}(t) =\sum_{n=1}^{\infty} \omega
(C^{-1}\lambda^{-n}t).
$$
Then
$$
\tilde{\omega}(t) \leq \int_{0}^{\infty} \omega
(C^{-1}\lambda^{-x} t) dx = \frac{1}{\log \lambda}
\int_{0}^{C^{-1}\lambda^{-1} t} \frac{\omega(y)}{y} dy <\infty
$$
for all $0\leq t\leq 1$ and $\tilde{\omega}(t) \to 0$ as $t\to 0$.

\vspace*{10pt}
\begin{example}~\label{diniexample} A $C^{1}$ Dini circle expanding endomorphism
$f$ is uniformly symmetric. Furthermore, $\varepsilon(t)\leq D
\tilde{\omega}(t)$ for some constant $D>0$ and $0\leq t\leq 1$.
\end{example}

\begin{proof}
Since $F(x+1) = F(x)+d$, then $F'(x+1)=F'(x)$ is a periodic
function. Since $f$ is $C^{1}$ expanding, there are two constants
$C_{1}>0$ and $\lambda
>1$ such that
$$
(F^{n})' (x) \geq C_{1}\lambda^{n}, \quad \forall x\in {\mathbb
R},\;\; n>0.
$$

For any $x, y\in {\mathbb R}$ and $n>0$, let $x_{k}=F^{-k}(x)$ and
$y_{k}=F^{-k}(y)$, $0\leq k\leq n$. Then
$$
\Big| \log \frac{(F^{-n})'(x)}{(F^{-n})'(y)} \Big|= \Big| \log
\frac{(F^{n})'(y_{n})}{(F^{n})'(x_{n})}\Big| \leq \sum_{k=1}^{n} |
\log F'(x_{k})-\log F'(y_{k})|
$$
$$
\leq \frac{1}{C_{1}\lambda} \sum_{k=1}^{n} | F'(x_{k})-F'(y_{k})|
\leq \frac{1}{C_{1}\lambda} \sum_{k=1}^{n} \omega (C^{-1}
\lambda^{-k} |x-y|).
$$
Let $C=1/(C_{1}\lambda)$. Then we have the following Dini
distortion property:
\begin{equation}~\label{dinidistortion}
e^{-C\tilde{\omega}(|x-y|)}\leq
\frac{(F^{-n})'(x)}{(F^{-n})'(y)}\leq e^{C\tilde{\omega}(|x-y|)},
\quad \forall x, y\in {\mathbb R}, \;\; \forall n>0.
\end{equation}
Furthermore, let
$$
\varepsilon (t) = \left\{ \begin{array}{ll}
e^{C\tilde{\omega}(t)}-1, & 0<t\leq 1\\
e^{C\tilde{\omega}(1)}-1, & t>1.
\end{array}\right.
$$
Then $\varepsilon (t)>0$ is a bounded function such that
$\varepsilon (t) \to 0$ as $t\to 0^{+}$ and such that
$$
\frac{1}{1+\varepsilon (t)} \leq
\frac{(F^{-n})'(\xi)}{(F^{-n})'(\eta)} =
\frac{|F^{-n}(x+t)-F^{-n}(x)|}{|F^{-n}(x)-F^{-n}(x-t)|}\leq
1+\varepsilon (t)
$$
for all $x\in {\mathbb R}$, $t>0$, and $n>0$,
where $\xi$ and $\eta$ are two numbers in $[0,1]$. Thus $F$ is
uniformly symmetric. Furthermore, we have a constant $D>0$ such
that $\varepsilon(t)\leq D \tilde{\omega}(t)$ for all $0\leq t\leq
1$. We proved the example.
\end{proof}

A circle homeomorphism $h$ is called H\"older if there is a constant $0<\alpha\leq 1$ such that
\begin{equation}~\label{hconj}
\sup_{x\not= y} \frac{|H(x)-H(y)|}{|x-y|^{\alpha}} <\infty.
\end{equation}
We say that a circle endomorphism $f$ is H\"older conjugate to another circle endomorphism $g$
if there is a H\"older circle homeomorphism $h$ such that
$$
f\circ h=h\circ g.
$$
A circle endomorphism $g$ is called topologically expanding if there are constants $C>0$ and $\lambda >1$ such that
\begin{equation}~\label{topexp}
|G^{n} (x) -G^{n}(y)| \geq C\lambda^{n}|x-y|, \quad \forall x, y\in [0,1], \quad \forall n\geq 0.
\end{equation}
Following the proof of our result in~\cite{Jiang9} about Katok's conjecture, we have the following more general result.

\vspace*{10pt}
\begin{theorem}~\label{KatokConj} Suppose that $f$ is a $C^{1}$ circle endomorphism and suppose that $f$ is H\"older conjugate to a
topologically expanding circle endomorphism $g$. Then $f$ itself is expanding.
\end{theorem}

Combining this theorem and Example~\ref{diniexample}, we have that

\vspace*{10pt}
\begin{example}~\label{hcexample} A $C^{1}$ Dini circle endomorphism
$f$ which is H\"older conjugate a topologically expanding circle endomorphism $g$
is uniformly symmetric.
\end{example}

\begin{proof}[Proof of Theorem~\ref{KatokConj}]
The proof of the theorem is almost similar to the proof given in~\cite{Jiang9}.
However, for the sake of the completeness of this paper, we give a detailed proof.

Suppose the degree of $g$ is $d$. Since $f$ is topologically conjugate to $g$, its degree is also $d$.
The preimage $g^{-1}(1)$ contains $d$ points and cuts $T$ into
$d$ closed intervals $\varpi_{0,g}=\{ J_{0,g}, \cdots,
J_{d-1,g}\}$. Actually $\varpi_{0,g}$ is a Markov partition in the meaning
that
\begin{itemize}
\item[i)] $T=\cup_{k=0}^{d-1} J_{k,g}$,
\item[ii)] $J_{i,g}$ and $J_{j,g}$ have disjoint interiors for $0\leq i\neq
j\leq d-1$,
\item[iii)] the restriction of $g$ on the interior of $J_{i,g}$ is one to one for every $0\leq i\leq
d-1$,
\item[iv)] $g(J_{i,g})=T$ for every $0\leq i\leq d-1$.
\end{itemize}
Thus we can generate a sequence of Markov partitions
$$
\varpi_{n,g}=g^{-n} \varpi_{0}
$$
for $n=1, 2, \cdots$. The set $\varpi_{n,g}$ contains all intervals
$J$ such that $g^{n}: J\to J_{k,g}$ for some $1\leq k\leq d$ is a
homeomorphism.

From~(\ref{topexp}), we have constants $C_{0}>0$ and $0\leq \tau_{0}<1$
such that
$$
\max_{J\in \varpi_{n,g}} |J| \leq C_{0}\tau^{n}_{0}, \quad \forall n\geq 0.
$$

Since $f$ is H\"older conjugate to $g$, we have a
homeomorphism $h$ satisfying~(\ref{hconj}) such that
$f\circ h=h\circ g$.
Let
$$
\varpi_{n,f} =\{ h(J)\;|\; J\in \varpi_{n,g}\}.
$$
Then we have a constant $C_{1}>0$ and $\tau_{1}=\tau_{0}^{\alpha}$ such that
$$
|J| \leq C_{1}\tau_{1}^{n},\quad \forall J\in \varpi_{n,f}, \quad \forall n\geq 0.
$$

We use $I$ to denote the lift interval of $J$ in the unit interval $[0,1]$.
Given any interval $J\in \varpi_{n,f}$, $F^{n}(I) =[m,m+1]$ for some integer $m\geq 0$.
For any $x,y\in I$,
$$
{\mathcal A}_{n} (x,y)=log \frac{(F^{n})'(x)}{(F^{n})'(y)} =
\sum_{i=0}^{n-1} \Big( \log F'(F^{i}(x)) -\log F'(F^{i}(y))\Big).
$$

Let
$$
a_{n}=\max_{J\in \varpi_{n,f}} \Big\{ \max_{x\in I} \log F'(x) -\min_{x\in I} \log F'(x)\Big\}
$$
and
$$
D_{n}=\sum_{i=1}^{n} a_{k} \quad \hbox{and} \quad E_{n} = e^{-D_{n}}.
$$
Then
$$
|{\mathcal A}_{n}(x,y)| \leq D_{n}
$$
since $F'$ is a periodic function of period $1$.

Since $\log F'$ is uniformly continuous on $[0,1]$, we have that $a_{n}\to 0$ as $n\to \infty$. This implies that
$$
\frac{D_{n}}{n} \to 0\quad \hbox{as}\quad n\to \infty
$$
and
$$
\sqrt[n]{E_{n}}= e^{-\frac{D_{n}}{n}} \to 1 \quad \hbox{as} \quad n\to \infty.
$$

Since $F^{n}(I) =[m,m+1]$, by the mean value theorem, we have a point $y_{n}\in I$ such that
$$
(F^{n})' (y_{n}) =1/|J| \geq C_{1}^{-1} \tau_{1}^{-n}, \quad \forall n\geq 0.
$$
This implies that
$$
(F^{n})' (x) \geq E_{n} (F^{n})' (y_{n}) \geq E_{n} C_{1}^{-1} \tau_{1}^{-n}= C_{1}^{-1} \Big( \sqrt[n]{E_{n}} \tau_{1}^{-1}\Big)^{n}, \quad \forall n\geq 0.
$$
Thus we have constants $C>0$ and $\lambda>1$ such that
$$
(F^{n})' (x) \geq C \lambda^{n}, \quad \forall n\geq 0.
$$
That is, $f$ is expanding. We proved the theorem.
\end{proof}

\section{Symbolic space and topological representation}

Suppose $f$ is a circle endomorphism of degree $d\geq 2$ with $f(1)=1$. Consider the
preimage $f^{-1}(1)$. As we have seen in the proof of Theorem~\ref{KatokConj}, $f^{-1}(1)$ cuts $T$ into $d$ closed
intervals $J_{0}$, $J_{1}$, $\cdots$, $J_{d-1}$, ordered by the
counter-clockwise order of $T$. Suppose $J_{0}$ has an endpoint
$1$. Then $J_{d-1}$ also has an endpoint $1$. Let
$$
\varpi_{0}=\{ J_{0}, J_{1}, \cdots, J_{d-1}\}.
$$
Then it is a Markov partition, that is,
\begin{itemize}
\item[{\rm i.}] $T=\cup_{k=0}^{d-1} J_{k}$,
\item[{\rm ii.}] the restriction of $f$ to the interior of $J_{i}$ is
injective for every $0\leq i\leq d-1$,
\item[{\rm iii.}] $f(J_{i})=T$ for every $0\leq i\leq d-1$.
\end{itemize}

Let $I_{0}$, $I_{1}$, $\cdots$, $I_{d-1}$ be the lifts of $J_{0}$,
$J_{1}$, $\cdots$, $J_{d-1}$ in $[0,1]$. Then we have that
\begin{itemize}
\item[{\rm i)}] $[0,1]=\cup_{k=0}^{d-1} I_{k}$,
\item[{\rm ii)}] $F(I_{i})=[i, i+1]$ for every $0\leq i\leq d-1$.
\end{itemize}
Let
$$
\eta_{0}=\{I_{0}, I_{1}, \cdots, I_{d-1}\}.
$$
Then it is a partition of $[0,1]$.

Consider the pull-back partition $\varpi_{n}=f^{-n}\varpi_{0}$ of
$\varpi_{0}$ by $f^{n}$. It contains $(d-1)^{n}$ intervals and is
also a Markov partition of $T$. Intervals $J$ in $\varpi_{n}$ can
be labeled as follows. Let $w_{n}=i_{0}i_{1}\cdots i_{n-1}$ be a
word of length $n$ of $0's$, $1's$, $\cdots$, and $(d-1)'s$. Then
$J_{w_{n}}\in \varpi_{n}$ if $f^{k}(J_{w_{n}}) \subset J_{i_{k}}$
for $0\leq k\leq n-1$. Then
$$
\varpi_{n} =\{ J_{w_{n}}\;|\; w_{n}=i_{0}i_{1}\cdots i_{n-1},
i_{k}\in\{ 0, 1, \cdots, d-1\}, k=0, 1,\cdots, d-1\}.
$$
Let $\eta_{n}$ be the corresponding lift partition of $\varpi_{n}$
in $[0,1]$ with the same labelings. Then
$$
\eta_{n} =\{ I_{w_{n}}\;|\; w_{n}=i_{0}i_{1}\cdots i_{n-1},
i_{k}\in\{ 0, 1, \cdots, d-1\}, k=0, 1,\cdots, d-1\}.
$$

Consider the space
$$
\Sigma=\prod_{n=0}^{\infty} \{ 0, 1, \cdots, d-1\}
$$
$$
=\{ w=i_{0}i_{1}\cdots i_{k} \cdots i_{n-1} \cdots \; |\; i_{k}\in
\{ 0, 1, \cdots, d-1\}, \; k=0, 1, \cdots \}
$$
with the product topology. It is a compact topological space. A
left cylinder for a fixed word $w_{n}=i_{0}i_{1}\cdots i_{n-1}$ of
length $n$ is
$$
[w_{n}] =\{ w' =i_{0}i_{1}\cdots i_{n-1}i_{n}'i_{n+1}' \cdots
\;|\; i_{n+k}' \in \{ 0, 1, \cdots, d-1\}, k=0, 1, \cdots\}
$$
All left cylinders form a topological basis of $\Sigma$. We call
it the {\em left topology}. The space $\Sigma$ with this left
topology is called the {\em symbolic space}.

For any $w=i_{0}i_{1}\cdots i_{n-1}i_{n} \cdots$, let
$$
\sigma (w) = i_{1}\cdots i_{n-1}i_{n} \cdots
$$
be the left shift map. Then $(\Sigma, \sigma)$ is called a
symbolic dynamical system.

For a point $w=i_{0}\cdots i_{n-1}i_{n} \cdots \in \Sigma$, let
$w_{n} =i_{0}\cdots i_{n-1}$. Then
$$
\cdots \subset J_{w_{n}} \subset J_{w_{n-1}} \subset \cdots
J_{w_{1}}\subset T.
$$
Since each $J_{w_{n}}$ is compact,
$$
J_{w}=\cap_{n=1}^{\infty} J_{w_{n}} \neq \emptyset.
$$
If every $J_{w}=\{ x_{w}\}$ contains only one point, then we
define the projection $\pi_{f}$ from $\Sigma$ onto $T$ as
$$
\pi_{f} (w) =x_{w}.
$$
The projection $\pi_{f}$ is 1-1 except for a countable set
$$
B=\{ w=i_{0}i_{1}\cdots i_{n-1}1000\cdots,  i_{0}i_{1}\cdots
i_{n-1}0(d-1)(d-1)(d-1)\cdots \}.
$$
From our construction, one can check that
$$
\pi_{f}\circ \sigma(w) = f \circ \pi_{f} (w), \quad w\in \Sigma.
$$

For any interval $I=[a,b]$ in $[0, 1]$, we use $|I|=b-a$ to mean
its Lebesgue length. Let
$$
\iota_{n,f}=\max_{w_{n}} |I_{w_{n}}|,
$$
where $w_{n}$ runs over all words of $\{ 0, 1, \cdots, d-1\}$ of
length $n$.

Two circle endomorphisms $f$ and $g$ are topologically conjugate
if there is an orientation-preserving circle homeomorphism $h$ of
$T$ such that
$$
f\circ h=h\circ g.
$$
The following result is first proved by Shub in~\cite{Shub} for
$C^{2}$ expanding circle endomorphisms by using the contracting
mapping theorem.

\vspace*{10pt}
\begin{theorem}~\label{topologicalconjugacy}
Let $f$ and $g$ be two circle endomorphisms such that both
$\iota_{n,f}$ and $\iota_{n,g}$ tend to zero as $n\to \infty$.
Then $f$ and $g$ are topologically conjugate if and only if their
topological degrees are the same.
\end{theorem}

\begin{proof}
The topological conjugacy preserves the topological degree. Thus if
$f$ and $g$ are topologically conjugate, then their topological
degrees are the same.

Now suppose $f$ and $g$ have the same topological degree. Then
they have the same symbolic space. Since both sets
$J_{w,f}=\{x_{w}\}$ and $J_{w,g}=\{y_{w}\}$ contain only a single
point for each $w$, we can define
$$
h (x_{w}) =y_{w}.
$$
One can check that $h$ is an orientation-preserving homeomorphism
with the inverse
$$
h^{-1}(y_{w})=x_{w}.
$$
\end{proof}

Therefore, for a fixed degree $d\geq 2$, there is only one topological
model $(\Sigma, \sigma)$  for dynamics of all circle endomorphisms
of degree $d$ with $\iota_{n}\to 0$ as $n\to \infty$.

\vspace*{10pt}
\begin{definition}~\label{boundedgeometrydef}
The sequence $\{\varpi_{n}\}_{n=0}^{\infty}$ of nested partitions
of $T$ is said to have bounded nearby geometry if there is a
constant $C>0$ such that for any $n\geq 0$ and any two intervals
$I, I'\in \eta_{n}$ with a same endpoint or one has an endpoint
$0$ and the other has an endpoint $1$ (in which case we say they
have a common endpoint by modulo $1$),
$$
C^{-1}\leq \frac{|I'|}{|I|} \leq C.
$$
The sequence $\{\varpi_{n}\}_{n=0}^{\infty}$ of nested partitions
of $T$ is said to have bounded geometry if there is a constant
$C>0$ such that
$$
\frac{|L|}{|I|} \geq C, \quad \forall \; L\subset I,\;\; L \in
\eta_{n+1}, \; I\in \eta_{n}, \quad \forall \; n\geq 0.
$$
\end{definition}

The bounded nearby geometry implies the bounded geometry since
each interval $I\in \eta_{n}$ is divided into $d$ subintervals in
$\eta_{n+1}$. But it is not true for the other direction.

\vspace*{10pt}
\begin{theorem}~\label{boundedgeometrythm}
Suppose $f$ is a uniformly symmetric circle endomorphism. Then the
sequence $\{\varpi_{n}\}_{n=0}^{\infty}$ of nested partitions of
$T$ has bounded nearby geometry and thus bounded geometry.
\end{theorem}

\vspace*{10pt}
\begin{proof}
Let $F$ with $F(0)=0$ be the lift of $f$. Define
$$
G_{k}(x) = F^{-1} (x+k): [0,1]\to [0,1], \quad \hbox{for}\quad
k=0, 1, \cdots, n-1.
$$
For any word $w_{n}=i_{0}i_{1}\cdots i_{n-1}$, define
$$
G_{w_{n}} = G_{i_{0}}\circ G_{i_{1}}\circ \cdots \circ
G_{i_{n-1}}.
$$
Then
$$
I_{w_{n}} =G_{w_{n}} ([0,1])= F^{-n} ([m, m+1]),
$$
where $m=i_{n-1}+i_{n-2}d+\cdots + i_{0}d^{n-1}$. Suppose
$I_{w_{n}}'$ is an interval in $\eta_{n}$ having a common endpoint
with $I_{w_{n}}$ modulo $1$. Then
$$
I_{w_{n}}' = F^{-n} ([m+1, m+2]) \quad \hbox{or}\quad F^{-n}
([m-1, m]).
$$
Thus
$$
\frac{1}{1+\varepsilon(1)}\leq  \frac{|I_{w_{n}}|}{|I_{w_{n}}'|}
\leq 1+\varepsilon (1).
$$
Let $C =1+\varepsilon(1)$. Then we have that
$$
C^{-1}\leq \frac{|I|}{|I'|} \leq C
$$
for any intervals $I, I' \in \eta_{n}$ with a common endpoint
modulo $1$, $n=0, 1, \cdots$. This means that
$\{\varpi_{n}\}_{n=0}^{\infty}$ has the bounded nearby geometry.
We proved the theorem.
\end{proof}

\vspace*{10pt}
\begin{corollary}~\label{quasisymmetriccojugacy}
Any two uniformly symmetric circle endomorphisms $f$ and $g$ of
the same degree $d\geq 2$ are topologically conjugate and the
conjugacy is a quasisymmetric homeomorphism.
\end{corollary}

\begin{proof}
From $f\circ h=h\circ g$ and $g(1)=1$, $h(1)$ is a fixed point of
$f$, that is, $f(h(1)) =h(1)$. Let $k(z) =z/h(1)$ and $\tilde{f} =
k\circ f\circ k^{-1}$. Then $\tilde{f}(1) =1$. Take $\tilde{h} =
k\circ h$. We have that $\tilde{h}(1)=1$ and $\tilde{f}\circ
\tilde{h}= \tilde{h}\circ g$. So $\tilde{h}$ is quasisymmetric if
and only if $h$ is quasisymmetric. So, without loss of generality,
we assume that $h(1)=1$.

Suppose
$$
\eta_{n,f} =\{I_{w_{n},f} \} \quad \hbox{and}\quad \eta_{n,g}=\{
I_{w_{n}, g}¡¢\}, \quad n=1, 2,\cdots
$$
are two sequences of Markov partitions for $f$ and $g$,
respectively.

From the bounded geometry property
(Theorem~\ref{boundedgeometrythm}), we have a constant $0<\tau<1$
such that
$$
\iota_{n,f} =\max_{w_{n}} |I_{w_{n},f}|, \;\; \iota_{n,g}
=\max_{w_{n}} |I_{w_{n},g}|\leq \tau^{n}, \quad \forall \; n=1,
2,\cdots.
$$
Then Theorem~\ref{topologicalconjugacy} implies that $f$ and $g$
are topologically conjugate.

Suppose $h$ is the topological conjugacy between $f$ and $g$ and
$H$ is its lift to ${\mathbb R}$. By adding all integers, the
sequence of partitions $\eta_{n,f}$ and $\eta_{n,g}$ induce two
sequences of partitions of ${\mathbb R}$, which we still denoted
as $\eta_{n,f}$ and $\eta_{n,g}$. Both of these sequences of
partitions have bounded nearby geometry.

Let $\Omega$ be the set of all endpoints of intervals $I\in
\eta_{n}$, $n=0,1\cdots,\infty$. Then it is dense in ${\mathbb
R}$.

For $x\in \Omega$. Consider the interval $[x-t, x]$. There is a
largest integer $n\geq 0$ such that there is an interval
$I=[a,x]\in \eta_{n,f}$ satisfying $[x-t,x]\subseteq I$. Suppose
$J=[b,x]\in \eta_{n+1,f}$. Then $J\subseteq [x-t,x]$. Let
$J'=[x,c]\in \eta_{n+1,f}$. From Theorem~\ref{boundedgeometrythm},
there is a constant $C>0$ such that
$$
C^{-1}\leq \frac{|J'|}{|J|}\leq C.
$$

If $|J'|> t$, we have $|J'|\leq Ct$. Let $J_{k}'=[x,c_{k}]\in
\eta_{n+k+1,f}$ for $k>0$. From the bounded geometry, there is a
$0<\tau<1$ such that
$$
|J_{k}'| \leq \tau^{k} C t.
$$
Let $k$ be the smallest integer
greater than $-\log C/\log \tau$. Then $|J_{k}'|\leq t$. This implies
that $J_{k}'\subseteq [x,x+t]$. So we have
$$
\frac{|H(J_{k}')|}{|H(I)|} \leq
\frac{|H(x+t)-H(x)|}{|H(x)-H(x-t)|} \leq \frac{|H(J')|}{|H(J)|},
$$
where $H(I)\in \eta_{n,g}$, $H(J), H(J')\in \eta_{n+1,g}$, and
$H(J_{k}')\in \eta_{n+k+1, g}$. Now from the bounded geometry for
$g$, we have a constant, still denote as $C>0$, such that
$$
C^{-1} \leq \frac{|H(J_{k}')|}{|H(I)|} \leq
\frac{|H(x+t)-H(x)|}{|H(x)-H(x-t)|} \leq \frac{|H(J')|}{|H(J)|}
\leq C.
$$

If $|J'|\leq t$, we have $|J'|\geq C^{-1} t$. Let
$J_{-k}'=[x,c_{-k}]\in \eta_{n-k+1,f}$ for $k\geq 0$. Then from
the bounded geometry, there is a constant, which we still denote
as $0<\tau<1$, such that $|J_{-k}'| \geq \tau^{-k} C^{-1} t$. Let
$k$ be the smallest integer greater than $-\log C/\log \tau$. Then
$|J_{-k}'|\geq t$. This implies that $J_{-k}'\supseteq [x,x+t]$.
So we have
$$
\frac{|H(J')|}{|H(I)|} \leq \frac{|H(x+t)-H(x)|}{|H(x)-H(x-t)|}
\leq \frac{|H(J_{-k}')|}{|H(J)|},
$$
where $H(I)\in \eta_{n,g}$, $H(J), H(J')\in \eta_{n+1,g}$, and
$H(J_{-k}')\in \eta_{n-k+1, g}$. Now from the bounded geometry for
$g$, we have a constant, which we still denote as $C>0$, such that
$$
C^{-1} \leq \frac{|H(J')|}{|H(I)|} \leq
\frac{|H(x+t)-H(x)|}{|H(x)-H(x-t)|} \leq
\frac{|H(J_{-k}')|}{|H(J)|} \leq C.
$$

For any $x\in {\mathbb R}$, since $\Omega$ is dense in $[0,1]$, we
have a sequence $x_{n}\in \Omega$ such that $x_{n}\to x$ as $n\to
\infty$. For any $t>0$, we have that
$$
C^{-1} \leq \frac{|H(x_{n}+t)-H(x_{n})|}{|H(x_{n})-H(x_{n}-t)|}
\leq C.
$$
Since $H$ is uniformly continuous on ${\mathbb R}$, we get that
$$
C^{-1} \leq \frac{|H(x+t)-H(x)|}{|H(x)-H(x-t)|} \leq C.
$$
We proved the theorem.
\end{proof}

\begin{remark} The bounded nearby geometry and the quasisymmetric
property for a conjugacy have been also studied for
one-dimensional maps with critical points
in~\cite{Jiang1,Jiang2,Jiang6}.
\end{remark}

\section{Dual symbolic space and geometric representation}

Suppose $f$ is a circle endomorphism. Suppose
$\{\eta_{n}\}_{n=0}^{\infty}$ is the sequence of partitions of
$[0,1]$. As we have seen in the previous section, for each
interval in $\eta_{n}$, there is a labeling
$w_{n}=i_{0}i_{1}\cdots i_{n-1}$. One can think of this kind of
labelings as the left topology: read ordered digits from the left
to the right. Now we read from the same ordered digits from the
right to the left, that is,
$$
w_{n}^{*} =j_{n-1}\cdots j_{1}j_{0}
$$
where $j_{n-1}=i_{0}$, $\cdots$, $j_{1}=i_{n-2}$, and
$j_{0}=i_{n-1}$. Thus we consider the dual symbolic space
$$
\Sigma^{*} =\{w^{*} =\cdots j_{n-1}\cdots j_{k}\cdots j_{1}j_{0}
\; |\; j_{k}\in \{ 0, 1, \cdots, d-1\}, \; k=0, 1, \cdots \}
$$
equipped with the right topology which is generated by all right
cylinders
$$
[w_{n}^{*}]=\{ w^{*} =\cdots j_{n}'j_{n-1}\cdots j_{1}j_{0}\; |\;
j_{n+k}' \in \{ 0, 1, \cdots, d-1\}, k=0, 1, \cdots\},
$$
where $w_{n}^{*}= j_{n-1}\cdots j_{1}j_{0}$ is a fixed word of $\{
0, 1, \cdots, d-1\}$ of length $n$.

Consider the right shift map
$$
\sigma^{*} : w^{*}= \cdots j_{n-1}\cdots j_{1}j_{0} \to \sigma^{*}
(w^{*}) =\cdots j_{n-1}\cdots j_{1}.
$$
Then we call $(\Sigma^{*}, \sigma^{*})$ the dual symbolic
dynamical system for $f$.

The dual derivative of $f$ is defined on the dual symbolic space
$\Sigma^{*}$ as follows.

For any $w^{*}=\cdots j_{n-1}\cdots j_{1}j_{0}\in \Sigma^{*}$, let
$$
w_{n}^{*} =j_{n-1}\cdots j_{1}j_{0}\quad \hbox{and}\quad
v_{n-1}^{*}= \sigma^{*} (w_{n}^{*}) = j_{n-1}\cdots j_{1}.
$$
Then
$$
I_{w_{n}} \subset I_{v_{n-1}}.
$$
Define
\begin{equation}~\label{prederivative}
D^{*}(f) (w_{n}^{*}) =\frac{|I_{v_{n-1}}|}{|I_{w_{n}}|}.
\end{equation}

\begin{definition}
If for every $w^{*} \in \Sigma^{*}$,
$$
D^{*}(f) (w^{*}) = \lim_{n\to \infty} D^{*}(f) (w_{n}^{*})
$$
exists, then we define a function
\begin{equation}~\label{derivative}
D^{*} (f)(w^{*}): \Sigma^{*}\to {\mathbb R}^{+}.
\end{equation}
We call this function the dual derivative of $f$.
\end{definition}

\begin{remark}
We used to call one divided by a dual derivative a scaling
function. The notion of the scaling function is first introduced
into the study of geometric Cantor sets on the line by Sullivan
in~\cite{Sullivandiff} where a scaling function is used to define
differentiable structures of geometric Cantor sets on the line. A
general version of scaling functions for any Markov maps is
defined in~\cite{Jiang1} (see also~\cite{Jiang2}). This general
notion of scaling function has been used extensively
in~\cite{Jiang3,Jiang4,Jiang5} as a complete smooth invariant in
the smooth classification of one-dimensional maps with critical
points. Since a circle endomorphism of degree $d\geq 2$ can be thought
as a Markov map, we use the the definition in~\cite{Jiang1} (see
also~\cite{Jiang2}). However, the notion of the dual derivative is
a more appropriate term in this paper for the study of dual
geometric Gibbs measure theory.
\end{remark}

A function $\phi^{*} (w^{*})$ on $\Sigma^{*}$ is called H\"older
continuous if there are constants $C>0$ and $0<\tau< 1$ such that
$$
|\phi^{*}(w^{*})-\phi^{*}(\tilde{w}^{*})| \leq C \tau^{n}
$$
as long as the first $n$ digits of $w^{*}$ and $\tilde{w}^{*}$
from the right are the same. If we consider a metric
$$
d(w^{*}, \tilde{w}^{*}) =\sum_{k=0}^{\infty}
\frac{|j_{k}-j_{k}'|}{d^{k}}
$$
on $\Sigma^{*}$, then $\phi^{*}(w^{*})$ being H\"older continuous
is equivalent to the condition that there are two constants $C>0$
and $0<\beta \leq 1$ such that
$$
|\phi^{*}(w^{*})-\phi^{*}(\tilde{w}^{*})| \leq C\big(
d(w^{*},\tilde{w}^{*})\big)^{\beta}, \quad w, \tilde{w}^{*}\in
\Sigma^{*}.
$$

\vspace*{10pt}
\begin{theorem}~\label{dualderivative}
Suppose $f$ is a uniformly symmetric circle
endomorphism. Then its dual derivative
$$
D^{*}(f)(w^{*}): \Sigma^{*}\to {\mathbb R}^{+}
$$
exists and is a continuous function. Furthermore, if $f$ is
$C^{1+\alpha}$, then $D^{*}(f)(w^{*})$ is H\"older continuous.
Actually when $f$ is $C^{1}$ Dini expanding, the modulus of
continuity of $D^{*}(f)(w^{*})$ is controlled by
$\tilde{\omega}(t)$.
\end{theorem}

We first prove the following lemma. Suppose $Q: [0,1] \to [0,1]$
is a homeomorphism such that $Q(0)=0$ and $Q(1)=1$. Let $M\geq 1$.
We say that $Q$ is $M$-quasisymmetric on $[0,1]$ if
$$
M^{-1} \leq \frac{|Q(x+t)-Q(x)|}{|Q(x)-Q(x-t)|} \leq M, \quad
\forall\; x-t,\; x,\; x+t \in [0,1], t>0.
$$

\vspace*{10pt}
\begin{lemma}~\label{qsdistortion}
There is a function $\zeta (M)>0$ satisfying $\zeta (M) \to 0$ as
$M\to 1$ such that for any $M$-quasisymmetric homeomorphism $Q$ on
$[0,1]$ such that $Q(0)=0$ and $Q(1)=1$,
$$
|Q(x)-x|\leq \zeta (M), \quad \forall\; x \in [0,1].
$$
\end{lemma}

\begin{proof}
Consider points $x_{n}=1/2^{n}$, $n=0, 1, \cdots$.
$M$-quasisymmetry  and the normalization  $Q(0)=0, Q(1)=1$ imply
that
$$
\frac{1}{1+M} H(\frac{1}{2^{n-1}}) \leq Q(\frac{1}{2^{n}}) \leq
\frac{1}{1+M^{-1}} Q(\frac{1}{2^{n-1}}).
$$
Similarly,
$$
\Big(\frac{1}{1+M}\Big)^{n} \leq Q(\frac{1}{2^{n}}) \leq
\Big(\frac{1}{1+M^{-1}}\Big)^{n}, \quad \forall\; n\geq 1.
$$
Furthermore, by $M$-quasisymmetry  and induction on $n=1,
2,\cdots$, yield
$$
\Big(\frac{1}{1+M}\Big)^{n} \leq Q(\frac{i}{2^{n}}) -
Q(\frac{i-1}{2^{n}}) \leq \Big( \frac{1}{1+M^{-1}}\Big)^{n}, \quad
\forall \; n\geq 1, \;\; 1\leq i\leq 2^{n}.
$$

Let
$$
\tau_n=
\max\left\{\left(\frac{M}{M+1}\right)^n-\frac{1}{2^n},\frac{1}{2^n}-\left(\frac{1}{M+1}\right)^n\right\},
\quad n=1,2, \cdots.
$$
Then for $n=1$,
$$
|Q(\frac{1}{2}) -\frac{1}{2}| \leq
\tau_{1}=\frac{1}{2}\frac{M-1}{M+1},
$$
and for any $n>1$, we have
$$
\max_{0\leq i\leq 2^{n}} \Big| Q(\frac{i}{2^{n}})
-\frac{i}{2^{n}}\Big| \leq \max_{0\leq i\leq 2^{n-1}} \Big|
Q(\frac{i}{2^{n-1}}) -\frac{i}{2^{n-1}}\Big| + \tau_{n}
$$
By summing over $k$ for $1 \leq k \leq n,$ we obtain
$$
\max_{0\leq i\leq 2^{n}} \Big| Q(\frac{i}{2^{n}})
-\frac{i}{2^{n}}\Big| \leq \delta_{n}=\sum_{k=1}^{n} \tau_{k}.
$$
If we put $\zeta (M) = \sup_{1\leq n<\infty}\{\delta_{n}\},$ by
summing geometric series, we obtain
$$
\zeta(M) = \max_{1\leq n<\infty} \Big\{
M-1+\frac{1}{2^{n}}-M\Big(\frac{M}{1+M}\Big)^{n},
1-\frac{1}{M}+\frac{1}{M}\Big(\frac{1}{M}\Big)^{n}
-\frac{1}{2^{n}}\Big\}.
$$
Clearly,  $\zeta(M)\to 0$ as $M\to 1$, and since the dyadic points
$$
\{ i/2^{n}\;\; |\;\; n=1, 2, \cdots ; 0\leq i\leq 2^{n}\}
$$
are dense in $[0,1]$, we conclude
$$
|Q(x)-x| \leq \zeta (M) \quad \forall \; x\in [0,1],
$$
which proves the lemma.
\end{proof}

\begin{proof}[Proof of Theorem~\ref{dualderivative}] Suppose
$w^*=\cdots j_{n-1}\cdots j_{1}j_{0}\in \Sigma^{*}$. Let
$$
w^{*}_{n} =j_{n-1}\cdots j_{1}j_{0}\quad \hbox{and}\quad
v_{n-1}^{*} = j_{n-1}\cdots j_{1}.
$$
By definition,
$$
D^{*} (f) (w_{n}^{*}) = \frac{|I_{v_{n-1}}|}{|I_{w_{n}}|},
$$
where $I_{w_{n}} \subset I_{v_{n-1}}$. Consider the sequence
$\{D^{*} (f) (w_{n}^{*})\}_{n=1}^{\infty}$.

Let $0<\tau<1$ be a constant such that
$$
\tilde{\iota}_{n} =\max_{w_{n}} |I_{w_{n}}| \leq \tau^{n}, \quad
\forall n\geq 1.
$$
For any $\epsilon>0$, let $n_{0}>0$ be an integer such that $\zeta
(1+\varepsilon(\tau^{n-1})) \leq \epsilon$ for all $n>n_{0}$. Then
for any $m>n> n_{0}$, we have that
$$
F^{m-n} (I_{v_{m-1}}) = I_{v_{n-1}}\quad \hbox{and}\quad F^{m-n}
(I_{w_{m}}) = I_{w_{n}}
$$
Since $F^{-(m-n)}| I_{v_{n-1}}$ is a $(1+\varepsilon
(\tau^{n-1}))$-quasisymmetric homeomorphism, from
Lemma~\ref{qsdistortion} (by normalizing $I_{v_{n-1}}$ to $[0,1]$
and $I_{w_{m}}$ to $[0,x]$ by a linear transformation),
$$
|D^{*} (f) (w_{m}^{*})- D^{*} (f) (w_{n}^{*})| =
\Big|\frac{|F^{-(m-n)}(I_{v_{n-1}})|}{|F^{-(m-n)}(I_{w_{n}})|} -
\frac{|I_{v_{n-1}}|}{|I_{w_{n}}|}\Big| \leq \zeta(1+\varepsilon
(\tau^{n-1})) \leq \epsilon.
$$
This implies that $\{D^{*} (f) (w_{n}^{*})\}_{n=1}^{\infty}$ is a
Cauchy sequence. Thus the limit
$$
D^{*} (f) (w^{*})= \lim_{n\to \infty} D^{*} (f) (w_{n}^{*})
$$
exists.

Now consider two points
$$
w^{*} = \cdots j_{m-1}\cdots j_{n}j_{n-1}\cdots j_{0}\quad
\hbox{and}\quad  \tilde{w}^{*} = \cdots j_{m-1}\cdots j_{n}'
j_{n-1}\cdots j_{0}.
$$
Let $w_{m}^{*}=j_{m-1}\cdots j_{n}j_{n-1}\cdots j_{0}$ and
$\tilde{w}_{m}^{*}=j_{m-1}\cdots j_{n}' j_{n-1}\cdots j_{0}$. Then
$w_{n}^{*}=\tilde{w}^{*}_{n}$. For any $m>n$,
$$
|D^{*} (f) (w_{m}^{*})- D^{*} (f) (\tilde{w}_{m}^{*})|
$$
$$
\leq |D^{*} (f) (w_{m}^{*})- D^{*} (f) (w_{n}^{*})| + |D^{*} (f)
(\tilde{w}_{m}^{*})- D^{*} (f) (w_{n}^{*})| \leq
2\zeta(1+\varepsilon (\tau^{n-1})).
$$
So by taking a limit,
$$
|D^{*} (f) (w^{*})- D^{*} (f) (\tilde{w}^{*})| \leq
2\zeta(1+\varepsilon (\tau^{n-1})).
$$
Thus we have that
$$
D^{*}(f)(w^{*}): \Sigma^{*}\to {\mathbb R}^{+}
$$
is a continuous function whose modulus of continuity is bounded by
$2\zeta(1+\varepsilon (\tau^{n-1}))$.

Moreover, if $f$ is a $C^{1+\alpha}$ expanding circle endomorphism
for some $0<\alpha\leq 1$, from the H\"older distortion
property~(\ref{smoothdistortion}), there is a constant $C>0$ such
that
$$
|D^{*} (f) (w^{*})- D^{*} (f) (\tilde{w}^{*})| \leq C\tau^{\alpha
(n-1)}.
$$
This implies that the dual derivative $D^{*}(f)(w^{*})$ is
H\"older continuous.

When $f$ is $C^{1}$ Dini, then there is a constant $C>0$ such that
$$
|D^{*} (f) (w^{*})- D^{*} (f) (\tilde{w}^{*})| \leq
C\tilde{\omega} (\tau^{n-1}).
$$
Thus the dual derivative $D^{*}(f)(w^{*})$ is continuous and its
modulus of continuity is controlled by $\tilde{\omega}
(\tau^{n-1})$. We proved the theorem.
\end{proof}

\section{Teichm\"uller spaces and dual derivatives}

For a fixed integer $d\geq 2$, let ${\mathcal C}^{1+}$ be the
space of all $C^{1+\alpha}$, $0<\alpha\leq 1$, expanding circle
endomorphisms of degree $d$. Take $q_{d} (z) =z^{d}$ as a
basepoint in ${\mathcal C}^{1+}$. A {\em marked $C^{1+}$ circle
endomorphism} by $q_{d}$ is a pair $(f, h_{f})$ where $f\in
{\mathcal C}^{1+}$ and $h_{f}$ is the orientation-preserving
homeomorphism of $T$ such that $h_{f}(1)=1$ and
$$
f\circ h_f= h_f\circ q_{d}.
$$

From Corollary~\ref{quasisymmetriccojugacy}, for any marked
$C^{1+}$ circle endomorphism $(f,h_{f})$ by $q_{d}$, $h_{f}$ is
quasisymmetric. Thus we can define Teichm\"uller equivalence
relation $\sim_{T}$, Teichm\"uller space, and Teichm\'uller type
metric as follows.

\begin{definition}
Two marked $C^{1+}$ circle endomorphisms are equivalent, denoted
as $(f,h_{f}) \sim_{T} (g,h_{g})$, if $h_{f}\circ h_{g}^{-1}$ is a
$C^{1}$-diffeomorphism.
\end{definition}

\begin{definition}
The Teichm\"uller space
$$
{\mathcal T}{\mathcal C}^{1+}=\{ [(f, h_f)] \;|\; f\in {\mathcal
C}^{1+}, \; \hbox{with the basepoint $[(q_{d}, id)]$} \}
$$
is defined as the space of all $\sim_{T}$-equivalence classes
$[(f,h_{f})]$ in the space of all marked $C^{1+}$ circle
endomorphisms by $q_{d}$.
\end{definition}

Now let us define the Teichm\"uller type metric $d_{\mathcal
T}(\cdot, \cdot)$ on ${\mathcal T}{\mathcal C}^{1+}$. We first
consider the universal Teichm\"uller space. We refer
to~\cite{Ahlfors,GardinerSullivanSymm,Lehto} as standard
references for this subject. Let $\mathcal{QS}$ be the set of all
quasisymmetric orientation-preserving homeomorphisms of the unit
circle $T$ factored by the space of all M\"obius transformations
of the circle. (Then $\mathcal{QS}$ may be identified with the set
of all quasisymmetric orientation-preserving homeomorphisms of the
unit circle fixing three points). Let ${\mathcal S}$ be the subset
of $\mathcal{QS}$ consisting of all symmetric
orientation-preserving homeomorphisms of the unit circle $T$. The
space ${\mathcal S}$ is a subgroup of $\mathcal{QS}$ closed in the
Teichm\"uller topology. For any $h\in \mathcal{QS}$, let
${\mathcal E}_{h}$ be the set of all quasiconformal extensions of
$h$ into the unit disk. For each $\tilde{h}\in {\mathcal E}_{h}$,
let
$$
\mu_{\tilde{h}} =\frac{\tilde{h}_{\overline{z}}}{\tilde{h}_{z}}
$$
be its complex dilatation. Let
$$
k_{\tilde{h}}=\|\mu(z)\|_{\infty}\quad \hbox{and}\quad
K_{\tilde{h}} =\frac{1+k_{\tilde{h}}}{1-k_{\tilde{h}}}.
$$
Here $K_{\tilde{h}}$ is called the quasiconformal dilatation of
$\tilde{h}$. Using quasiconformal dilatation, we can define a
distance in $\mathcal{QS}$ by
$$
d_{\mathcal T}(h_{1}, h_{2}) =\frac{\ 1\ }{\ 2\ }\inf \{ \log
K_{\tilde{h}_{1}\tilde{h}_{2}^{-1}}\;|\; \tilde{h}_{1}\in
{\mathcal E}_{h_{1}},\tilde{h}_{2} \in {\mathcal E}_{2}\}.
$$
Here $(\mathcal{QS}, d)$ is called the universal Teichm\"uller
space. It is a complete metric space and a complex manifold with
complex structure compatible with the Hilbert transform.

The topology coming from the metric $d_{\mathcal T}$ on
$\mathcal{QS}$ induces a topology on the factor space
$\mathcal{QS} {\rm \ mod \ }{\mathcal S}$. Given two cosets
${\mathcal S} f$ and ${\mathcal S} g$ in this factor space, define
a metric by
$$
\overline{d}_{T} ({\mathcal S}f, {\mathcal S}g) =\inf_{A, B\in
{\mathcal S}} d(Af, Bg).
$$
The quotient space $\mathcal{QS} {\rm \ mod \ }{\mathcal S}$ with
this metric is a complete metric space and a complex manifold. The
topology on $(\mathcal{QS} {\rm \ mod \ } {\mathcal S},
\overline{d}_{T})$ is the finest topology which makes the
projection $\pi: \mathcal{QS} \to \mathcal{QS} {\rm \ mod \ }
{\mathcal S}$ continuous, and $\pi$ is also holomorphic.

An equivalent topology on the quotient space $\mathcal{QS} {\rm \
mod \ }{\mathcal S}$ can be defined as follows. For any $h\in
\mathcal{QS}$, let $\tilde{h}$ be a quasiconformal extension of
$h$ to a small neighborhood $U$ of $T$ in the complex plane. Let
$$
\mu_{\tilde{h}}
=\frac{\tilde{h}_{\overline{z}}}{\tilde{h}_{z}},\quad z\in U
$$
and
$$
k_{\tilde{h}}=\|\mu(z)\|_{\infty, U} \quad \hbox{and}\quad
B_{\tilde{h}} =\frac{1+k_{\tilde{h}}}{1-k_{\tilde{h}}}.
$$
Then the boundary dilatation $h$ is defined as
$$
B_{h} =\inf_{U, \tilde{h}} B_{\tilde{h}},
$$
where the infimum is taken over all quasiconformal extensions
$\tilde{h}$ of $h$ in a neighborhood $U$ of $T$. It is known that
$h$ is symmetric if and only if $B_{h}=1$. Define
$$
\tilde{d} (h_{1},h_{2}) =\frac{1}{2} \log B_{h_{2}^{-1}h_{1}}.
$$
The two metrics $\overline{d}$ and $\tilde{d}$ on $\mathcal{QS}
{\rm \ mod \ } {\mathcal S}$ are equal.

The Teichm\"uller type metric $d_{\mathcal T} (\cdot, \cdot)$ on
${\mathcal T}{\mathcal C}^{1+}$ is defined similarly as follows.
Let $\Pi$ and $\Pi'$ be two points in ${\mathcal T}{\mathcal
C}^{1+}$. Then
$$
d_{\mathcal T} (\Pi, \Pi') = \frac{1}{2} \log B_{h_{f}^{-1}\circ
h_{g}},
$$
where $\Pi, \Pi' \in {\mathcal T}{\mathcal C}^{1+}$ and $(f,
h_{f})\in \Pi$ and $(g,\tau_{g})\in \Pi'$. Since $d_{\mathcal
T}(\cdot,\cdot)$ is defined by $\tilde{d}(\cdot,\cdot)$, it easy
to check it satisfies the symmetric condition and the triangle
inequality. If we have that $d_{\mathcal T} (\Pi, \Pi')=0$ if and
only if $\Pi=\Pi'$, then $d_{\mathcal T}(\cdot,\cdot)$ is indeed a
metric. To prove this property, we need the following rigidity
result.

Suppose $f, g\in {\mathcal C}^{1+}$ are conjugate by an
orientation-preserving homeomorphism $h$, that is,
$$
f \circ h=h\circ g.
$$
If $h$ is differentiable at $p \in T$, then,
from the last equation, $h$ is differentiable at all points in
$$
BI(p)= \cup_{n=0}^{\infty} f^{-n} (p),
$$
the set of all backward images of $p$.

\vspace*{10pt}
\begin{definition}~\label{diffub}
We call $h$ differentiable at $p\in T$ with uniform bound if
there are a small neighborhood $Z$ of $p$ and a constant $C>0$
such that
$$
C^{-1} \leq |h'(q)| \leq C, \quad q\in BI(p)\cap Z.
$$
\end{definition}

\vspace*{10pt}
\begin{theorem}~\label{onepointrigidity}
Suppose $f, g\in {\mathcal C}^{1+}$ are conjugate by an
orientation-preserving homeomorphism $h$, that is, $f \circ
h=h\circ g$. Then $h$ is a $C^{1}$-diffeomorphism if and only if
$h$ is differentiable at one point with uniform bound.
\end{theorem}

\begin{proof}
Note that $h$ is differentiable if and only if its lift $H$ is
differentiable. If $H$ is a $C^{1}$-diffeomorphism, then
$$
1= H(1)-H(0)= \int_{0}^{1} H'(x) dx.
$$
So there is at least one point in $[0,1]$ such that $H'(x)\neq 0$.
This is the ``only if" part.

To prove the ``if" part, suppose $H$ is differentiable at $x_{0}$ with uniform bound.
Let $p_{0}=\pi (x_{0})$. Then the set of all backward images $BI(p_{0})$ of $p_{0}$ is dense in $T$.
The lift set $\tilde{BI}(x_{0})$ of $BI(p_{0})$ to $[0,1]$ are all points
$$
x_{nm} = F^{-n}(x_{0}+m), \quad n=0, 1, \cdots, \;\; m=0, 1,
\cdots, d^{n}-1.
$$
Since
$$
H(F^{n}(x_{mn}))= G^{n} (H(x_{nm})) \pmod{1},
$$
$$
H'(x_{nm}) =
\frac{H'(x_{0})(F^{n})'(x_{nm})}{(G^{n})'(H(x_{nm}))}.
$$
So $H$ is differentiable at every point in $\tilde{S}$ with
non-zero derivatives. Thus we can take $x_{0}\in (0,1)$.

Let
$$
x_{0}\in \cdots \subset I_{w_{k}}\subset I_{w_{k-1}}\subset \cdots
I_{w_{1}} \subset [0,1]
$$
be a sequence of nested intervals in the sequence of Markov
partitions $\{ \eta_{k}\}_{k=0}^{\infty}$. Assume $I=\overline{Z}=I_{w_{n_{0}}}$ is the
closure of a neighborhood of $x_{0}$ in Definition~\ref{diffub}, that is, there is a constant
$C_{0}>0$ such that
$$
C^{-1}_{0} \leq |H'(x)| \leq C_{0}, \quad x\in \tilde{BI}(x_{0})\cap I.
$$

Consider the set $S(I)$ of all intervals
$J\in \eta_{n_{0}+k}$ such that $J\subset I$ and $F^{k}(J)=I
\pmod{1}$ for $k=1,2,\cdots$. Let $\Omega(I)$ be the union of all
these intervals. Then, just by the expanding property of $f$, the
set $\Omega (I)$ has a full Lebesgue measure in $I$.

For any $J\in S(I)$, $F^{k}(J)=I \pmod{1}$ and $G^{k}(H(J)) =H(I)
\pmod{1}$ for some $k\geq 1$. We have
$$
\frac{|H(J)|}{|J|} = \frac{(F^{k})' (\xi)}{(G^{k})'
(\eta)}\frac{|H(I)|}{|I|}.
$$

Take $x\in \tilde{BI}(x_{0})\cap J$. Then $y=F^{k}(x)\in \tilde{BI}(x_{0})\cap I$
and
$$
\frac{(F^{k})'(x)}{(G^{k})'(H(x))} =\frac{H'(x)}{H'(y)}.
$$
Thus
$$
C^{-2}_{0}\leq \frac{(F^{k})'(x)}{(G^{k})'(H(x))}\leq C_{0}^{2}.
$$
This implies
$$
C^{-2}_{0}  \frac{(F^{k})' (x)}{(F^{k})' (\xi)} \frac{(G^{k})'
(\eta)}{(G^{k})' (x)}\frac{|H(I)|}{|I|}\leq \frac{|H(J)|}{|J|}
\leq C^{2}_{0} \frac{(F^{k})' (\xi)}{(F^{k})' (x)} \frac{(G^{k})'
(x)}{(G^{k})' (\eta)}\frac{|H(I)|}{|I|}.
$$
From the H\"older distortion property~(\ref{smoothdistortion}),
there is a constant $C_{1}>1$ such that
$$
C_{1}^{-1} \leq \frac{|H(J)|}{|J|} \leq C_{1}.
$$
Since both $\Omega (I)$ and $H(\Omega (I))$ have full measures in
$I$ and $H(I)$, respectively, from the additive formula, this
implies that $H|I$ is bi-Lipschtz.

Since $H|I$ is bi-Lipschitz, $H'$ exists a.e. in $I$ and is
integrable. Since $(H|I)'(x)$ is measurable and $H|I$ is a
homeomorphism, we can find a point $y_{0}$ in $I$ and a subset
$E_{0}$ containing $y_{0}$ such that
\begin{itemize}
\item[{\rm 1)}] $H|I$ is differentiable at every point in $E_{0}$;
\item[{\rm 2)}] $y_{0}$ is a density point of $E_{0}$;
\item[{\rm 3)}] $H'(y_{0})\neq 0$; and
\item[{\rm 4)}] the derivative $H'|E_{0}$ is
continuous at $y_{0}$.
\end{itemize}

Since $[0,1]$ is compact, there is a subsequence $\{F^{
n_{k}}(y_{0}) \pmod{1} \}_{k=1}^{\infty}$ converging to a point
$z_{0}$ in $[0,1]$. Without loss of generality, assume $z_{0}\in
(0,1)$. Let $I_{0}=(a, b)$ be an open interval about $z_{0}$.
There is a sequence of interval $\{I_{k}\}_{k=1}^{\infty}$ such
that $y_{0}\in
 I_{k}\subseteq I$ and $F^{ n_{k}}: I_{k}\rightarrow I_{0} \pmod{1}$ is a
$C^{1+\alpha}$ diffeomorphism. Then $|I_{k}|$ goes to zero as $k$
tends to infinity.

From the H\"older distortion property~(\ref{smoothdistortion}),
there is a constant $C_{2}>0$, such that
$$
\Big| \log \Big( \frac{|(F^{ n_{k}})'(w)|}{|(F^{
 n_{k}})'(z)|}\Big) \Big|\leq  C_{2}, \quad \forall w, z\in I_{k},\;\; \forall  k\geq 1.
$$

Since $y_{0}$ is a density point of $E_{0}$, for any integer
$s>0$, there is an integer $k_{s}>0$ such that
$$ \frac{|E_{0}\cap
I_{k}|}{|I_{k}|} \geq 1-\frac{1}{s}, \quad \forall k\geq k_{s}.
$$
Let $E_{k}=F^{n_{k}}(E_{0}\cap I_{k})\pmod{1}$. Then $H$ is
differentiable at every point in $E_{k}$ and, from the H\"older
distortion property~(\ref{smoothdistortion}), there is a constant
$C_{3}>0$ such that
$$
\frac{|E_{k}\cap I_{0}|}{|I_{0}|} \geq 1-\frac{C_{3}}{s}, \quad
\forall k\geq k_{s}.
$$
Let
$$
E=\cap_{s=1}^{\infty}\cup_{k\geq k_{s}} E_{k}.
$$
Then $E$ has full measure in $I_{0}$ and $H$ is differentiable at
every point in $E$ with non-zero derivative.

Next, we are going to prove that $H'|E$ is uniformly continuous.
For any $x$ and $y$ in $E$, let $z_{k}$ and $w_{k}$ be the
preimages of $x$ and $y$ under the diffeomorphism
$F^{n_{k}}:I_{k}\rightarrow I_{0} \pmod{1}$. Then $z_{k}$ and
$w_{k}$ are in $E_{0}$. From $ H\circ F=G\circ H \pmod{1} $, we
have that
$$
 H'(x) = \frac{(G^{ n_{k}})'(H(z_{k}))}{(F^{ n_{k}})'(z_{k})}H'(z_{k})
$$
and
$$
H'(y) = \frac{(G^{
n_{k}})'(H(w_{k}))}{(F^{n_{k}})'(w_{k})}H'(w_{k}).
$$
So
$$
\Big| \log \Big( \frac{H'(x)}{H'(y)}\Big) \Big| \leq \Big| \log
\Big| \frac{(G^{ n_{k}})'(H(z_{k}))}{ (G^{
n_{k}})'(H(w_{k}))}\Big| \Big| + \Big| \log \Big| \frac{(F^{
n_{k}})'(w_{k})}{(F^{n_{k}})'(z_{k})}\Big| \Big| + \Big| \log
\Big( \frac{H'(z_{k})}{H'(w_{k})}\Big) \Big|.
$$

Suppose both $f$ and $g$ are $C^{1+\alpha}$ for some $0<\alpha\leq
1$. From the H\"older distortion
property~(\ref{smoothdistortion}), there is a constant $C_{4}>0$
such that
$$\Bigl|
\log \Big| \frac{(F^{ n_{k}})'(w_{k})}{(F^{ n_{k}})'(z_{k})}\Big|
\Bigr| \leq C_{4} |x-y|^{\alpha}
$$
and
$$ \Bigl| \log
\Big| \frac{(G^{ n_{k}})'(H(z_{k}))}{(G^{n_{k}})'(H(w_{k}))}\Big|
\Bigr| \leq C_{4} |H(x)-H(y)|^{\alpha} $$ for all $k\geq 1$.
Therefore,
$$ \Big| \log \Big(
\frac{H'(x)}{H'(y)} \Big) \Big| \leq C_{4}\Big(|x-y|^{\alpha}
+|H(x)- H(y)|^{\alpha}\Big) +\Big| \log \Big(
 \frac{H'(z_{k})}{H'(w_{k})} \Big) \Big|
$$
for all $k\geq 1$. Since $H'|E_{0}$ is continuous at $y_{0}$, the
last term in the last inequality tends to zero as $k$ goes to
infinity. Hence
$$ \Big| \log \Big( \frac{H'(x)}{H'(y)} \Big)
\Big| \leq C_{4}\Big( |x-y|^{\alpha} +|H(x)- H(y)|^{\alpha}\Big).
$$
This means that $H'|E$ is uniformly continuous. So it can be
extended to a continuous function $\phi$ on $I_{0}$. Because
$H|I_{0}$ is absolutely continuous and $E$ has full measure,
$$ H(x)
=H(a)+\int_{a}^{x}H'(x)dx =H(a) +\int_{a}^{x}\phi(x)dx
$$
on $I_{0}$. This implies that $H|I_{0}$ is actually $C^{1}$.
(This, furthermore, implies that $H|I_{0}$ is $C^{1+\alpha}$).

Now for any $x\in [0,1]$, let $J$ be an open interval about $x$.
By the expanding condition on $f$, there is an integer $n>0$ and
an open interval $J_{0}\subset I_{0}$ such that $F^{ n}: J_{0}\to
J \pmod{1}$ is a $C^{1+\alpha}$ diffeomorphism. By the equation
$H\circ F=G\circ H$, we have that $H|J$ is $C^{1+\alpha}$.
Therefore, $H$ is $C^{1+\alpha}$. We proved the theorem.
\end{proof}

\begin{remark} This kind of the rigidity phenomenon has been also
studied for one-dimensional dynamical systems with critical points
in~\cite{Jiang1,Jiang2,Jiang3,Jiang4,Jiang5}.
\end{remark}

\begin{remark}
If $f$ and $g$ in Theorem~\ref{onepointrigidity} are both $C^{1+1}$, then one can prove that $h$ is bi-Lipschitz by using
a different argument which was given by Sullivan in his lectures at the CUNY Graduate Center during 1986-1989~\cite{SullivanNote}
as follows.

Suppose $h$ is differentiable at a point $x_{0}$ on the circle. Then
$$
h(x)= h(x_{0}) +h'(x_{0}) (x-x_{0}) +o (|x-x_{0}|)
$$
for $x$ close to $x_{0}$. Suppose
$$
f\circ h=h\circ g.
$$
Consider $\{x_{n}=f^{n}(x_{0})\}_{n=0}^{\infty}$. Let $0<a<1$ be a fixed number. Consider the interval $I_{n}=(x_{n}, x_{n}+a)$.
Let $J_{n}=(x_{0}, z_{n})$ be an interval such that
$$
f^{n}: J_{n}\to I_{n}
$$
is a $C^{1+1}$ diffeomorphism.
Let $f^{-n}: I_{n}\to J_{n} $ denote its inverse.
Since $f$ is expanding, the length $|J_{n}|\to 0$ as $n\to \infty$.
Similarly, we have that
$$
g^{n}: h(J_{n}) \to h(I_{n})
$$
is a $C^{1+1}$ diffeomorphism. Let $g^{-n}: h(I_{n}) \to h(J_{n})$ be its inverse.
Then we have that
$$
h(x)= g^{n}\circ h\circ f^{-n} (x), \quad x\in I_{n}.
$$
Let
$$
\alpha_{n} (x) = \frac{x-x_{0}}{x_{n}-x_{0}}: J_{n}\to (0,1)
$$
and
$$
\beta_{n} (x) = \frac{x-h(x_{0})}{h(x_{n})-h(x_{0})}: h(J_{n})\to (0,1).
$$
Then
$$
h(x) = (g^{n}\circ \beta_{n}^{-1})\circ (\beta_{n}\circ h\circ \alpha_{n}^{-1})\circ (\alpha_{n}\circ f^{-n}) (x), \;\;x\in I_{n}
$$
From the H\"older distortion property~(\ref{smoothdistortion}) for $\alpha =1$, we get
$$
\Big| \log |\frac{(f^{-n})' (x)}{(f^{-n})'(y)}|\Big| \leq C |x-y|, \quad \forall \; x, y\in I_{n}
$$
and
$$
\Big| \log |\frac{(g^{-n})' (x)}{(f^{-n})'(y)}|\Big| \leq C |x-y|, \quad \forall \; x, y\in h(I_{n}).
$$
This implies that $g^{n}\circ \beta_{n}^{-1}$ and $\alpha_{n}\circ f^{-n}$ are a sequence of b-Lipschitz homeomorphisms with a uniform Lipschitz constant. Therefore, they have convergent subsequences. The map $\beta_{n}\circ h\circ \alpha_{n}^{-1}$ converges to a linear map. Since the unit circle is compact and all $I_{n}$ with a fixed length $a$,
$\cap_{n=1}^{\infty} I_{n}$ contains an interval $I$. Thus $h$ is a bi-Lipschitz homeomorphism on $I$. Since $f$ and $g$ are expanding, this implies that $h$ is bi-Lipschitz on the whole unit circle $T$.

However, this argument does not work for the case when $0<\alpha <1$. The reason is that in this case, we have only
$$
\Big| \log |\frac{(f^{-n})' (x)}{(f^{-n})'(y)}|\Big| \leq C |x-y|^{\alpha}, \quad \forall \; x, y\in I_{n}
$$
and
$$
\Big| \log |\frac{(g^{-n})' (x)}{(f^{-n})'(y)}|\Big| \leq C |x-y|^{\alpha}, \quad \forall \; x, y\in h(I_{n})
$$
from the H\"older distortion property~(\ref{smoothdistortion}). Therefore, $g^{n}\circ \beta_{n}^{-1}$ and $\alpha_{n}\circ f^{-n}$ are only a sequence of $\alpha$-H\"older homeomorphisms with a uniform H\"older constant. We can not conclude that $h$ is bi-Lipschitz. The method developed in~\cite{Jiang1,Jiang2,Jiang3,Jiang4,Jiang5} (which is presented in the proof of Theorem~\ref{onepointrigidity}) is, in particular, useful for maps having only $C^{1+\alpha}$ smoothness for $0<\alpha<1$.
\end{remark}

\vspace*{10pt}
\begin{theorem}~\label{smoothT}
Suppose $f, g\in {\mathcal C}^{1+}$. Then $(f,h_{f}) \sim_{T}
(g,h_{g})$ if and only if $D^{*}(f) =D^{*}(g)$. Furthermore,
$d_{\mathcal T} (\Pi, \Pi')=0$ if and only if $\Pi=\Pi'$.
\end{theorem}

\begin{proof}
Let $\{ \eta_{n,f}\}_{n=1}^{\infty}$ and $\{
\eta_{n,g}\}_{n=1}^{\infty}$ be the corresponding sequences of
nested partitions on $[0,1]$ for $f$ and $g$. Let
$$
h=h_{f}\circ h_{g}^{-1}.
$$
Then
$$
f\circ h= h\circ g.
$$
Suppose $H$ is the lift of $h$ such that $H(0)=0$. Then for any
interval $I_{w_{n}}\in \eta_{n,f}$, $H(I_{w_{n}})\in \eta_{n,g}$.

Suppose $(f,h_{f}) \sim_{T} (g,h_{g})$. Then $h$ is a
$C^{1}$-diffeomorphism of $T$. For any $w^{*}=\cdots j_{n-1}\cdots
j_{1}j_{0}\in \Sigma^{*}$, let $w^{*}_{n} =j_{n-1}\cdots
j_{1}j_{0}$ and $v_{n-1}^{*} = j_{n-1}\cdots j_{1}$. Then
$$
D^{*}(g)(w_{n}^{*}) = \frac{|H(I_{v_{n-1}})|}{|H(I_{w_{n}})|} =
\frac{H'(\xi)}{H'(\varrho)} \frac{|I_{v_{n-1}}|}{|I_{w_{n}}|}
=\frac{H'(\xi)}{H'(\varrho)} D^{*}(f)(w_{n}^{*}).
$$
This implies that
$$
D^{*}(g)(w^{*}) =D^{*}(f)(w^{*}).
$$

Now suppose $D^{*}(g)(w^{*}) =D^{*}(f)(w^{*})$. Since $f$ and $g$
are both $C^{1+\alpha}$ expanding for some $0<\alpha\leq 1$, there
are constants $C>0$ and $0<\tau<1$ such that
$$
|D^{*}(f) (w^{*}) - D^{*}(f) (w_{n}^{*})|\leq C\tau^{n} \quad
\hbox{and}\quad |D^{*}(g) (w^{*}) - D^{*}(g) (w_{n}^{*})|\leq
C\tau^{n}.
$$
This implies that there is a constant $C'>0$ such that
$$
\frac{D^{*}(g) (w_{n}^{*})}{D^{*}(f) (w_{n}^{*})} \leq
1+C'\tau^{n}, \quad \forall n>0.
$$
Let $C''=\prod_{n=0}^{\infty} (1+C'\tau')$. Then
$$
\frac{|H(I_{w_{n}})|}{|I_{w_{n}}|} = \prod_{k=0}^{n}
\frac{D^{*}(g) (w_{n-k}^{*})}{D^{*}(f) (w_{n-k}^{*})}\leq C'',
\quad \forall w_{n}, \; \forall n>0.
$$
From the additive formula, we conclude that $H$ is Lipschitz
continuous. But a Lipschitz continuous function is absolutely
continuous (at this point, we can also use a theorem of Shub and
Sullivan~\cite{ShubSullivan} to show that $h$ is a
$C^{1}$-diffeomorphism), so it is differentiable almost
everywhere. Since $H$ is a homeomorphism, it must have a
differentiable point with non-zero derivative. Now
Theorem~\ref{onepointrigidity} implies that that $h$ is
$C^{1}$-diffeomorphism.

From the definition, $d_{\mathcal T} (\Pi, \Pi')=0$ if and only if
$h=h_{f}^{-1}\circ h_{g}$ is symmetric. If $\Pi=\Pi'$, then
$h=h_{f}^{-1}\circ h_{g}$ is $C^{1}$-diffeomorphism. So it is
symmetric.

On the other hand, if $h=h_{f}^{-1}\circ h_{g}$ is symmetric, then
from Lemma~\ref{qsdistortion}, there is a bounded function
$\varepsilon (t)>0$ such that $\varepsilon (t)\to 0$ as $t\to 0$
and a constant $0<\tau<1$ such that
$$
|D^{*}(g)(w_{n}^{*})-D^{*}(f)(w_{n}^{*})| = \Big|
\frac{|H(I_{v_{n-1}})|}{|H(I_{w_{n}})|}-
\frac{|I_{v_{n-1}}|}{|I_{w_{n}}|}\Big| \leq \varepsilon
(\tau^{n}).
$$
We get
$$
D^{*}(g)(w^{*})=D^{*}(f)(w^{*}).
$$
This further implies that $h$ is a $C^{1}$-diffeomorphism. So
$\Pi=\Pi'$.
\end{proof}

\begin{definition} We call $d_{\mathcal T}(\cdot,\cdot)$ the
Teichm\"uller metric on ${\mathcal T}{\mathcal C}^{1+}$.
\end{definition}

Following Theorem~\ref{smoothT}, we can set up a one-to-one
correspondence between the Teichm\"uller space ${\mathcal
T}{\mathcal C}^{1+}$ and the space of all H\"older continuous dual
derivatives:
$$
\Pi =[(f, h_{f})] \to D^{*}(f)(w^{*}).
$$
Therefore,
\begin{equation}~\label{smoothrepresentation}
{\mathcal T}{\mathcal C}^{1+}=\{ D^{*}(f)(w^{*})\;\; |\;\; f\in
{\mathcal C}^{1+}\}
\end{equation}
equipped with the Teichm\"uller metric $d_{\mathcal T}(\cdot,
\cdot)$. However, this is not a complete space. Next we will study
the completion of this space.

Let $d\geq 2$ be the same fixed integer. Suppose ${\mathcal
U}{\mathcal S}$ is the space of all uniformly symmetric circle
endomorphisms of degree $d$. We define the Teichm\"uller space for
${\mathcal U}{\mathcal S}$ as we did for ${\mathcal C}^{1+}$.

Let $q_{d} (z) =z^{d}$ be the basepoint in ${\mathcal U}{\mathcal
S}$. A {\em marked circle endomorphism} by $q_{d}$ is a pair
$(f,h_{f})$, where $f\in {\mathcal U}{\mathcal S}$ and $h_{f}$ is
the orientation-preserving homeomorphism of $T$ such that
$h_{f}(1)=1$ and
$$
f\circ h_f= h_f\circ q_{d}.
$$
From Corollary~\ref{quasisymmetriccojugacy}, for any marked circle
endomorphism $(f,h_{f})$ by $q_{d}$, $h_{f}$ is quasisymmetric.
Thus we can define Teichm\"uller equivalence relation $\sim_{T}$,
Teichm\"uller space, and Teichm\'uller metric as follows.

\begin{definition}
Two marked circle endomorphisms are equivalent, denoted as
$(f,h_{f}) \sim_{T} (g,h_{g})$, if $h_{f}\circ h_{g}^{-1}$ is a
symmetric homeomorphism.
\end{definition}

\begin{definition}
The Teichm\"uller space
$$
{\mathcal T}{\mathcal U}{\mathcal S}=\{ [(f, h_f)] \;|\; f\in
{\mathcal U}{\mathcal S}, \; \hbox{with the basepoint $[(q_{d},
id)]$}\}
$$
is the space of all $\sim_{T}$-equivalence classes $[(f, h_f)]$ in
the space of all marked circle endomorphisms by $q_{d}$.
Teichm\"uller metric $d_{\mathcal T}(\cdot, \cdot)$ is defined as
$$
d_{\mathcal T}(\Psi, \Psi') = \frac{1}{2} \log B_{h_{f}^{-1}\circ
h_{g}}
$$
where $(f, h_{f})\in \Psi$ and $(g,h_{g})\in \Psi'$.
\end{definition}

If $f, g\in {\mathcal C}^{1+}$ and if the conjugacy $h$ between
$f$ and $g$ is symmetric, then from Theorem~\ref{smoothT}, $h$
must be a $C^{1}$-diffeomorphism. This implies that the
Teichm\"uller space ${\mathcal T} {\mathcal C}^{1+}$ is indeed a
subspace of the Teichm\"uller space ${\mathcal T}{\mathcal
U}{\mathcal S}$. Furthermore, we have that

\vspace*{10pt}
\begin{theorem}~\label{completion}
The space $({\mathcal T}{\mathcal U}{\mathcal S}, d_{\mathcal
T}(\cdot, \cdot))$ is a complete space and is the completion of
the space $({\mathcal T} {\mathcal C}^{1+}, d_{\mathcal T}(\cdot,
\cdot))$.
\end{theorem}

Our proof of this theorem needs some result for asymptotically conformal circle
endomorphisms in~\cite{GardinerJiangAAACCE}. For the purpose of self-contained of this paper
and for the convenience of the reader, we includes some materials from~\cite{GardinerJiangAAACCE}
in the next section. Therefore, we delay the proof of Theorem~\ref{completion} into the next section.

\section{Asymptotically conformal circle endomorphisms}

Suppose $g$ is a quasiconformal homeomorphism defined on a plane domain $\Omega$.
Let
$$
\mu(z) =\frac{g_{\overline{z}}}{g_{z}}
$$
for $z\in \Omega$ and let
$$
K_z(g)=\frac{1+|\mu (z)|}{1-|\mu(z)|}.
$$
Here $K_{z}(g)$ is called the dilatation of $g$ at $z$.

Suppose $H$ is a quasisymmetric homeomorphism of the real line.
Define the skew quasisymmetric distortion function as
$$
\rho (x,y,k) =\frac{|H(x+ky)-H(x)|}{|H(y)-H(x)|}.
$$
In particular, let $\rho(x,y) =\rho (x,y,1)$.
The Beurling-Ahlfors extension procedure provides a canonical
extension $\tilde{H}$ of any quasisymmetric homeomorphism $H$ to the whole complex plane such
that the Beltrami coefficient $\mu$ of $\tilde{H}$ satisfies
$\|\mu\|_{\infty}<1.$  Furthermore, it satisfies the following
well-known theorem (see~\cite{GardinerSullivanSymm}).

\vspace*{10pt}
\begin{theorem}~\label{oneBAextension} The Beurling-Ahlfors extension of a
quasisymmetric self-mapping $H$ of the real axis has a  Beltrami coefficient $\mu$ with
$|\mu(x+iy)|\leq \eta(y)$ for some vanishing function $\eta(y)$
if, and only if, there is a vanishing function $\epsilon(y)$ such
that
$$
\frac{1}{1+\epsilon(y)} \leq \rho_H(x,y) \leq
1+\epsilon(y).
$$
\end{theorem}

A generalization of Theorem~\ref{oneBAextension} can be founded in~\cite{Cui} and in~\cite{GardinerJiangAAACCE} with a complete proof.

\vspace*{10pt}
\begin{theorem}~\label{twoBAextension} Suppose the skew quasisymmetric
distortion functions $\rho_0(x,y,k)$ and $\rho_1(x,y,k)$ of $H_0$
and $H_1$ satisfy the inequality
$$
|\rho_0(x,y,k)-\rho_1(x,y,k)|,\;\; |\rho_0(x,-y,k)-\rho_1(x,-y,
k)| \leq \epsilon(y)
$$
for $x,y>0\in {\mathbb R}$ and $0<k\leq 1$, where $\epsilon(y)$ is
a vanishing function, that is, $\epsilon (y) \to 0$ as $y\to {\mathbb R}$. Suppose furthermore that $\mu_0$ and $\mu_1$
are the Beltrami coefficients of the Beurling-Ahlfors extensions
$\tilde{H}_0$ and $\tilde{H}_1,$ that is,
$$
\mu_0(z)=\frac{\tilde{H_0}_{\overline{z}}}{\tilde{H_0}_z}
{\ \rm \ and \ }
\mu_1(z)=\frac{\tilde{H_1}_{\overline{z}}}{\tilde{H_1}_z}.$$ Then
there is a vanishing function $\eta(y)$ depending only on
$\epsilon(y)$ such that
$$|\mu_0(x+iy)-\mu_1(x+iy)| \leq \eta(y).$$

Conversely, given two quasiconformal maps $\tilde{H}_0$ and
$\tilde{H}_1$ preserving the real axis and a vanishing function
$\eta(y)$ such that
$$
|\mu_{0}(z)-\mu_{1}(z)|\leq\eta(y),
$$
then there is a vanishing function $\epsilon(y)$ such that
$$
|\rho_0(x,y,k)-\rho_1(x,y,k)|,\;\; |\rho_0(x,-y,k)-\rho_1(x,-y,
k)| \leq \epsilon(y)
$$
for $x,y>0\in {\mathbb R}$ and $0<k\leq 1$, where $H_{0}$ and
$H_{1}$ are the restrictions of $\tilde{H}_{0}$ and
$\tilde{H}_{1}$ to the real axis.
\end{theorem}

\begin{proof}
We adapt the proof from~\cite{GardinerJiangAAACCE}. We take the following formulas as the definition of the
Beurling-Ahlfors extension:
$$
{\widetilde H} = U+iV,
$$ where
\begin{equation}\label{ba1}
U(x,y)=\frac{1}{2y} \int_{x-y}^{x+y} H(s) ds =
\frac{1}{2}\int_{-1}^1 H(x+ky)dk
\end{equation}
and
\begin{equation}\label{ba2}
V(x,y)=\frac{1}{y} \int_x^{x+y} H(s)ds - \frac{1}{y}
\int_{x-y}^{x} H(s)ds. \end{equation}
 In (\ref{ba1}) and (\ref{ba2}) we have
chosen a normalization slightly different from the one given in
\cite{Ahlfors}.  It has the property that the extension of
the identity is the identity and  the extension is affinely
natural, by which we mean that for affine maps $A$ and $B,$
$$ \widetilde{id_{\mathbb R}} = \ id_{\mathbb C}$$ and
$$\widetilde{A \circ H \circ B} = A \circ {\widetilde H} \circ B.$$

Note that
\begin{equation}\label{1}
\int_0^1 \rho(x,y,k) dk =\frac{1}{H(x)-H(x-y)} \left(\frac{1}{y}\int_x^{x+y}
H(s) ds -H(x) \right)
\end{equation}
and
\begin{equation}\label{2}
\int_0^1 \rho(x,-y,k) dk = \frac{1}{H(x+y)-H(x)}
\left(H(x)-\frac{1}{y}\int_{x-y}^x H(s)ds\right).
\end{equation}

Let

\begin{equation}\label{leftright}
\begin{array}{lll}
  L & = & H(x)-H(x-y)\\
  R & = & H(x+y)-H(x)\\
  L' & = &  H(x) - \frac{1}{y}\int_{x-y}^{x}H(s)ds , \\
 R' & = & \frac{1}{y}\int_{x}^{x+y} H(s) ds - H(x).\\
 \end{array}
 \end{equation}

\noindent and let $\rho_+ (x,y)= \int_0^1 \rho(x,y,k) dk$ and
$\rho_-(x,y)= \int_0^1 \rho(x,-y,k) dk$. Let
$\rho(x,y)=\rho_{H}(x,y)$. Then

\vspace{.1in}

 \begin{equation}\label{avgs}
 \begin{array}{lll}
  \rho(x,y) & = & R/L\\
   \rho_+(x,y) & = & R'/L\\
   \rho_-(x,y) & = & L'/R.
   \end{array}
 \end{equation}

\vspace{.1in}

\noindent Notice that  for symmetric homeomorphisms the quantity
$\rho$ approaches $1$ and the two quantities $\rho_+$ and $\rho_-$
approach $1/2$ as $y$ approaches zero.
 The complex dilatation of $\tilde{H}$ is given by
 $$\mu(z)=\frac{K(z)-1}{K(z)+1}$$ where
 $$K(z)=\frac{{\tilde H}_{z} + {\tilde H}_{\overline{z}}}
 {{\tilde H}_{z} - {\tilde H}_{\overline{z}}}=
 \frac{(U+iV)_z + (U+iV)_{\overline{z}}}{(U+iV)_z - (U+iV)_{\overline{z}}}$$
 $$=\frac{(U+iV)_x -i(U+iV)_y + (U+iV)_x +i(U+iV)_y}{(U+iV)_x -
 i(U+iV)_y-
 (U+iV)_x-i(U+iV)_y}$$
 $$=\frac{U_x+iV_x}{V_y-iU_y}.$$
 Thus
 $$K(z)=\frac{1+ia}{b-ic},$$
where $a=V_x/U_x, b=V_y/U_x {\rm \ and \ } c=U_y/U_x.$

 To find estimates for these three ratios we must find expressions
 for the four partial derivatives of $U$ and $V$ in (\ref{ba1}) and
 (\ref{ba2}). In the notation of (\ref{leftright})

\vspace{.1in}

$\begin{array}{lll}
U_x & = & \frac{1}{2y}(R+L),\\
V_x & = & \frac{1}{y}\left(R-L \right),\\
V_y & = & \frac{1}{y}\left(R+L\right)-\frac{1}{y}\left(R'+L'\right) , \\
U_y & = & \frac{1}{2y}\left(R-L\right) -\frac{1}{2y}\left(R'-L'\right) . \\
\end{array}$

\vspace{.05in}

\noindent Thus

\vspace{.05in}

$\begin{array}{cll}
a(1+\rho)& = & 2\frac{R-L}{R+L}\cdot \frac{R+L}{L},\\

b(1+\rho) & = & 2\frac{R+L-R'-L'}{R+L}\cdot \frac{R+L}{L}=
2 \left(R/L + 1 - R'/L  -(R/L)(L'/R)\right),\\

c(1+\rho) & = & \frac{R-L-R'+L'}{R+L} \cdot \frac{R+L}{L}= R/L -1
-R'/L + (R/L)(L'/R).

\end{array}$

\vspace{.1in}

 \noindent Finally, we obtain
\begin{equation}\label{rho}
\begin{array}{cll}
a& = & \frac{2(\rho -1)}{\rho+1}, \\

b & = & \frac{2(\rho+1-\rho_+ -\rho\rho_-)}{\rho+1},\\

c& = & \frac{\rho -1 + \rho_+ +\rho \rho_-}{\rho+1 } .

\end{array}
\end{equation}

Since $K(z)=(1+ia)/(b-ic),$ $K(z)+1=(1+ia+b-ic)/(b-ic),$ we have
$$\mu_1(z)-\mu_0(z) = \frac{K_1(z)-1}{K_1(z)+1}-\frac{K_0(z)-1}{K_0(z)+1} =
2\frac{K_1(z)-K_0(z)}{(K_1(z)+1)(K_0(z)+1)}=$$
$$2\frac{(1+ia_1)(b_0-ic_0)-(1+ia_0)(b_1-ic_1)}{(1+ia_1+b_1-ic_1)
(1+ia_0+b_0-ic_0)}=$$
\begin{equation}\label{main}
2\frac{(a_1-a_0)(i
b_1+c_1)+(b_0-b_1)(1+ia_1)+(c_1-c_0)(i-a_1)}{(1+ia_1+b_1-ic_1)
(1+ia_0+b_0-ic_0)}. \end{equation}

From the equation for $b$ in (\ref{rho}) and the inequalities
$\rho_+<\rho$ and $\rho \rho_- <1,$ we see that $b>0.$  Since this
inequality is true for $b_1$ and for $b_0,$ it follows that the
denominator in (\ref{main}) is greater than $1.$
 These equations show that
if $a_0, b_0, c_0$ converge to $a_1, b_1, c_1,$ as $y$ approaches
zero, then $\mu_0$ approaches $\mu_1.$  Clearly $\rho_0$
approaches $\rho_1$ implies $a_0$ approaches $a_1$.

From the hypothesis
$$
|\rho_0(x,y,k)-\rho_1(x,y,k)|,\;\; |\rho_0(x,-y,k)-\rho_1(x,-y,k)|
\leq \epsilon (y)
$$
for $x,y>0\in {\mathbb R}$ and $0<k \leq 1$, we have that
$$
|\rho_{1+}(x,y)-\rho_{0+}(x,y)| \leq \int_0^1
|\rho_1(x,y,k)-\rho_0(x,y,k)| dk \leq \epsilon (y)
$$
and
$$
|\rho_{1-}(x,y)-\rho_{0-}(x,y)| \leq \int_0^1
|\rho_1(x,-y,k)-\rho_0(x,-y,k)| dk \leq \epsilon (y).
$$
This implies that $b_0, c_0$ converge to $b_1, c_1$, as $y$
approaches zero. This completes the proof of the first half of the
theorem.

Since the subsequent arguments do not require the second half, we
only sketch the proof.  Notice that if $\tilde{H}_0$ and
$\tilde{H}_1$ are quasiconformal self-maps of the complex plane
preserving the real axis with Beltrami coefficients $\mu_0$ and
$\mu_1$ satisfying
$$
|\mu_0(z)-\mu_1(z)| \leq \epsilon(y)
$$
for a vanishing function $\epsilon(y),$ then the quasiconformal
map $\tilde{H}_1 \circ (\tilde{H}_0)^{-1}$ has Beltrami
coefficient $\sigma$ with
$$
|\sigma(z)|\leq \epsilon'(y)
$$
for another vanishing function $\epsilon'(y).$  Then $\tilde{H}_1
\circ (\tilde{H}_0)^{-1}$ carries the extremal length problem for
the family of curves joining $[-\infty, \tilde{H}_0(x-y)]$ to
$[\tilde{H}_0(x),\tilde{H}_0(x+ky)]$ to the extremal length
problem for the family of curves joining $[-\infty,
\tilde{H}_1(x-y)]$ to $[\tilde{H}_1(x),\tilde{H}_1(x+ky)].$ If
$\Lambda_0(x,y,k)$ and $\Lambda_1(x,y,k)$ are these two extremal
lengths, then by the Gr\"otzsch argument there is another
vanishing function $\epsilon''(y)$ such that
$$\left|\log \frac{\Lambda_0(x,y,k)}{\Lambda_1(x,y,k)} \right| \leq \epsilon''(y).$$

In~\cite{Ahlfors} Ahlfors shows that if
$\Lambda$ is the extremal length of the curve family that joins
the interval $[-\infty,-1]$ to $[0,m],$ $\Lambda$ is an increasing
real analytic function of $m.$ In particular,
\begin{equation}\label{CI}
\left|\log \frac{m_0}{m_1}\right|<\epsilon {\rm \ if \ and \ only
\ if \ } \left|\log
\frac{\Lambda_0}{\Lambda_1}\right|<\epsilon'
\end{equation} and
$\epsilon$ and $\epsilon'$ approach zero simultaneously.

Hence by~(\ref{CI}) there is another vanishing function $\eta(y)$
such that
$$
\left|\frac{H_0(x+ky)-H_0(x)}{H_0(x)-H_0(x-y)}
-\frac{H_1(x+ky)-H_1(x)}{H_1(x)-H_1(x-y)}\right| \leq \eta(y).
$$
Similarly, we have that
$$
\left|\frac{H_0(x)-H_0(x-ky)}{H_0(x+y)-H_0(x)}
-\frac{H_1(x)-H_1(x-ky)}{H_1(x+y)-H_1(x)}\right| \leq \eta(y).
$$
This completes the proof of the second half of the theorem.
\end{proof}

\vspace*{10pt}

\begin{definition}~\label{asymconf}
We call a circle endomorphism $f$ of degree $d\geq 2$ a uniformly asymptotically conformal if
it has a reflection invariant extension $\tilde{f}$ defined in a small
annulus $r<|z|<1/r$ such that
$$
\tilde{f}(1/\overline{z})=1/\overline{\tilde{f}(z)},
$$
and such that for every $\epsilon>0$ there exists a possibly
smaller annulus $U=\{z:r'<|z|<1/r'\}$ such that
\begin{equation}\label{uac1}
K_z(\tilde{f}^{-n})<1+\epsilon
\end{equation}
for almost all $z$ in $U.$ and all $n>0$.
\end{definition}

From the quasiconformal mapping theory (see~\cite{Ahlfors}),
if $\tilde{f}$ is a uniformly asymptotically conformal, then the restriction $f$ of $\tilde{f}$
to the unit circle $T$ is uniformly symmetric.
It is also easy to see that if $\tilde{f}$ acting on a
neighborhood of the unit circle with $\tilde{f}(1)=1$ is uniformly asymptotically conformal if and only if there is a unique
lift $\tilde{F}$ to an infinite strip containing ${\mathbb R}$ and
bounded by lines parallel to ${\mathbb R}$ such that
\begin{itemize}
\item[1)] $\pi \circ \tilde{F} = \tilde{f} \circ \pi,$,
\item[2.)] $\tilde{F}(0)=0$,
\item[3.)] $\tilde{F} (z+1) = \tilde{F}(z) +d$, and
\item[4.)] $\tilde{F}$ preserves the real axis and $\tilde{F}(\overline{z})=\overline{\tilde{F}(z)}$.
\end{itemize}

In light of the above, we have an equivalent definition.

\begin{definition}
We call a circle endomorphism $f$ a uniformly asymptotically conformal if
for every $\epsilon>0$, there is a $\delta>0$ such that if the absolute value of $y = Im \ z$ is
less than $\delta$, then
\begin{equation}\label{unif1}
K_z(\tilde{F}^{-n})<1+\epsilon
\end{equation}
for all $n>0$.
\end{definition}

The following theorem is proved in~\cite{GardinerJiangAAACCE}. For the purpose of self-contained of this paper, we include the proof.

\vspace*{10pt}
\begin{theorem}~\label{uaatouac} A circle endomorphism $f$ of degree $d\geq 2$ is
uniformly symmetric if and only if it is uniformly asymptotically conformal.
\end{theorem}

Before to prove this theorem, we prove the following lemma. Let $\zeta(M)$ be
the function in Lemma~\ref{qsdistortion}.

\vspace{10pt}
\begin{lemma}~\label{sdc}
Let $\vartheta (M) =M-1 + M\zeta (M)$. Then for any homeomorphism
$H$ of ${\mathbb R}$ and any $x, y>0\in {\mathbb R}$, if $H$
restricted to the interval $[x-y,x+y]$ is $M$-quasisymmetric, then
$$
\max \left\{|\rho_{H}(x,y, k)-k|,  |\rho_{H}(x,-y, k)-k|
\right\}\leq \vartheta (M),\quad \forall\; 0<k\leq 1.
$$
\end{lemma}

\begin{proof}
Consider $\hat{H}(k)= (H(x+ky) -H(x))/(H(x+y)-H(x))$. Then
$\hat{H}(1)=1$ and  $\hat{H}(0)=0.$ Also, $\hat{H}$ is
quasisymmetric because
$$\frac{\hat{H}(k+j)-\hat{H}(k)}{\hat{H}(k)-\hat{H}(k-j)}=
\frac{{H}(x+ky+jh)-{H}(x+ky)}{{H}(x+ky)-{H}(x+ky-jy)}
$$
for any $0\leq k\leq 1$ and $j>0$ such that $[k-j,k+j]\subset
[0,1]$ and this is bounded above by $M$ and below by $1/M$ because
$H$ is $M$-quasisymmetric. So, from Lemma~\ref{qsdistortion},
$$
k-\zeta(M)
 \leq \frac{H(x+ky) -H(x)}{H(x+y)-H(x)}
 \leq k+\zeta(M).
$$
Thus
$$
(k-\zeta(M))\rho_{H} (x,y) \leq \frac{H(x+ky) -H(x)}{H(x)-H(x-y)}
\leq (k+\zeta (M))\rho_{H}(x,y).
$$
Since $ 1/M \leq \rho_{H} (x,y) \leq M $ and we are assuming that
$0<k \leq 1,$ this implies that
$$
 |\rho_{H} (x,y,k)-k| \leq  \vartheta (M)=M-1+M\zeta (M).
$$
Similarly, we have that
$$
 |\rho_{H} (x,-y,k)-k| \leq  \vartheta (M)=M-1+M\zeta (M).
$$
\end{proof}

\begin{proof}[Proof of Theorem~\ref{uaatouac}]
We only need to prove the "only if" part. Let $F$ be the lift to the real axis of
$f$ such that $F(0)=0,$ $F (x+1) = F(x)+d$ and such that
$\pi \circ F = f \circ \pi.$  By Theorem~\ref{quasisymmetriccojugacy}, there
is a quasisymmetric homeomorphism $H$ of ${\mathbb R}$ fixing $0$
and $1$ such that
\begin{itemize}
\item[i)] $H \circ P \circ H^{-1} = F$ where $P(x)=dx,$  and
\item[ii)] $H \circ T \circ H^{-1} = T$  where $T(x)=x+1.$
\end{itemize}
It will suffice to find an extension
$\tilde{F}$ of $F$ such that
\begin{enumerate}
\item $\tilde{F} \circ T (z) = T^d \circ \tilde{F}(z)$ and
\item the Beltrami coefficients $\mu_{\tilde{F}^{-n}}$ of $\tilde{F}^{-n}$
satisfy
$$
|\mu_{\tilde{F}^{-n}}(x+iy)|\leq \epsilon(y)
$$
where $\epsilon(y)$ is independent of $n$ and $x.$
\end{enumerate}

Let $\tilde{H}$ be a reflection invariant quasiconformal extension of $H$.
We define
$$
\tilde{F} =\tilde{H} \circ P \circ\tilde{H}^{-1}
$$
since $\tilde{H}$ extends $H,$ clearly $\tilde{F}$ extends $F$ and is a reflection invariant extension.

Suppose $\rho_0(x,y,k)$ and $\rho_1(x,y,k)$ are the skew
quasisymmetric distortions of $F^{-n} \circ H$ and $H.$  By
Lemma~\ref{sdc}, there is a vanishing function $\epsilon(y)$
such that
$$
|\rho_0(x,y,k)-\rho_1(x,y,k)|,\;\; |\rho_0(x,-y, k)-\rho_1(x,-y,
k)| \leq \epsilon(y)
$$
for all real numbers $x,$ all $y>0,$ all $k$ with $0<k\leq 1$ and
all $n\geq 1$. Applying Theorem~\ref{twoBAextension}, there is another
vanishing function $\eta(y)$ such that the Beltrami coefficients
$\mu_{\widetilde{F^{-n} \circ H}}$ and $\mu_{\tilde{H}}$ satisfy
$$
|\mu_{\widetilde{F^{-n} \circ H}}(z)- \mu_{\tilde{H}}(z)| \leq
\eta(y), \quad \forall\; n>0.
$$

Since  $$\widetilde{F^{-n} \circ H}
    = \widetilde{H \circ P^{-n}} = \tilde{H} \circ P^{-n},$$
    we conclude that $$|\mu_{\tilde{H}}(d^{-n}z)-\mu_{\tilde{H}}(z)| \leq
    \eta(y).$$
    Also, since the Beurling-Ahlfors extension is affinely natural,
    $\mu_{\tilde{H}}(T(z))=\mu_{\tilde{H}}(z)$
    and $\tilde{H} \circ T \circ \tilde{H}^{-1}(z) = T(z).$ We
    conclude that $\tilde{F}=\tilde{H} \circ P \circ
    \tilde{H}^{-1}$ is uniformly asymptotically
    conformal.
\end{proof}

\begin{proof}[Proof of Theorem~\ref{completion}]
Suppose $\{ \kappa_n\}_{n=1}^{\infty} =\{ [(f_n,
h_n)]\}_{n=1}^{\infty}$ is a Cauchy sequence in ${\mathcal
T}{\mathcal U}{\mathcal S}$. Then
$$
d_{\mathcal T}({\mathcal S} h_n,{\mathcal S}h_m) \to 0 \qquad
\hbox{as} \qquad m,n\to \infty.
$$
We may assume by working modulo ${\mathcal S}$ that
$h_n^{-1}h_{m}$ tends to the identity map as $m$ and $n$ go to
infinity. Therefore, $\{ h_n\}_{n=1}^{\infty}$ is a Cauchy
sequence in the universal Teichm\"uller space and $h_n$ tends to a
quasisymmetric map $h$ as $n$ goes to infinity.

Since $f_{n} = h_{n}^{-1}q_{d}h_{n}$ for all $n\geq 1$, $f_{n}=
h^{-1}_{n}h_m f_{m} h^{-1}_{m}h_{n}$ for all $n,m\geq 1$. Let
$g_{n,w_{k}}$ be inverse branches of $f^{k}_n$ defined on
$T\setminus \{1\}$. By considering their lifts to ${\mathbb R}$,
we can think of them as defined on the whole circle $T$. Then
$$
g_{n,w_{k}} = h^{-1}_{n}h_m g_{m,w_{k}} h^{-1}_{m}h_{n}.
$$
Let $k_{nm} = h_{n}^{-1}h_{m}$ and let
$$
\rho_{nm} =\sup_{x\in T, t>0}
\frac{|k_{nm}(x+t)-k_{nm}(x)|}{|k_{nm}(x)-k_{nm}(x-t)|}
$$
be its quasisymmetric distortion. Then $\rho_{nm}\to 1$ as $n,m\to
\infty$. Let
$$
\rho (g_{n, w_{k}}, t)=\sup_{x\in T} \frac{|g_{n,
w_{k}}(x+t)-g_{n, w_{k}}(x)|}{|g_{n, w_{k}}(x)-g_{n,
w_{k}}(x-t)|}, \quad t>0,
$$
be the quasisymmetric distortion of $g_{n, w_{k}}$. Then we have
that
$$
\rho (g_{n, w_{k}}, t) \leq \rho_{nm}^{2} \rho (g_{m, w_{k}},
t),\quad \forall\; n,m \geq 1.
$$
So there is a positive bounded function $\epsilon (t)\to 1$ as
$t\to 0$ such that
$$
\rho(g_{n, w_{k}}, t)\leq \epsilon (t), \quad \forall\; n\geq
1,\;\; \forall \; w_{k}, \;\; \forall \; t>0.
$$

Define $f=h^{-1}q_{d}h$. Let $g_{w_{k}}$ be inverse branches of
$f^{k}$ defined on $T\setminus \{1\}$. By considering their lifts
to ${\mathbb R}$, we think of them as defined on the whole circle
$T$. Let
$$
\rho(g_{w_{k}}, t) =\sup_{x\in T}
\frac{|g_{w_{k}}(x+t)-g_{w_{k}}(x)|}{|g_{w_{k}}(x)-g_{w_{k}}(x-t)|},
\quad t>0,
$$
be the quasisymmetric distortion of $g_{w_{k}}$.

Let $l_n=h^{-1}h_n$. Then $f=l_nf_nl_n^{-1}$ for all $n>0$. Let
$$
\rho(l_n) =\sup_{x\in T, t>0}
\frac{|l_{n}(x+t)-l_{n}(x)|}{|l_{n}(x)-l_{n}(x-t)|}
$$
be the quasisymmetric constant of $l_n$. Then $\{
\rho(l_n)\}_{n=1}^{\infty}$ is a bounded sequence. (Actually,
$\rho_{n}\to 1$ as $n\to \infty$.)

Since $g_{w_{k}} =l_n g_{n,w_{k}}l_n^{-1}$,
$$
\rho(g_{w_{k}}, t)\leq (\rho(l_n))^2 \epsilon(t) \leq \sup_{n\geq
1} \{(\rho(l_n))^2\} \epsilon(t), \quad \forall\; w_{k},\;\;
\forall \;n>0, \;\;\forall \; t>0.
$$
This means that $f$ is uniformly symmetric, so $[(f, h)]\in
{\mathcal T}{\mathcal U}{\mathcal S}$. But $f_{n}\to f$ as $n\to
\infty$ in ${\mathcal T}{\mathcal U}{\mathcal S}$. Thus ${\mathcal
T}{\mathcal U}{\mathcal S}$ is complete.

For any $[(f,h)]\in {\mathcal T}{\mathcal U}{\mathcal S}$ and any
$\epsilon>0$, we will prove that there is an analytic circle map
$f_{\epsilon}$ in ${\mathcal C}^{1+\alpha}$ such that
$[(f_{\epsilon}, h_{\epsilon})]$ is in the $\epsilon$-neighborhood
of $[(f,h)]$ in ${\mathcal T}{\mathcal U}{\mathcal S}$. We use a
technique in complex dynamics (refer to~\cite{DouadyHubbard}) to
construct $f_{\epsilon}$ as follows.

Consider a quasiconformal extension $\tilde{h}$ of $h$ to the
complex plane. Then $\tilde{f}=\tilde{h}q_{d}\tilde{h}^{-1}$ is a
quasiregular map of the complex plane. Let
$$
\mu_{\tilde{f}^n}(z)
=\frac{(\tilde{f}^n)_{\overline{z}}(z)}{(\tilde{f}^n)_{z}(z)}
$$
be the Beltrami coefficient of $\tilde{f}^n$. Assume
$\mu_{\tilde{f}^n}(z)$ is symmetric about the unit circle, that
is, $\mu_{\tilde{f}^n}(z)=\mu_{\tilde{f}^n}(1/\overline{z})$.

Since $f$ is uniformly symmetric, from Theorem~\ref{uaatouac},
we can pick an extension
$\tilde{f}$ (equivalently, pick an extension $\tilde{h}$ of the
conjugacy $h$) such that there is a function $\gamma(t)\to 0$ as
$t\to 0$ and such that $|\mu_{\tilde{f}^n} (z)| \leq
\gamma(|z|^{2n}-1)$ for all $n>0$ and a.e.  $z$. From calculus,
$$
\mu_{\tilde{f}^n}(z) =\frac{\mu_{\tilde{h}}(q_{d}^n(z))
-\mu_{\tilde{h}}(z)}{1+ \mu_{\tilde{h}}(q_{d}^n(z))
\overline{\mu_{\tilde{h}}(z)}}\Theta (z), \quad \hbox{where
$|\Theta (z)|=1$}.
$$
This implies that
$$
|\mu_{\tilde{h}}(q_{d}^n(z)) -\mu_{\tilde{h}}(z)| \leq C \gamma
(|z|^{2^n}-1)
$$
for all $n>0$ and a.e. $z$ where $C>0$ is a constant. For any
$\epsilon>0$, we have a $\delta >0$ such that
$\gamma(t)<\epsilon/C$ for all $0\leq t<\delta$. Let
$$
A_0=\{ z\in {\mathbb C}\; |\; 1-\delta <|z|<(1-\delta)^{1/2}
\}\cup
      \{ z\in {\mathbb C}\; |\; (1+\delta)^{1/2} <|z|<1+\delta \}
$$
and set $A_n=q_{d}^{-n}(A_0)$. Define $\mu (z) =
\mu_{\tilde{h}}(z)$ for $z\in {\mathbb C}\setminus
(\cup_{n=1}^{\infty} A_n)$ and $\mu =\mu_{\tilde{h}}(q_{d}^n(z))$
for $z\in A_n$ and $n>0$. Then $\mu$ is a Beltrami coefficient
defined on the complex plane and symmetric about the unit circle.
Let $\varphi$ be a quasiconformal homeomorphism solving the
Beltrami equation $\varphi_{\overline{z}} =\mu (z)\varphi_{z}$.
Then $\phi |T$ is a homeomorphism of $T$. Define
$$
\tilde{f}_{\epsilon} =\varphi q_{d} \varphi^{-1}.
$$
From calculus,
$$
\mu_{\tilde{f}_{\epsilon}}(z) =\frac{\mu(q_{d}(z)) -\mu(z)}{1+
\mu(q_{d}(z))\overline{\mu(z)}}\Theta (z), \quad \hbox{where
$|\Theta (z)|=1$}.
$$
So $(\tilde{f}_{\epsilon})_{\overline{z}} =0$ for
$(1-\delta)^{1/2}< |z|< (1+\delta)^{1/2}$, that is,
$f_{\epsilon}=\tilde{f}_{\epsilon}|T$ is analytic. Because
$|\mu(z) -\mu_{\tilde{f}} (z)|<\epsilon$ for all $z\in {\mathbb
C}$, $f_{\epsilon}$ is $\epsilon$-approximate to $f$ in the metric
$d_{\mathcal T} (\cdot, \cdot)$.

The sequence of Markov partitions
$\{\varpi_{n,f}\}_{n=0}^{\infty}$ is just an image of the sequence
of Markov partions $\{ \varpi_{n,q_{d}}\}_{n=0}^{\infty}$ under
$\varphi|T$. Since $\varphi|T$ is quasisymmetric and $\{
\eta_{n,q_{d}}\}_{n=0}^{\infty}$ has bounded geometry, then
$\{\varpi_{n,f}\}_{n=0}^{\infty}$ has bounded geometry. A real
analytic circle endomorphism having bounded geometry is expanding
(refer to~\cite[Chapter 3]{Jiang2} or~\cite{Jiang9}). Thus $f_{\epsilon} \in
{\mathcal C}^{1+}$. This completes the proof.
\end{proof}

\section{Contractibility}

By definition, a topological space $X$ is contractible if there is a continuous map
$$
\Psi(x,t) : X\times [0,1]\to X
$$
such that $\Psi(x,0)=x$ and $\Psi(x,1)=x_{0}$ for all $x\in X$ where $x_{0}$ is a fixed point in $X$.
In this section, we prove the following theorem.

\vspace*{10pt}
\begin{theorem}~\label{contractibility}
The space ${\mathcal T}{\mathcal U}{\mathcal S}$ is contractible.
\end{theorem}

\begin{proof}
Let $x_{0}=[(q_{d}, id)]$ be the basepoint of ${\mathcal T}{\mathcal U}{\mathcal S}$. For any $x\in {\mathcal T}{\mathcal U}{\mathcal S}$, let $(f, h_{f})$ be a representation in $x$. From Theorem~\ref{uaatouac}, $(f, h_{f})$ has an extension $(\tilde{f}, \tilde{h}_{f})$ over an annulus neighborhood $\{z\;|\; 1/r<|z|<r\}$ for some $r>1$ such that $\tilde{f}$ is symmetric about $T$ and such that
$$
|\mu_{\tilde{f}^{-n}} (z)| \leq \eta(y), \quad z=x+yi
$$
for a vanishing function $\eta(y)$, where $\tilde{f}^{-n}$ means any inverse branch of $\tilde{f}^{n}$. Since
$$
\tilde{f}^{n} =\tilde{h}_{f}\circ q_{d}^{n}\circ \tilde{h}_{f}^{-1},
$$
we have that
$$
\mu_{\tilde{f}^{-n}} (z) =\theta(z) \frac{\mu_{\tilde{h}_{f}} (q_{d}^{-n}(z)) -\mu_{\tilde{h}_{f}}(z)}{1-\overline{\mu_{\tilde{h}_{f}} (q_{d}^{-n}(z))}\mu_{\tilde{h}_{f}}(z)}
$$
where $|\theta (z)|=1$ and $\|\mu_{\tilde{h}_{f}}\|_{\infty} \leq k<1$. (Again $\tilde{f}^{-n}$ means any inverse branch of $\tilde{f}^{n}$ and $q^{-n}$ means the corresponding inverse branch of $q^{n}$.)
This implies that
$$
|\mu_{\tilde{h}_{f}} (q_{d}^{-n}(z)) -\mu_{\tilde{h}_{f}}(z)| \leq \tilde{\eta}(y), \quad z=x+iy
$$
for a vanishing function $\tilde{\eta}(y)$.

Let $\mu=\mu_{\tilde{h}_{f}}$ and $h^{t\mu}$ be the unique solution of the Beltrami equation with Beltrami coefficient $t\mu$ for $0\leq t\leq 1$.
From the measurable Riemann mapping theorem, we know that $h^{t\mu}$ depends on $t$ and $\mu$ continuously. (Actually, if we consider $t$ as a complex parameter, then $h^{t\mu}$ depends on $t$ and $\mu$ holomorphically.)
Then $\tilde{f}_{t} = h^{t\mu}\circ q_{d}\circ (h^{t\mu})^{-1}$ is a continuous family of circle endomorphisms such that
$$
\mu_{\tilde{f}_{t}^{-n}} (z) =\theta_{t}(z) \frac{t(\mu_{\tilde{h}_{f}} (q_{d}^{-n}(z)) -\mu_{\tilde{h}_{f}}(z))}{1-t^{2}\overline{\mu_{\tilde{h}_{f}} (q_{d}^{-n}(z))}\mu_{\tilde{h}_{f}}(z)} \leq \hat{\eta} (y), \quad z=x+iy
$$
for a vanishing function $\hat{\eta} (y)$. Thus $(\tilde{f}_{t}, h^{t\mu})$ is a continuous family of uniformly asymptotically conformal maps.

Let $f_{t}=\tilde{f}_{t}|T$ and $h_{t}=h^{t\mu}|T$. Then $(f_{t}, h_{t})$ is a continuous family of marked uniformly symmetric circle endomorphism.
Thus $\tau_{t} = [(f_{t}, h_{t})]$ is a continuous path in ${\mathcal T}{\mathcal U}{\mathcal S}$ connecting $\tau$ and the basepoint $[q_{d}]$.

Define
$$
\Psi (\tau, t) = \tau_{t}: {\mathcal T}{\mathcal U}{\mathcal S}\times [0,1]\to {\mathcal T}{\mathcal U}{\mathcal S}
$$
is a continuous homotopy map moving every point to the basepoint. So ${\mathcal T}{\mathcal U}{\mathcal S}$ is contractible.
\end{proof}

\vspace*{5pt}
\begin{remark}~\label{ContractibilityHolder}
For any fixed $0<\alpha \leq 1$, let ${\mathcal C}^{1+\alpha}$ be the space of all $C^{1+\alpha}$ circle expanding endomorphisms.
Then ${\mathcal C}^{1+\alpha}$ is also a contractible space. It is a fact communicated to me by Anatole Katok. The proof can be as follows.
Let $f$ be a $C^{1+\alpha}$ circle expanding endomorphism. There is a $C^{1+\alpha}$ circle diffeomorphism $h$ such that
$g=h^{-1}\circ f\circ h$ preserves the Lebesgue measure (see, for example,~\cite{Jiang8}). The derivative $h'$ is the unique fixed point of the positive transfer operator (or called Ruelle's Perron-Fr\"obius operator)
$$
{\mathcal L} \phi (z) = \sum_{w\in f^{-1}(z)} \frac{\phi (w)}{f'(w)}
$$
from the space of all $\alpha$-H\"older continuous functions into itself. Since the fixed point can be obtained from the contracting fixed point theorem (see, for example,~\cite{Jiang10}), the fixed point depends on $f$ continuously. Let $H$ be the corresponding map for $h$ on the real line. Define
$$
H_{t}' (x) = tH'(x) +(1-t), \quad 0\leq t\leq 1.
$$
Then we have that $H_{t}'(x)$ is an $\alpha$-H\"older continuous positive periodic function of period $1$ and that
$H_{t}(x) =\int_{0}^{x} H_{t}'(\xi) d\xi$ is a $C^{1+\alpha}$ diffeomorphism of the real line satisfying $H_{t}(x+1) =H_{t}(x)+1$ and $H_{t}(0)=0$. Thus it a $C^{1+\alpha}$ circle diffeomorphism. Let $h_{t}$ fixing $1$ be the corresponding $C^{1+\alpha}$ diffeomorphism of $T$.
Define $f_{t} = h_{t}^{-1}\circ f\circ h_{t}$. Since $h_{t}$ is $C^{1+\alpha}$ circle diffeomorphism, $f_{t}$ is a $C^{1+\alpha}$ expanding circle endomorphism for any $0\leq t\leq 1$. Then $\gamma_{1}(t)=\{ f_{t}\}_{0\leq t\leq 1}$ is a continuous curve in ${\mathcal C}^{1+\alpha}$ connecting $f$ and $g$.

Let $G$ be the corresponding map for $g$ on the real line. We have that $G(x+1) =G(x)+d$ and $G'(x+1)=G'(x)$.
The fact that $g$ preserves the Lebesgue measure is equivalent to that
$$
\sum_{0\leq i\leq d-1} \frac{1}{G' (G^{-1} (x+i))} =1, \quad \forall x\in [0,1].
$$
This implies that $G'(x)>1$ for all $x\in [0,1]$. Define
$$
G_{t}'(x) = t G'(x) +(1-t) d, \quad x\in {\mathbb R}, \;\; 0\leq t\leq 1.
$$
For any $0\leq t\leq 1$, we have that $G_{t}'(x+1) =G_{t}'(x)$ for all $x\in {\mathbb R}$ and that $G_{t}'(x) >1$ for all $x\in [0,1]$.
Define $G_{t} (x) =\int_{0}^{x} G_{t}'(\xi)d\xi$. Then $G_{t} (x+1) =G_{t}(x)+d$ for all $x\in {\mathbb R}$. Thus $G_{t}$ is
a $C^{1+\alpha}$ expanding circle endomorphism. Let $g_{t}$ be the corresponding one on $T$. Then $\gamma_{2}(t) =\{ g_{t}\}_{0\leq t\leq 1}$ is a continuous curve in ${\mathcal C}^{1+\alpha}$ connecting $g$ and $q_{d}(z)=z^{d}$.

Define
$$
\Psi (f, t)= \left\{ \begin{array}{ll}
                      \gamma_{1} (2t), & 0\leq t\leq \frac{1}{2};\\
                      \gamma_{2} (2t-1), & \frac{1}{2}\leq t\leq 1.
                      \end{array}
             \right.
$$
Then $\Psi (f, t) :{\mathcal C}^{1+\alpha} \to {\mathcal C}^{1+\alpha}$ is a continuous homotopy map moving every point to the point $q_{d}(z)=z^{d}$. So ${\mathcal C}^{1+\alpha}$ is contractible.

Following the above argument, we have also that ${\mathcal C}^{1+}$ and ${\mathcal T}{\mathcal C}^{1+}$ are both contractible spaces.
\end{remark}

\section{Linear model and
dual derivative}

Suppose $f$ is a uniformly symmetric circle endomorphism of degree
$d\geq 2$. Let $D^{*}(f)(w^{*})$ be its dual derivative.
$$
\delta = D^{*}(f) (\cdots 000) =\lim_{n\to \infty}
\frac{F^{-n}(1)}{F^{-(n+1)}(1)}.
$$

Let $\Upsilon(x)=x+1$ be the translation by $1$ on the real line
${\mathbb R}$. For any $x\in {\mathbb R}=(-\infty, \infty)$, let
$$
[x] =\{ y\in {\mathbb R}\;\; |\;\; \Upsilon^{n}(x)=y, \;\;
\hbox{for some integer $n$}\}
$$
Then the unit circle can be thought as a topological space
$$
{\mathbb R}/\Upsilon =\{ [x]\}
$$
with linear Lebesgue metric introduced from ${\mathbb R}$. The
copies of the unit circle are $[k,k+1)$ for all integers $k$. The
circle endomorphism $f$ can be thought of as a map
$$
[x] \to [F(x) \pmod{1}].
$$

For each $n>0$, consider the homeomorphism
$$
\vartheta_{n}(x) =\frac{F^{-n}(x)}{F^{-n}(1)} : {\mathbb R}\to
{\mathbb R}.
$$
The $\{ \vartheta_{n}(x)\}$ is a sequence of uniformly symmetric
homeomorphisms of ${\mathbb R}$. We would like to show that $\{
\vartheta_{n}(x)\}$ is a convergent sequence and uniformly on any
compact set of ${\mathbb R}$ as follows.

For any $\epsilon
>0$, there is an $n_{0}>0$ such that $F^{-m}$ on $[0, F^{-n}(1)]$
is $(1+\epsilon)$-quasisymmetric for any $m>n\geq n_{0}$. Then
$$
H(y) = \frac{F^{-m+n} (F^{-n}(1)y)}{F^{-m}(1)} :[0,1]\to[0,1]
$$
is a $(1+\epsilon)$-quasisymmetric homeomorphism with $H(0)=0$ and
$H(1)=1$. From Lemma~\ref{qsdistortion}
$$
|H(y)-y| \leq \frac{\epsilon}{2}, \quad \forall y\in [0,1].
$$
Thus for $y=F^{-n}(x)/F^{-n}(1)$ for $x\in [0,1]$, we have
$$
|\vartheta_{n+m}(x) - \vartheta_{n}(x)|< \epsilon, \;\;\forall
m>n\geq n_{0}.
$$
This implies that $\{ \vartheta_{n}(x)\}_{n=0}^{\infty}$ is a
uniformly convergent Cauchy sequence on $[0,1]$. Thus it converges
uniformly to a function $\vartheta (x)$ on $[0,1]$. Similarly the
sequence of inverses $\{ \vartheta_{n}^{-1}(y) = F^{n}(F^{-n}(1)
y)\}_{n=0}^{\infty}$ is also a uniformly convergent Cauchy
sequence, it converges uniformly to a function which is the
inverse of $\vartheta (x)$. So $\vartheta (x)$ is a homeomorphism.
A direct calculation implies that $\vartheta (x)$ is symmetric on
$[0,1]$. For any fixed $k>0$, $F^{-k}$ maps $[0, d^{k}]$ onto
$[0,1]$. Using the relation
$$
\vartheta_{n+k} (x) = \frac{F^{-n}(1)}{F^{-n-k}(1)} \vartheta_{n}
(F^{-k}(x)), \; x\in [0, d^{k}],
$$
we get that $\{ \vartheta_{n}(x)\}_{n=0}^{\infty}$ is a uniformly
convergent Cauchy sequence on $[0,d^{k}]$ and converges to a
uniformly symmetric homeomorphism $\vartheta (x)$, since $\{
F^{-k}\}_{k=0}^{\infty}$ is uniformly symmetric. We conclude that
$\{ \vartheta_{n}(x)\}$ is a convergent sequence and converges
uniformly on any compact set of ${\mathbb R}^{+}=[0, \infty)$; and
the limit function is a uniformly symmetric homeomorphism
$\vartheta$ of ${\mathbb R}^{+}$. Moreover, $\vartheta (x)$
conjugates $F$ to a linear map $x\to \delta x$ on ${\mathbb
R}^{+}$, that is,
$$
\vartheta \circ F \circ \vartheta^{-1} (x) = \delta x, \;\;
\forall x\geq 0.
$$
Similar, we can prove the above on ${\mathbb R}^{-1}=(-\infty,
0]$.

The linear model of $f$ is the conjugate function of the linear
equivalence $\Upsilon$ by $\vartheta$, that is,
\begin{equation}~\label{linearmodel}
L (x) = \vartheta \circ \Upsilon \circ \vartheta^{-1} (x)
=\vartheta (\vartheta^{-1}(x) +1).
\end{equation}
Since
$$
F\circ \Upsilon (x) = \Upsilon^{d} \circ F,
$$
We have
\begin{equation}~\label{multiequation}
L(0)=1, \quad \hbox{and}\quad L(x) =\delta^{-1} L^{d} (\delta x),
\quad \forall x\in {\mathbb R}.
\end{equation}

Now we have a new point of view for the unit circle and the circle
endomorphism $f$: For any $x\in [0, 1)$, let
$$
[x]_{L} =\{ y\in {\mathbb R}\;\; |\;\; L^{n}(x)=y, \;\; \hbox{for
some integer $n$}\}.
$$
The unit circle can be thought as a topological space
$$
{\mathbb R}/L=\{ [x]_{L}\}
$$
with the metric introduced from $L$. Copies of the unit circle now
are all intervals $[\vartheta(k), \vartheta(k+1))= [L^{k}(0),
L^{k+1}(0))$ for all integers $k$. The circle endomorphism $f$ can
be thought of as a map
$$
[x]_{L} \to [\delta x \pmod{L}].
$$

\vspace*{10pt}
\begin{theorem}~\label{linearmodelthm} Suppose $f$ and $g$ are
two uniformly symmetric circle endomorphisms of the same degree
$d\geq 2$. Let $h$ be the conjugacy between $f$ and $g$, that is,
$$
h\circ f=g\circ h.
$$
Then $h$ is a symmetric homeomorphism if and only if the linear
models of $f$ and $g$ are the same and $D^{*}(f)(\cdots 000)=
D^{*}(g)(\cdots 000)$.
\end{theorem}

\begin{proof}
Suppose $h$ is symmetric. Applying Lemma~\ref{qsdistortion}, we
have that
$$
D^{*}(f)(\cdots 000)= D^{*}(g)(\cdots 000).
$$

Let $F$, $G$, and $H$ be the lifts of $f$, $g$, and $h$. Then
$$
F^{-n} (x) = H (G^{-n} (H^{-1}(x))).
$$
Since $H(1)=1$, we get
$$
\frac{F^{-n}(x)}{F^{-n}(1)} = \frac{H\circ G^{-n}\circ H^{-1}(x)}{
H\circ G^{-n}(1)}.
$$
Since $H$ is symmetric, we get, by using Lemma~\ref{qsdistortion},
$$
\vartheta_{f} (x) = \vartheta_{g}\circ H^{-1}(x).
$$
So
$$
L_{f} (x) =\vartheta_{f} (\vartheta_{f}^{-1}(x)+1) =
\vartheta_{g}( H^{-1}(H (\vartheta_{g}^{-1}(x) +1)))
$$
$$
= \vartheta_{g}( H^{-1}(H(\vartheta_{g}^{-1}(x))) +1) =
\vartheta_{g}(\vartheta_{g}^{-1}(x) +1) =L_{g} (x).
$$

Conversely, suppose $L_f=L_g$ and
$$
\delta=D^{*}(f)(\cdots 000)= D^{*}(g)(\cdots 000).
$$
Then
$$
L_{f} (x) =\vartheta_{f} (\vartheta_{f} (x)+1) = \vartheta_{g}
(\vartheta_{g} (x)+1)=L_{g}(x)
$$
Let
$$
H(x) =\vartheta_{g}^{-1}\circ \vartheta_{f} (x).
$$
We have $H(x+1)=H(x)+1$. So $H$ is a symmetric circle
homeomorphism. Since
$$
F(x) =\vartheta_{f}^{-1} (\delta \vartheta_{f} (x)) \quad
\hbox{and}\quad  G(x) =\vartheta_{g}^{-1} (\delta \vartheta_{g}
(x)),
$$
we get that
$$
F(x) =H^{-1} \circ G\circ H (x).
$$
So $f$ and $g$ are symmetrically conjugate. We proved the theorem.
\end{proof}

Suppose $\{ \eta_{n}\}_{n=0}^{\infty}$ is the sequence of nested
partitions on $[0,1]$ for $f$. For any $w^{*}\in \Sigma^{*}$, let
$w_{n}^{*}=j_{n-1}\cdots j_{1}j_{0}$ and $v_{n-1}^{*}
=j_{n-1}\cdots j_{1}$. Since $\vartheta(x)$ is symmetric on
$[0,1]$ with $\vartheta(0)=0$ and $\vartheta(1)=1$, from
Lemma~\ref{qsdistortion},
$$
D^{*}(f)(w^{*}) =\lim_{n\to \infty}
\frac{|I_{v_{n-1}}|}{|I_{w_{n}}|} =\lim_{n\to\infty}
\frac{|\vartheta(I_{v_{n-1}})|}{|\vartheta(I_{w_{n}})|}.
$$

Consider non-negative integers
$$
k=j_{0}+j_{1}d+\cdots+j_{n-1}d^{n-1} \quad \hbox{and}\quad
l=j_{1}+j_{2}d+\cdots +j_{n-1}d^{n-2}.
$$
Then $k=dl+j_{0}$ and
$$
I_{w_{n}} = F^{-n} ([k, k+1])\quad \hbox{and}\quad I_{v_{n-1}} =
F^{-(n-1)} ([l, l+1]).
$$
Since $\vartheta (F^{-n}(x)) =\delta^{-n}\vartheta (x)$ and since
$\delta \vartheta (l) = \vartheta (F(l)) =\vartheta (dl)$,
$$
\frac{|\vartheta(I_{v_{n-1}})|}{|\vartheta(I_{w_{n}})|}
=\frac{\delta |\vartheta ([l,l+1])|}{|\vartheta ([k,k+1])|} =
\frac{|\vartheta ([k-j_{0},k+d -j_{0}])|}{|\vartheta ([k,k+1])|} .
$$
This implies that
$$
\frac{|\vartheta ([k-j_{0},k+d -j_{0}])|}{|\vartheta ([k,k+1])|} =
D^{*}(f) (\cdots 000j_{n-1}\cdots j_{1}j_{0}).
$$
Since $L^{k}(0)=\vartheta (k)$, the above equality says that all
values of $L^{k}(0)$ are uniquely determined by
$$
\{ D^{*}(f) (\cdots 000 j_{n-1}\cdots j_{1}j_{0}) \;\;|\;\;
j_{k}=0, 1, \cdots (d-1), k=0, 1, \cdots, n-1\}.
$$
Using Equation~(\ref{multiequation}), we get that the linear model
$L$ is uniquely determined by the dual derivative $D^{*}(f)
(w^{*})$. Thus we have a corollary of Theorem~\ref{linearmodel}.

\vspace*{10pt}
\begin{corollary}~\label{same1}
Suppose $f$ and $g$ are two uniformly symmetric circle
endomorphisms of the same degree $d\geq 2$. Let $h$ be the conjugacy
between $f$ and $g$ such that $h(1)=1$, that is,
$$
h\circ f=g\circ h.
$$
Then $h$ is a symmetric homeomorphism if and only if the dual
derivatives of $f$ and $g$ are the same, that is,
$$
D^{*}(f)(w^{*})= D^{*}(g)(w^{*}), \quad \forall w^{*}\in
\Sigma^{*}.
$$
\end{corollary}

Following the above theorem we set a one-to-one correspondence
between the Teichm\"uller space ${\mathcal T}{\mathcal U}{\mathcal
S}$ and the space of all continuous dual derivatives:
$$
\Pi =[(f, h_{f})] \to D^{*}(f)(w^{*}).
$$
Therefore,
\begin{equation}~\label{qcrepresentation}
{\mathcal T}{\mathcal U}{\mathcal S} =\{ D^{*}(f)(w^{*})\;\; |\;\;
f\in {\mathcal U}{\mathcal S}\}
\end{equation}
equipped with the Teichm\"uller metric $d_{\mathcal T}(\cdot,
\cdot)$. This is a complete space.

\section{Characterization of dual derivatives}

Assume $d=2$ in this section. Suppose $f\in {\mathcal U}{\mathcal
S}$. Let $D^{*}(f)(w^{*})$ be its dual derivative. Then it is easy
to see the following summation condition:
\begin{equation}~\label{summationcondition}
\frac{1}{D^{*}(f) (w^{*}0)}+\frac{1}{D^{*}(f) (w^{*}1)} =1, \quad
\forall w^{*} \in \Sigma^{*}.
\end{equation}
Another non-trivial condition is the following compatibility
condition:
\begin{equation}~\label{compatibilitycondition}
\prod_{n=0}^{\infty} \frac{D^{*}(f)(w^{*} 0\overbrace{1\ldots
1}^{n})}{D^{*}(f) (w^{*} 1\underbrace{0\ldots 0}_{n})} =const,
\quad \forall w^{*}\in \Sigma^{*}.
\end{equation}
The convergence is uniform. And moreover, if $f\in {\mathcal
C}^{1+}$, then the convergence is exponential. We give a proof of
this non-trivial condition as follows.

First let us set up a relation between the dual derivative
$D^{*}(f)(w^{*})$ and the linear model $L$. Suppose $\vartheta$ is
the symmetric homeomorphism such that
$$
L(x) =\vartheta \Upsilon \vartheta^{-1}(x)\quad \hbox{and}\quad
\delta x = \vartheta F \vartheta^{-1}(x).
$$
Then $L^{k}([0,1])= [\vartheta (k),\vartheta (k+1)]$ for every
integer $k$.

For any $w^{*}=\cdots j_{n-1}\cdots j_{0}\in \Sigma^{*}$, let
$w_{n}^{*}=j_{n-1}\cdots j_{0}$ and define integers
$$
k=k(w^{*}_{n}) =\sum_{q=0}^{n-1} j_q 2^q\quad \hbox{and}\quad
l=k(\sigma^{*}(w^{*}_{n}))=\sum_{q=0}^{n-2} j_{q+1} 2^{q}.
$$
Then $k= 2l+j_{0}$. By the definitions,
$$
D^{*}(f)(w^{*}) =
\lim_{n\to\infty}\frac{|I_{w^{*}_n}|}{|I_{\sigma^*(w^{*}_n)}|} =
\lim_{n\to\infty}\frac{|\vartheta (I_{w^{*}_n})|}{|\vartheta
(I_{\sigma^*(w^{*}_n)})|}.
$$
Note that
$$
I_{w^{*}_n} = G_{j_{n-1}} \circ \cdots \circ G_{j_0}(I) =F^{-1}
\circ \Upsilon^{j_{n-1}} \circ \cdots \circ F^{-1} \circ
\Upsilon^{j_0}([0,1])
$$
$$
=F^{-n} (\Upsilon^{j_0 + 2j_1 + \cdots
+ 2^{n-1}j_{n-1}}([0,1])) = F^{-n}([k, k+1]),
$$
and, similarly,
$$
I_{\sigma^*(w^{*}_n)} = F^{-n+1}([l,l+1]).
$$
Therefore, since $\vartheta (F^{-n}(x)) = \delta^{-n} \vartheta
(x)$ and since
$$
\delta \vartheta (l) =\vartheta (F(l))=\vartheta (F( \Upsilon^l
(0)))=\vartheta \Upsilon^{2l}(F(0)) =\vartheta (2l),
$$
we have
$$
\frac{|\vartheta (I_{w^{*}_{n}})|}{|\vartheta
(I_{\sigma^{*}(w^{*}_{n})})|} = \frac{|\vartheta
(F^{-n}([k,k+1]))|}{|\vartheta (F^{-n+1}([l, l+1]))|} =
\frac{|\vartheta ([k, k+1])|}{\delta |\vartheta ([l, l+1])|}
=\frac{|\vartheta ([k, k+1])|}{|\vartheta([k-j_{0}, k-j_{0}+2])|}.
$$

Let $I=[0,1]$. Since $\delta I = I \cup L(I)$ and $\vartheta
(k)=L^{k}(0)$, we can rewrite
$$
\frac{|\vartheta([k, k+1])|}{|\vartheta([k-j_{0}, k-j_{0}+2])|}=
\frac{|L^{k}(I)|}{|L^{k-j_{0}}(I+L(I))|} =\Big(
1+\frac{|L^k(I)|}{|L^{(-1)^{j_{0}}} (L^k(I))|}\Big)^{-1}.
$$
Thus we get
$$
D^{*}(f)(w^{*}) = \lim_{n\to
\infty}\frac{|L^{k}(I)|}{|L^{k-j_{0}}(I+L(I))|} =\lim_{n\to
\infty}\Big( 1+\frac{|L^{(-1)^{j_{0}}}
(L^k(I))|}{|L^k(I)|}\Big)^{-1}.
$$

For any $w^{*}=\cdots w_{n}^{*}=\cdots j_{n-1}\cdots j_{1}j_{0}
\in \Sigma^{*}$, define
$$
sol(w^{*})= \lim_{n\to \infty} \frac{|L^{k}(I)|}{|L^{k-1}(I)|}.
$$
(This is similar to a solenoid function defined
in~\cite{SullivanSol,PintoSullivan}). Then, by considering
$w^{*}=v^{*}j_{0}$, we have
\begin{equation}~\label{scalinglinear0}
D^{*}(f)(v^{*}0) = \lim_{n\to \infty}
\left(1+\frac{|L^{k+1}(I)|}{|L^k(I)|}\right)^{-1}= (1 +
sol(v^{*}1))^{-1}
\end{equation}
and
\begin{equation}~\label{scalinglinear1}
D^{*}(f)(v^{*}1) =\lim_{n\to \infty}
\left(1+\frac{|L^{k-1}(I)|}{|L^{k}(I)|}\right)^{-1}= \Big( 1 +
\frac{1}{sol(v^{*}1)}\Big)^{-1}.
\end{equation}
These two equations combining with the summation
condition~(\ref{summationcondition}) imply that
\begin{equation}~\label{scalingsol}
sol(v^{*}1)=\frac{D^{*}(f)(v^{*}0)}{D^{*}(f)(v^{*}1)}.
\end{equation}
Since
$$
\frac{\delta L^k(I)}{\delta L^{k-1}(I)} =\frac{L^{2k}(\delta
I)}{L^{2k-2}(\delta I)}=
\frac{L^{2k}(I)+L^{2k+1}(I)}{L^{2k-2}(I)+L^{2k-1}(I)},
$$
we have the following formula:
$$
\frac{sol(w^{*}0)}{sol(w^{*})}=\lim_{n\to \infty}
\frac{\frac{L^{2k-2}(I)+L^{2k-1}(I)}{L^{2k-1}(I)}}{\frac{L^{2k}(I)+L^{2k+1}(I)}{L^{2k}(I)}}
 =
\frac{1+\frac{L^{2k-2}(I)}{L^{2k-1}(I)}}{1+sol(w^{*}1)}.
$$
Equations~(\ref{scalinglinear0}),~(\ref{scalinglinear1}),
and~(\ref{scalingsol}) imply that
$$
\frac{D^{*}(f)(w^{*}01)}{D^{*}(f)(w^{*}10)} =
\frac{1+[sol(w^{*}01)]^{-1}}{1+[sol(w^{*}11)]}=\frac{sol(w^{*}10)}{sol(w^{*}1)}.
$$
(Note that for $w^{*}1$, $2k-1$ corresponds to $w^{*}01$.)
Similarly,
$$
\frac{D^{*}(f)(w^{*}011)}{D^{*}(f)(w^{*}100)} =
\frac{1+[sol(w^{*}011)]^{-1}}{1+[sol(w^{*}101)]}=\frac{sol(w^{*}100)}{sol(w^{*}10)}.
$$
Proceeding by induction, we conclude
\begin{equation}~\label{solseq}
\hbox{sol} (w^{*} 1\underbrace{0\cdots 0}_{n-1}) =
\prod_{i=0}^{n-1} \frac{D^{*}(f) (w^{*} 0\overbrace{1\cdots
1}^{i})}{D^{*}(f)(w^{*} 1\underbrace{0\cdots 0}_{i})}.
\end{equation}
From the quasiymmetric distortion property
(Lemma~\ref{qsdistortion}), ~(\ref{solseq}) converges uniformly to
$sol(\cdots 0 \cdots 0)$ for all $w^{*}\in\Sigma^*$. If $f\in
{\mathcal C}^{1+}$, from the H\"older distortion property
~(\ref{smoothdistortion}), ~(\ref{solseq}) converges exponentially
to $sol(\cdots 0 \cdots 0)$ for all $w^{*}\in\Sigma^*$. We proved
the compatibility condition~(\ref{compatibilitycondition}).

In the paper~\cite{CuiJiangQuas,CuiGardinerJiang,Jiang7}, we
further proved that the conditions~(\ref{summationcondition})
and~(\ref{compatibilitycondition}) are also sufficient as follows.

\vspace*{10pt}
\begin{theorem}~\label{sfr} Let $\Psi (w^{*})$ be a positive
continuous function on $\Sigma^{*}$. Then it is a dual derivative
of an $f\in {\mathcal U}{\mathcal S}$ if and only if it satisfies
the conditions~(\ref{summationcondition})
and~(\ref{compatibilitycondition}). Furthermore, if $\Psi (w^{*})$
is a H\"older continuous function on $\Sigma^{*}$, then it is a
dual derivative of an $f\in {\mathcal C}^{1+}$ if and only if it
satisfies the conditions~(\ref{summationcondition})
and~(\ref{compatibilitycondition}).
\end{theorem}

The proof of this theorem is technical so we will not include in
this paper. The reader who is interested in this topic can refer
to~~\cite{CuiJiangQuas,CuiGardinerJiang,Jiang7}.

\vspace*{10pt}
\begin{remark}
A similar result to the second half of Theorem~\ref{sfr} was
studied by Pinto and Sullivan in~\cite{PintoSullivan} for a
solenoid function. They introduced a matching condition for a
function on $\Sigma^{*}$ and proved that a H\"older continuous
function on $\Sigma^{*}$ is a solenoid function of an $f\in
{\mathcal C}^{1+}$ if and only if it satisfies the matching
condition. Furthermore, using some relation between the solenoid
function and the linear model for an $f\in{\mathcal U}{\mathcal
S}$, Cui proved in~\cite{Cui} that two uniformly symmetric circle
endomorphisms are symmetric conjugate if and only if they have the
same eigenvalues at the corresponding periodic points.
We would like to note that~\cite{Jiang7} contains a much easy and
straightforward understanding to these results.
\end{remark}

From Theorem~\ref{sfr}, we have the following representations for
the Teichm\"uller spaces when $d=2$:
$$
{\mathcal T}{\mathcal C}^{1+} =\{ \Psi^{*}(w^{*}) \;|\;
\Psi^{*}(w^{*}) \hbox{ is H\"older continuous and
satisfies~(\ref{summationcondition})
and~(\ref{compatibilitycondition})}\}
$$
and
$$
{\mathcal T}{\mathcal U}{\mathcal S} =\{ \Psi^{*}(w^{*}) \;|\;
\Psi^{*}(w^{*}) \hbox{ is continuous and
satisfies~(\ref{summationcondition})
and~(\ref{compatibilitycondition})}\}
$$

\section{The maximum distance and the Teichm\"uller distance.}

We have introduce a Teichm\"uller metric $d_{\mathcal T}(\cdot,\cdot)$
on ${\mathcal T}{\mathcal U}{\mathcal S}$ which is a complete metric. We also showed that
${\mathcal T}{\mathcal U}{\mathcal S}$ can be represented by continuous functions $\Psi^{*}(w^{*})$ on $\Sigma^{*}$.
For a function $\Psi^{*}(w^{*})$, we can define the maximum norm
$$
||\Psi^{*}|| =\max_{w^{*}\in \Sigma^{*}}|\Psi^{*}(w^{*})|.
$$
This gives a distance
$$
d_{max} (\Psi^{*}, \tilde{\Psi}^{*}) =|| \Psi^{*}-\tilde{\Psi}^{*}||
$$
on ${\mathcal T}{\mathcal U}{\mathcal S}$. We call it the maximum metric.
Since ${\mathcal T}{\mathcal U}{\mathcal S}$ contains all positive continuous functions satisfying (\ref{summationcondition})
and~(\ref{compatibilitycondition}) taking values in $(1, \infty)$. We have following theorems.

\vspace*{10pt}
\begin{theorem}~\label{unifcont}
The identity map
$$
id_{TM}: ({\mathcal T}{\mathcal U}{\mathcal S}, d_{\mathcal T}(\cdot,\cdot)) \to ({\mathcal T}{\mathcal U}{\mathcal S}, d_{\max}(\cdot,\cdot))
$$
is uniformly continuous.
\end{theorem}

\vspace*{10pt}
\begin{theorem}~\label{contonly}
The identity map
$$
id_{MT}: ({\mathcal T}{\mathcal U}{\mathcal S}, d_{\max}(\cdot,\cdot)) \to ({\mathcal T}{\mathcal U}{\mathcal S}, d_{\mathcal T}(\cdot,\cdot))
$$
is continuous.
\end{theorem}

\vspace*{10pt}
\begin{corollary}~\label{sametop}
The topologies induced by the Techm\"uller metric $d_{\mathcal T}(\cdot,\cdot)$ and by the maximum metric $d_{\max}(\cdot,\cdot)$ are the same.
\end{corollary}

\begin{proof}[Proof of Theorem~\ref{unifcont}]
Suppose $\Pi, \Pi'\in {\mathcal T}{\mathcal U}{\mathcal S}$.
Suppose $K=\exp(2d_{{\mathcal T}}(\Pi,\Pi'))\geq 1$.
For any $\epsilon >0$, we have two marked circle endomorphisms $(f, h_{f})\in \Pi$ and $(g, h_{g})\in \Pi'$ such that
$$
h=h_{g}\circ h_{f}^{-1}: T\to T
$$
can be extended to a $K(1+\epsilon)$-quasiconformal map $\tilde{h}$ defined on an annulus $\{z\in {\mathbb C}\;|\; \frac{1}{r}< |z|<r\}$ for some $r>1$. This implies that there is a $\delta>0$ and $M =M (K, \epsilon)>0$ such that $M\to 1$ as $K\to 1$ and $\epsilon \to 0$ and such that
$H$, which is a lift of $h$, is a $(\delta, M)$-quasisymmetric, that is,
$$
M^{-1} \leq \frac{|H(y)-H(\frac{x+y}{2})|}{|H(\frac{x+y}{2})-H(x)|} \leq M
$$
for any $x,y$ with $|x-y|\leq \delta$. Note that $h\circ f=g\circ h$ and $H\circ F=G\circ H \pmod{1}$.

For any point $w^{*}= \cdots w_{n}^{*}\in \Sigma_{A}^{*}$, we have
that
$$
I_{w^{*}_{n}, f}\in \eta_{n,f} \quad \hbox{and}\quad
I_{\sigma^{*}(w^{*}_{n}), f}\in \eta_{n-1,f}
$$
and
$$
I_{w_{n}^{*}, g}=H(I_{w^{*}_{n}, g})\in \eta_{n,g}  \quad
\hbox{and}\quad I_{\sigma^{*}(w^{*}_{n}),
g}=H(I_{\sigma^{*}(w^{*}_{n}), g})\in \eta_{n-1,g}.
$$
Note that
$$
I_{w^{*}_{n}, f}\subset I_{\sigma^{*}(w^{*}_{n}), f}\quad
\hbox{and}\quad I_{w_{n}^{*}, g}\subset I_{\sigma^{*}(w^{*}_{n}),
g}.
$$

Let $n_{0}>0$ be an integer such that
$$
|I_{\sigma^{*}(w^{*}_{n}), f}|\leq \delta
$$
for all $n\geq n_{0}$. Then $H|I_{\sigma^{*}(w^{*}_{n}), f}$ is a
$M$-quasisymmetric homeomorphism.

By rescaling $I_{\sigma^{*}(w^{*}_{n}), f}$ and
$I_{\sigma^{*}(w^{*}_{n}), g}$ into the unit interval $[0,1]$ by
linear maps, we can think $H|I_{\sigma^{*}(w^{*}_{n}), f}$ is a
$M$-quasisymmetric homeomorphism of $[0,1]$ and fixes $0$ and
$1$. Then Lemma~\ref{qsdistortion} implies that
$$
\Big|\frac{1}{D^{*}(g)(w_{n}^{*})} - \frac{1}{D^{*}(f)(w_{n}^{*})}\Big| =\Big|
\frac{|H(I_{w^{*}_{n}, f})|}{|H(I_{\sigma^{*}(w^{*}_{n}), f})|}
-\frac{|I_{w^{*}_{n}, f}|}{|I_{\sigma^{*}(w^{*}_{n}), f}|}\Big|
\leq \zeta(M) .
$$
This implies that
$$
|D^{*}(g)(w^{*}) - D^{*}(f)(w^{*})| \leq \zeta(M).
$$
Therefore,
$$
d_{max} (D^{*}(g), D^{*}(f)) \leq \zeta (M (d_{T}(\Pi, \Pi')(1+\epsilon)).
$$
This proved the theorem.
\end{proof}

\begin{proof}[Proof of Theorem~\ref{contonly}]
Suppose $id_{MT}$ is not continuous.
That is, we have a real number $\epsilon>0$
and a point $\Psi^{*}$ and a sequence of points $\{ \Psi_{m}\}_{m=1}^{\infty}$
in the Teichm\"uller space ${\mathcal T}{\mathcal U}{\mathcal S}$ such that
$$
d_{\max} (\Psi^{*}_{m}, \Psi^{*})=\|\Psi^{*}_{m}-\Psi^{*}\| \to 0\quad \hbox{as}\quad m\to \infty
$$
but
$$
d_{{\mathcal T}} (\Pi^{*}_{m}, \Pi^{*})\geq \epsilon, \quad \forall\; m,
$$
where $\Pi$ and $\Pi_{m}$ are the corresponding points for $\Psi^{*}$ and $\Psi^{*}_{m}$.
Let $(f, h_{f})\in \Pi$ be a fixed representation
and $(f_{m}, h_{f_{m}})\in \Pi_{m}$ for each $m$ be
a representation.

Let $F$ and $F_{m}$ be the corresponding circle endomorphisms of the real line.
Let $\{\eta_{n}\}_{n=0}^{\infty}$ and $\{\eta_{m,n}\}_{n=0}^{\infty}$ be the corresponding
sequences of nested Markov partitions. Since $\|D^{*}(f_{m})-D^{*}(f)\| \to 0$ as $m\to \infty$,
we have a constant $a=a(D^{*}(f))>0$ such that $D^{*}(f_{m})(w^{*}) \geq a$ for sufficient large $m$
and all $w^{*}\in \Sigma^{*}$. Let us assume this true for all $m$.
Since $\Sigma^{*}$ is a compact set, we have that there is another
constant $b=b(a)>0$ such that $D^{*}(f_{m})(w_{n}^{*}) \geq b$ (see (\ref{prederivative}))
for all $m$ and all $n$. This says that the collection of the sequences
$\{\eta_{m,n}\}_{n=0}^{\infty}$ of nested Markov partitions
has uniformly bounded geometry. From a method in~\cite{Jiang6}, which is also shown in the proof of Corollary~\ref{quasisymmetriccojugacy}
and which gives a calculation of quasisymmetric dilatation
from bounded geometry, we have a constant $M>0$ such that the quasisymmetric dilatations of all $H_{m}$ are less than or equal to $M$, that is,
Since
$$
QS_{m} =\sup_{x\in {\mathbb R}, t>0} \frac{|H_{m}(x+t)-H_{m}(x)|}{|H_{m}(x)-H_{m}(x-t)|} \leq M.
$$
for every $m>0$. This says that by modulo all M\"obius transformations, $\{H_{m}\}_{m}$ is a compact subset in the universal Teich\"uller space.
Thus $\{[(f_{m}, h_{m})]\}_{m=1}^{\infty}$ is a compact set in ${\mathcal T}{\mathcal U}{\mathcal S}$. So
it has a convergent subsequence. Let us assume that $\{[(f_{m}, h_{m})]\}_{m=1}^{\infty}$
itself is convergent and converges to $[(f_{0}, h_{0})]$ as $m\to \infty$. Let $H_{0}$ be a lift of $h_{0}$. Then $H_{m}$ tends to $H_{0}$ modulo all M\"obius transformations under the maximal norm on the real line. Assume $H_{m}\to H_{0}$ uniformly on the real line as $m\to \infty$.

For any $w_{n}^{*}$, let $I_{w_{n}^{*}, f_{m}}\in \eta_{m,n}$
and $I_{w_{m}^{*}, f_{0}}\in \eta_{n,f_{0}}$. We have that $|I_{w_{n}^{*}, f_{m}}|\to |I_{w_{m}^{*}, f_{0}}|$
as $m\to \infty$ for each fixed $n$ and $w_{n}^{*}$.

Since the sequences $\{\eta_{m,n}\}_{n=0}^{\infty}$ of nested Markov partitions have
uniformly bounded geometry, this again says that there are constants $C=C(S)>0$ and $0<\mu=\mu (D^{*}(f))<1$
such that $\nu_{n,m}\leq C\mu^{n}$ for all $n$ and $m$,
where
$$
\nu_{n,m} =\max_{I\in \eta_{n,m}} |I|.
$$
This implies that $D^{*}(f_{m})(w_{n}^{*}) \to D^{*}(f_{m})(w^{*})$ and $D^{*}(f_{0})(w^{*}_{n}) \to D^{*}(f_{0})(w^{*})$
as $n\to \infty$ uniformly on $m\geq 1$ and $w^{*}\in \Sigma^{*}$. Thus we can change double
limits for each $w^{*}\in \Sigma^{*}$,
$$
D^{*}(f) (w^{*}) =\lim_{m\to \infty} D^{*}(f_{m})(w^{*}) = \lim_{m\to \infty} \lim_{n\to \infty} D^{*}(f_{m})(w_{n}^{*})
$$
$$
=\lim_{n\to \infty}\lim_{m\to \infty} D^{*}(f_{m})(w_{n}^{*}) =\lim_{n\to \infty} D^{*}(f_{0})(w_{n}^{*}) = D^{*}(f_{0})(w^{*}).
$$
From Corollary~\ref{same1}, this implies that $[(f,h)]=[(f_{0},h_{0})]=\Pi$.
This is a contradiction to our original assumption. The contradiction says that
$id_{MT}$ is continuous at each point $\Psi^{*}$.
\end{proof}

\section{$\sigma$-invariant measures and dual $\sigma^{*}$-invariant measures}

Consider the symbolic dynamical system $(\Sigma, \sigma)$ and a
positive H\"older continuous function $\psi (w)$. The standard
Gibbs theory (refer to~\cite{Bowen,Ruelle1,Ruelle2,
Sinai1,Sinai2}) implies that there is a number $P=P(\log \psi)$
called the pressure and a $\sigma$-invariant probability measure
$\mu= \mu_{\psi}$ such that
$$
C^{-1} \leq \frac{\mu ([w_{n}])}{\exp (-Pn + \sum_{i=0}^{n-1} \log
\psi (\sigma^{i} (w)))}\leq C
$$
for any left cylinder $[w_{n}]$ and any point $w=w_{n}\cdots$,
where $C$ is a fixed constant. Here, $\mu$ is a $\sigma$-invariant
probability measure means that
$$
\mu (\sigma^{-1} (A)) =\mu (A)
$$
for all Borel sets of $\Sigma$. A $\sigma$-invariant probability
measure satisfying the above inequalities is called the Gibbs
measure with respect to the given potential $\log \psi$.

Two positive H\"older continuous functions $\psi_{1}$ and
$\psi_{2}$ are said to be cohomologous equivalent if there is a
continuous function $u=u(w)$ on $\Sigma$ such that
$$
\log \psi_{1}(w) -\log \psi_{2} (w) = u (\sigma (w)) -u(w).
$$
If two functions are cohomologous to each other, they have the
same Gibbs measure. Therefore, the Gibbs measure can be thought of
as a representation of a cohomologous class.

The Gibbs measure $\mu$ for a given potential $\log \phi$ is also
an equilibrium state for this potential as follows. Consider the
measure-theoretical entropy $h_{\mu}(\sigma)$. Since the Borel
$\sigma$-algebra of $\Sigma$ is generated by all left cylinders,
then $h_{\mu} (\sigma) $ can be calculated as
$$
h_{\mu} (\sigma)  =\lim_{n\to \infty} \frac{1}{n} \sum_{w_{n}}
\Big( -\mu ([w_{n}]) \log \mu ([w_{n}])\Big)
$$
$$
= \lim_{n\to \infty} \sum_{w_{n}} \Big( -\mu ([w_{n}]) \log
\frac{\mu( [w_{n}])}{\mu(\sigma([w_{n}]))}\Big),
$$
where $w_{n}$ runs over all words $w_{n}=i_{0}\cdots i_{n-1}$ of
$\{ 0, 1,\cdots, d-1\}$ of length $n$. Then $\mu$ is an
equilibrium state in the sense that
$$
P(\log \psi) = h_{\mu} (\sigma) + \int_{\Sigma} \log \psi (w) d\mu
(w) =\sup \{ h_{\nu}(\sigma) +\int_{\Sigma} \log \psi (w)
d\nu(w)\},
$$
where $\nu$ runs over all $\sigma$-invariant probability measures.
The measure $\mu$ is unique in this case.

There is a natural way to transfer a $\sigma$-invariant
probability measure $\mu$ (not necessarily a Gibbs measure) to a
$\sigma^{*}$-invariant probability measure $\mu^{*}$ as follows:

Given any right cylinder $[w_{n}^{*}]$ in $\Sigma^{*}$ where
$w^{*}_{n}=j_{n-1}\cdots j_{0}$, then
$$
w_{n}=i_{0}\cdots i_{n-1}=j_{n-1}\cdots j_{0}=w_{n}^{*},
$$
which defines a left cylinder
$$
[w_{n}] =\{ w'=i_{0}'\cdots i_{n-1}'i_{n}'\cdots \; |\;
i_{0}'=i_{0}, \cdots, i_{n-1}'=i_{n-1}\}.
$$
Define
$$
\mu^{*} ([w_{n}^{*}])= \mu([w_{n}]).
$$
Then
$$
\mu^{*} ([w_{n}^{*}])= \mu([w_{n}])= \mu(\sigma^{-1}([w_{n}]))
$$
$$
=\mu (\cup_{i=0}^{d-1} [iw_{n}]) = \sum_{i=0}^{d-1} \mu ([iw_{n}])
=\sum_{j=0}^{d-1} \mu^{*}([(jw_{n})^{*}]).
$$
This implies that $\mu^{*}$ satisfies the finite additive law for
all cylinders, i.e., if $A_{1}$, $\cdots$, $A_{k}$ are finitely
many pairwise disjoint right cylinders in $\Sigma^{*}$, then
$$
\mu^{*} (\cup_{l=1}^{k} A_{k}) =\sum_{l=1}^{k} \mu^{*} (A_{l}).
$$
Also $\mu^{*}$ satisfies the continuity law in the sense that if
$\{ A_{n}\}_{n=1}^{\infty}$ is a sequence of decreasing cylinders
and tends to the empty set, then $\mu^{*} (A_{n})$ tends to zero
as $n$ goes to $\infty$. The reason is that since a cylinder of
$\Sigma^{*}$ is a compact set, a sequence of decreasing cylinders
tending to the empty set must be eventually all empty. The Borel
$\sigma$-algebra in $\Sigma^{*}$ is generated by all right
cylinders. So $\mu^{*}$ extends to a measure on $\Sigma^{*}$. We
have the following proposition.

\begin{proposition}
The probability measure $\mu^{*}$ is a $\sigma^{*}$-invariant
probability measure.
\end{proposition}

\begin{proof} We have seen that $\mu^{*}$ is a measure on
$\Sigma^{*} $. Since $\mu^{*}(\Sigma^{*})=1$, it is a probability
measure. For any right cylinder $[w_{n}^{*}]$,
$$
\mu^{*} ((\sigma^{*})^{-1} ([w_{n}^{*}]) = \mu^{*}
(\cup_{j=0}^{d-1} [w_{n}^{*}j]) =\sum_{j=0}^{d-1} \mu^{*}
([w_{n}^{*}j]) = \sum_{i=0}^{d-1} \mu ([w_{n}i])
$$
$$
= \mu (\cup_{i=0}^{d-1} [w_{n}i]) = \mu ([w_{n}])
=\mu^{*}([w_{n}^{*}]).
$$
So $\mu^{*}$ is $\sigma^{*}$-invariant. We proved the proposition.
\end{proof}

We call $\mu^{*}$ a dual $\sigma^{*}$-invariant probability
measure. A natural question now is as follows.

\vspace*{10pt}
\begin{question}
Is a dual invariant probability measure a Gibbs measure with
respect to some continuous or H\"older continuous potential on
$\Sigma^{*}$?
\end{question}

Some more interesting geometric questions from the Teichm\"uller
point of view are the followings. Consider a metric induced from
the dual probability invariant measure $\mu^{*}$ (in the case that
$\mu^{*}$ is supported on the whole $\Sigma^{*}$ and has no atomic
point), that is,
$$
d(w^{*}, \tilde{w}^{*}) =\mu^{*}([w_{n}^{*}])
$$
where $[w_{n}^{*}]$ is the smallest right cylinder containing both
$w^{*}$ and $\tilde{w}^{*}$.

\vspace*{10pt}
\begin{question}
Is $\sigma^{*}$ differentiable under the metric $d(\cdot, \cdot)$?
More precisely, does the limit
$$
\frac{d\sigma^{*}}{dw^{*}} (w^{*})= \lim_{n\to \infty}
\frac{\mu^{*}(\sigma^{*}([w_{n}^{*}]))}{\mu^{*}([w_{n}^{*}])}
$$
exists for every $w^{*}=\cdots w_{n}^{*} \in \Sigma^{*}$? If it
exists, is the limiting function continuous or H\"older continuous
on $\Sigma^{*}$?
\end{question}

\vspace*{10pt}
\begin{question}
Given a positive continuous or H\"older continuous function
$\psi^{*} (w^{*})$ on $\Sigma^{*}$. Can we find a
$\sigma^{*}$-invariant measure $\mu^{*}$ on $\Sigma^{*}$, such
that the right shift map $\sigma^{*}$ under the metric $d(\cdot,
\cdot)$ induced from this measure is $C^{1}$ with the derivative
$\psi^{*} (w^{*})$?
\end{question}

Actually, there is a measure-theoretical version related to these
questions. I will first give a review of this theory.

\section{$g$-measures}

Let $X=\Sigma^{*}$ (or $\Sigma$) and let $f$ be $\sigma^{*}$ (or
$\sigma$). Let ${\mathcal B}$ be the Borel $\sigma$-algebra of
$X$. Let ${\mathcal M}(X)$ be the space of all finite Borel
measures on $X$. Let ${\mathcal M}(X, f)$ be the space of all
$f$-invariant probability measures in ${\mathcal M}(X)$. Let
${\mathcal C} (X)$ be the space of all continuous real functions
on $X$. Then ${\mathcal M}(X)$ is the dual space of ${\mathcal
C}(X)$. Denote
$$
<\phi, \mu> =\int_{X} \phi (x) d\mu, \quad \hbox{$\phi \in
{\mathcal C}(X)$ and $\mu\in {\mathcal M}(X)$}.
$$

A real non-negative continuous function $\psi$ on $X$ is called a
$g$-function~\cite{Keane} if
$$
\sum_{fy=x} \psi(y) =1.
$$

For a $g$-function $\psi (x)$, define the transfer operator
${\mathcal L}_{\psi}$ from ${\mathcal C}(X)$ into itself as
$$
{\mathcal L}_{\psi} \phi (x) =\sum_{f(y)=x} \phi (y) \psi(y),
\quad \phi \in {\mathcal C}(X).
$$
One can check that ${\mathcal L}_{\psi} \phi = {\mathcal L}_{1}
(\psi \phi)$ and if $\psi$ is a $g$-function, then ${\mathcal
L}_{\psi}1=1$. Let ${\mathcal L}_{\psi}^{*} $ be the dual operator
of ${\mathcal L}_{\psi}$, that is, ${\mathcal L}_{\psi}^{*}$ is
the operator from ${\mathcal M}(X)$ into itself satisfying
$$
<\phi, {\mathcal L}_{\psi}^{*}\mu> =<{\mathcal L}_{\psi}\phi,
\mu>, \quad \forall\; \phi \in {\mathcal C}(X) \;\; \hbox{and}\;\;
\forall\; \mu \in {\mathcal M}(X).
$$

\vspace*{10pt}
\begin{definition}~\label{gmeasure}
Suppose $\psi$ is a $g$-function. Then a probability measure
$\mu\in {\mathcal M}(X)$ is called a $g$-measure for $\psi$ if it
is a fixed point of ${\mathcal L}_{\psi}^{*}$, that is,
$$
{\mathcal L}_{\psi}^{*} \mu =\mu.
$$
\end{definition}

\begin{lemma}
Suppose $\psi$ is a $g$-function. Then any $g$-measure $\mu$ for
$\phi$ is an $f$-invariant measure.
\end{lemma}

\begin{proof}
For any Borel set $B\in {\mathcal B}$,
$$
\mu (f^{-1}(B)) = <1_{f^{-1}(B)}, \mu>= < 1_{B}\circ f, {\mathcal
L}_{\psi}^{*} \mu>
$$
$$
= <{\mathcal L}_{\psi} 1_{B}\circ f, \mu>=<1_{B}, \mu>=\mu (B).
$$
So $\mu$ is $f$-invariant.
\end{proof}

For any $\mu\in {\mathcal M}(X)$, let $\tilde{\mu} ={\mathcal
L}_{1}^{*}\mu$.

\vskip5pt
\begin{lemma}
$$
\tilde{\mu} (B) = \sum_{j=0}^{d-1} \mu (f(B\cap [j])),
$$
where $B$ is any Borel subset in ${\mathcal B}$ and $[j]$ is the
right (or left) cylinder of $j$. Moreover, if $\mu\in {\mathcal
M}(X,f)$, $\mu$ is absolutely continuous with respect to
$\tilde{\mu}$.
\end{lemma}

\vskip5pt

\begin{proof}
For any Borel subset $B\in {\mathcal B}$,
$$
\tilde{\mu} (B) = <1_{B}, {\mathcal L}_{1}^{*}\mu> = <{\mathcal
L}_{1} 1_{B}, \mu>.
$$
But
$$
{\mathcal L}_{1} 1_{B} (x) = \sum_{j=0}^{d-1} 1_{B}(xj) =
\sum_{j=0}^{d-1} 1_{f(B\cap [j])} (x).
$$
So we have that
$$
\tilde{\mu} (B) = \sum_{j=0}^{d-1} \mu (f(B\cap [j])).
$$

If $\mu$ is $f$-invariant, then we have that
$$
\tilde{\mu} (B) = \sum_{j=0}^{d-1} \mu (f(B\cap [j])) =
\sum_{j=0}^{d-1} \mu (f^{-1}(f(B\cap [j])))\geq \sum_{j=0}^{d-1}
\mu (B\cap [j]) =\mu (B).
$$
Therefore, $\mu(B)=0$ whenever $\tilde{\mu}(B) =0$. So $\mu$ is
absolutely continuous with respect to $\tilde{\mu}$.
\end{proof}

Suppose $\mu\in {\mathcal M}(X,f)$. Then $\mu$ is absolutely
continuous with respect to $\tilde{\mu}$. So the Radon-Nikod\'ym
derivative
$$
D_{\mu}(x)=\frac{d\mu} {d\tilde{\mu}} (x), \quad
\tilde{\mu}-a.e.\; x
$$
of $\mu$ with respect to $\tilde{\mu}$ exists $\tilde{\mu}$-a.e.
and is a $\tilde{\mu}$-measurable function.

The following theorem is in Leddrapier's paper~\cite{Ledrappier}
and is used in Walters' paper~\cite{Walters} for the study of a
generalized version of Ruelle's theorem. We give a complete proof
here.

\vskip5pt
\begin{theorem}
Suppose $\psi\geq 0$ is a $g$-function and $\mu\in {\mathcal
M}(X)$ is a probability measure. The followings are equivalent:
\begin{itemize}
\item[i)] $\mu$ is a $g$-measure for $\psi$, i.e., ${\mathcal L}_{\psi}^{*}\mu
=\mu$.
\item[ii)] $\mu\in {\mathcal M}(X,f)$ and $D_{\mu} (x) =\psi(x)$ for
$\tilde{\mu}$-a.e. $x$.
\item[iii)] $\mu\in {\mathcal M}(X,f)$ and
$$
E[\phi | f^{-1}({\mathcal B})] (x) = {\mathcal L}_{\psi} \phi
(fx)=\sum_{fy=fx} \psi (y) \phi(y), \; \hbox{for $\mu$-a.e.\; x},
$$
where $E[\phi|f^{-1}({\mathcal B})]$ is the conditional
expectation of $\phi$ with respect to $f^{-1}({\mathcal B})$.
\item[iv)] $\mu\in {\mathcal M}(X, f)$ and is an equilibrium state
for the potential $\log \psi$ with the meaning that
$$
0=h_{\mu}(f) +\int_{X} \log \psi\; d\mu = \sup \{ h_{\nu}(f) +
\int_{X} \log \psi\; d\nu \; |\; \nu \in {\mathcal M}(X,f)\}.
$$
(Note that the pressure $P(\log \psi)=0$ for a $g$-function
$\psi$.)
\end{itemize}
\end{theorem}

\begin{proof}
We first note that since $C(X)$ is dense in the space
$L^{1}(\tilde{\mu})$ of all $\tilde{\mu}$-measurable and
integrable functions (as well as in the space $L^{1}(\mu)$), then
$<\cdot,\cdot>$ can be extended to $L^{1}(\tilde{\mu})$ (as well
as to $L^{1}(\mu)$). We already know that a $g$-measure for $\psi$
is $f$-invariant.

First, we prove $i)$ implies $ii)$. For any $\phi(x) \in C(X)$,
$$
<\phi \psi, \tilde{\mu}> = <\phi \psi, {\mathcal L}_{1}^{*}\mu>
=<{\mathcal L}_{1}(\phi \psi), \mu>= <{\mathcal L}_{\psi}\phi,
\mu>
$$
$$
=<\phi , {\mathcal L}_{\psi}^{*}\mu> = <\phi, \mu>= <\phi D_{\mu},
\tilde{\mu}>
$$
Thus $D_{\mu}=\psi$ for $\tilde{\mu}$-a.e.\;  $x$.

Second, we prove that $ii)$ implies $i)$. Since, for any $\phi(x)
\in C(X)$,
$$
<\phi, \mu>= <\phi D_{\mu}, \tilde{\mu}> =<\phi \psi, \tilde{\mu}>
=<\phi \psi, {\mathcal L}_{1}^{*} \mu>
$$
$$
= <{\mathcal L}_{1}(\phi \psi), \mu> =< {\mathcal L}_{\psi} \phi,
\mu> =< \phi, {\mathcal L}_{\psi}^{*}\mu>.
$$
This implies ${\mathcal L}_{\psi}^{*}\mu=\mu$. Thus $\mu$ is a
$g$-measure for $\psi$.

We prove $i)$ implies $iii)$. For any Borel set $B\in {\mathcal
B}$,
$$
<({\mathcal L}_{\psi} \phi)\circ f \cdot 1_{f^{-1}(B)}, \mu> =
<({\mathcal L}_{1} (\psi\phi))\circ f \cdot 1_{B}\circ f, \mu> =
<{\mathcal L}_{1} (\psi\phi) \cdot 1_{B}, \mu>
$$
$$
=<{\mathcal L}_{1} (\psi\phi 1_{B}\circ f), \mu> =<{\mathcal
L}_{\psi} (\phi 1_{B}\circ f), \mu> = < \phi 1_{B}\circ f, \mu >
=< \phi 1_{f^{-1}(B)}, \mu>
$$
That is,
$$
E[\phi | f^{-1} ({\mathcal B})] =({\mathcal L}_{\psi} \phi)\circ
f, \quad \hbox{$\mu$-a.e. $x$}.
$$
Note that
$$
\big( ({\mathcal L}_{\psi} \phi)\circ f\big) (x) = \sum_{y\in
f^{-1} (fx)} \psi(y)\phi (y).
$$

We now prove that $iii)$ implies $i)$. Since, for any $\phi\in
C(X)$,
$$
E[\phi | f^{-1} ({\mathcal B})] ={\mathcal L}_{\psi} \phi (fx),
\quad \hbox{$\mu$-a.e. $x$},
$$
then,
$$
<\phi, \mu> =  <({\mathcal L}_{\psi} \phi)\circ f, \mu>=
<{\mathcal L}_{\psi} \phi, \mu> = <\phi, {\mathcal
L}_{\psi}^{*}\mu>.
$$
Thus ${\mathcal L}_{\psi}^{*}\mu=\mu$.

We prove that $ii)$ implies $iv)$. For any $\nu\in {\mathcal
M}(X,f)$, let
$$
D_{\nu} = \frac{d\nu}{d\tilde{\nu}}, \quad \tilde{\mu}-a.e. x,
$$
be the Radon-Nikod\'ym derivative. We claim that
$$
h_{\nu}(f) =-\int_{X} \log D_{\nu} d\nu.
$$
We prove this claim. Since $-\log D_{\nu}$ is a non-negative
$\tilde{\nu}$-measurable function and since $\nu$ is absolutely
continuous with respect to $\tilde{\nu}$, it is also a
$\nu$-measurable function. Thus
$$
\int_{X} -\log D_{\nu} d\nu = \int_{X} -D_{\nu}\log D_{\nu}
d\tilde{\nu}.
$$

By the definition,
$$
D_{\nu} (x) =\lim_{n\to \infty} \frac{\nu ([i_{0}i_{1}\cdots
i_{n-1}])}{\nu ([i_{1}\cdots i_{n-1}])}, \quad \tilde{\nu}-a.e.
\;\; x=i_{0}i_{1}\cdots i_{n-1}\cdots.
$$
Let
$$
D_{n,\nu} (x) = \frac{\nu ([i_{0}i_{1}\cdots i_{n-1}])}{\nu
([i_{1}\cdots i_{n-1}])}.
$$
Then
$$
D_{\nu} (x) =\lim_{n\to \infty} D_{n,\nu}(x), \quad
\tilde{\nu}-a.e. \;\; x.
$$
and
$$
-D_{\nu} (x) \log D_{\nu} (x)  =\lim_{n\to \infty} (-D_{n,\nu}(x)
\log D_{n,\nu}(x)), \quad \tilde{\nu}-a.e.\;\; x
$$

Since $-t\log t$ is a positive bounded function on $[0,1]$, by the
Lebesgue control convergence theorem,
$$
\int_{X} \lim_{n\to \infty} (-D_{n,\nu} (x) \log D_{n,\nu} (x))\;
d\tilde{\nu} = \lim_{n\to \infty} \int_{X} -D_{n,\nu}(x) \log
D_{n,\nu}(x)d\tilde{\nu}.
$$
However,
$$
\int_{X} -D_{n,\nu}(x) \log D_{n,\nu}(x)d\tilde{\nu}
$$
$$
= \sum_{[i_{0}\cdots i_{n-1}]} - \frac{ \nu([i_{0}i_{1}\cdots
i_{n-1}])}{\nu([i_{1}\cdots i_{n-1}])} \log \big( \frac{
\nu([i_{0}i_{1}\cdots i_{n-1}])}{\nu([i_{1}\cdots i_{n-1}])}\big)
\tilde{\nu} ([i_{0}i_{1}\cdots i_{n-1}])
$$
$$
=\sum_{[i_{0}\cdots i_{n-1}]} - \nu([i_{0}i_{1}\cdots i_{n-1}])
\log \big( \frac{ \nu([i_{0}i_{1}\cdots
i_{n-1}])}{\nu([i_{1}\cdots i_{n-1}])}\big).
$$
Note that $\tilde{\nu} ([i_{0}i_{1}\cdots
i_{n-1}])=\nu([i_{1}\cdots i_{n-1}])$. But we know that
$$
h_{\nu}(f) = \lim_{n\to \infty} \frac{1}{n} \sum_{[i_{0}\cdots
i_{n-1}]}  - \nu([i_{0}i_{1}\cdots i_{n-1}]) \log
\nu([i_{0}i_{1}\cdots i_{n-1}])
$$
$$
= \lim_{n\to \infty}\sum_{[i_{0}\cdots i_{n-1}]} -
\nu([i_{0}i_{1}\cdots i_{n-1}]) \log \big( \frac{
\nu([i_{0}i_{1}\cdots i_{n-1}])}{\nu([i_{1}\cdots i_{n-1}])}\big).
$$
We proved the claim.

The claim says that $h_{\nu}(f) =-<\log D_{\nu}, \nu>$ for any
$\nu\in {\mathcal M}(X,f)$. Then
$$
h_{\nu}(f) +<\log \psi, \nu> = <\log \frac{\psi}{D_{\nu}}, \nu>
\leq < \frac{\psi}{D_{\nu}}-1, \nu > =<\frac{\psi}{D_{\nu}},\nu>-1
$$
$$
= <\psi, \tilde{\nu}>-1 = < {\mathcal L}_{1}\psi, \nu>-1 =
<1,\nu>-1=1-1=0.
$$
Note that here we use the inequality
\begin{equation}~\label{log}
\log t \leq t-1 \quad \hbox{and}\quad \log t =t-1 \quad\hbox{if
and only if}\quad t=1.
\end{equation}
The assumption in $ii)$ is that $D_{\mu}=\psi$, $\tilde{\mu}-a.e.
x$. But $\mu \ll \tilde{\mu}$, $D_{\mu}=\psi$, $\mu$-a.e. $x$ too.
So we have that
$$
h_{\mu} +\int_{X}\log \psi\; d\mu =0.
$$
So $\mu$ is an equilibrium state for the potential $\log \psi$ in
the meaning that
$$
0=h_{\mu}(f) +\int_{X} \log \psi \; d\mu = \sup \{ h_{\nu}(f) +
\int_{X} \log \psi \; d\nu \; |\; \nu \in {\mathcal M}(X,f)\}.
$$

Last we prove that $iv)$ implies $i)$. Suppose $\mu\in {\mathcal
M}(X,f)$ is an equilibrium state for the potential $\log \psi$. We
have that
$$
h_{\mu}(f) +<\log \psi, \mu>=0.
$$
We already know that
$$
h_{\mu}(f) +<\log D_{\mu}, \mu>=0.
$$
So we have that
$$
h_{\mu}(f) +<\log \psi, \mu> = h_{\mu}(f) +<\log D_{\mu}, \mu>.
$$
Therefore,
$$
0=<\log \psi -\log D_{\mu}, \mu> =< \log \frac{\psi}{D_{\mu}},
\mu>
$$
$$
\leq <\frac{\psi}{D_{\mu}} -1, \mu> = <\frac{\psi}{D_{\mu}},\mu>-1
=<\psi, \tilde{\mu}> -1
$$
$$
= <\psi, {\mathcal L}_{1}^{*}\mu> =<{\mathcal L}_{1}\psi, \mu>-1
=<1,\mu>-1=1-1=0.
$$
Formula~(\ref{log}) implies that
\begin{equation}~\label{muequal}
\frac{\psi(x)}{D_{\mu}(x)}=1,\quad \mu-a.e.\;\; x.
\end{equation}

\begin{remark} This cannot implies that
$$
\frac{\psi(x)}{D_{\mu}(x)}=1, \quad \tilde{\mu}-a.e. \; \;x,
$$
since $\tilde{\mu}$ may not be absolutely continuous with respect
to $\mu$. So this will not implies $ii)$. However, if $\psi (x)
>0$ for all $x\in X$, then
$$
D_{\tilde{\mu}}(x)=\frac{d\tilde{\mu}} {d\mu} (x) =
\frac{1}{D_{\mu}(x)} =\frac{1}{\psi (x)}, \quad \mu-a.e.\; x.
$$
This implies that $\tilde{\mu}$ is absolutely continuous with
respect to $\mu$. Then Equation~(\ref{muequal}) implies $ii)$.
\end{remark}

For any $\phi(x) \in C(X)$,
$$
<\phi, {\mathcal L}_{\psi}^{*}\mu> = <{\mathcal L}_{\psi}\phi,
\mu> = <{\mathcal L}_{1}(\psi\phi), \mu>
$$
$$
= <\psi\phi,  {\mathcal L}_{1}^{*}\mu>=< \phi\psi, \tilde{\mu}> =
<\phi \frac{\psi}{D_{\mu}}, \mu> =<\phi, \mu>.
$$
This says that ${\mathcal L}_{\psi}^{*}\mu=\mu$. We proved $i)$.
\end{proof}

\vspace*{10pt} For any $\sigma$-invariant probability measure
$\mu$, let $\mu^{*}$ be the dual $\sigma^{*}$-invariant
probability measure which we have constructed in the previous
section. Then we have a $\tilde{\mu}$-measurable function
$$
D_{\mu} (w) =\lim_{n\to \infty} \frac{\mu ([w_{n}])}{\mu([\sigma
(w_{n})])}, \quad \hbox{for $\tilde{\mu}$-a.e. $w=w_{n}\cdots \in
\Sigma$}
$$
and a $\tilde{\mu}^{*}$-measurable function
$$
D_{\mu^{*}} (w^{*}) =\lim_{n\to \infty} \frac{\mu^{*}
([w_{n}^{*}])}{\mu^{*}([\sigma^{*}(w_{n}^{*})])}, \quad \hbox{for
$\tilde{\mu}^{*}$-a.e. $w^{*}=\cdots w_{n}^{*} \in \Sigma^{*}$}.
$$
Now a question related to those questions in the end of the
previous section is as follows.

\vspace*{10pt}

\begin{question}
Can $D_{\mu^{*}}(w^{*})$ (or $D_{\mu}(w)$) be extended to a
continuous or H\"older continuous $g$-function?
\end{question}

In the next two sections, we give an affirmative answer to this
question.

\section{Gibbs measures and dual geometric Gibbs
measures}

Consider $f\in {\mathcal C}^{1+}$. One over the derivative
$1/f'(x)$ can be lifted to a positive H\"older continuous function
$$
\psi (w) =\psi_f (w) =\frac{1}{f'(\pi_{f} (w))}
$$
on the symbolic space $\Sigma$. By thinking of $\log \psi $ as a
potential on $(\Sigma, \sigma)$, there is a unique
$\sigma$-invariant measure $\mu= \mu_{\psi}$ (Gibbs measure for
the potential $\log \psi$) as we have mentioned in the previous
section such that
$$
C^{-1} \leq \frac{\mu ([w_{n}])}{\exp (\sum_{i=0}^{n-1}\log \psi
(\sigma^{i} (w)))}\leq C
$$
for any left cylinder $[w_{n}]$ and any $w=w_{n}\cdots \in
[w_{n}]$, where $C$ is a fixed constant. (Note that $P=P(\log
\psi)=0$ in this case.)

Every element $\Phi=[(f,h_{f})]$ in the Teichm\"uller space
${\mathcal T}{\mathcal C}^{1+}$ can also be represented by the
Gibbs measure $\mu$ for the potential $\log \psi (w)$. The reason
is that for every $(g, h_{g})\in \Psi$, $h=h_{f}\circ h_{g}^{-1}$
is a $C^{1}$ diffeomorphism of $T$ such that
$$
f(h(x)) =h(g (x)).
$$
Then
$$
f'(h(x)) h'(x) = h'(g(x)) g'(x).
$$
Therefore,
$$
\log \psi_{f}(w) -\log \psi_{g} (w) = \log h' (w) - \log h'(\sigma
(w)).
$$
So $\psi_{g}$ and $\psi_{f}$ are cohomologous to each other.

The Gibbs measure $\mu$ in this context enjoys the following
geometric property too: The push-forward measure
$$
\mu_{geo} =(\pi_{f})_{*}\mu
$$
is a $C^{1+\alpha}$ smooth $f$-invariant measure for some
$0<\alpha\leq 1$. This means that there is a $C^{\alpha}$ function
$\rho$ on $T$ such that
$$
\mu_{geo} (A) =\int_{A} \rho (x) dx, \quad \hbox{for all Borel
subsets $A$ on $T$}.
$$

There is another way to find the density $\rho$. First it is a
standard method to find an invariant measure for a dynamical
system $f$. Let $\mu_{0}$ be the Lebesgue measure on $T$. Consider
the push-forward measure $\mu_{n}= (f^{n})_{*}\mu_{0}$ by the
$n^{th}$ iterates of $f$. Sum up these measures to get
$$
\nu_{n} = \frac{1}{n} \sum_{k=0}^{n-1} \mu_{n}.
$$
Any weak limit of a subsequence of $\{ \nu_{n}\}$ will be an
$f$-invariant measure. Since we start with an $f\in {\mathcal
C}^{1+}$, we can prove that the sequence $\{\mu_{n}\}$ is actually
convergent in the $C^{1}$ topology to a $C^{1+\alpha}$ smooth
measure $\mu_{geo}$ for some $0<\alpha\leq 1$ as follows:  Each
$\mu_{n}= (f^{n})_{*}\mu_{0}$ has an $\alpha$-H\"older continuous
density
$$
\rho_{n} (x) =\sum_{f^{n}(y)=x} \frac{1}{(f^{n})'(y)}.
$$
Following the theory of transfer operators (refer
to~\cite{Jiang8}), $\rho_{n}(x)$ converges uniformly to an
$\alpha$-H\"older continuous function $\rho (x)$. Thus
$$
\mu_{geo} (A) =\int_{A} \rho (x) dx
$$
is the limit of $\mu_{n}$ and is a $C^{1+\alpha}$ smooth
$f$-invariant probability measure.

Let $y= h(z) =\mu_{geo}([1,z])$ be the distribution function of
$\mu_{geo}$, where $[1,z]$ is the oriented arc on $T$ from $1$ to
$z$. Then $\varsigma=h(z)$ is a $C^{1+\alpha}$-diffeomorphism of
$T$. Let
$$
g(\varsigma) = h\circ f\circ h^{-1} (\varsigma), \quad z=
h^{-1}(\varsigma).
$$
(Note that $g$ here means a circle endomorphism, not a
$g$-function!) Then $g$ preserves the Lebesgue measure
$d\varsigma$ (which means that $g_{*}(d\varsigma)=d\varsigma$, or
equivalently, the Lebesgue measure is $g$-invariant). Since the
Lebesgue measure is an ergodic $g$-invariant measure, $(g, h_{g})$
is unique in the Teichm\"uller point $\Pi=[(f, h_{f})]$.

By considering $\psi_{g}(w) =1/g'(\pi_{g} (w))$, then
$\psi_{g}(w)$ is a $g$-function on $\Sigma$ and $\mu$ is a
$g$-measure. Thus $\mu$ is an equilibrium state for the potential
$\log \psi_{g}(w)$. It follows that $\mu_{geo}$ is also an
equilibrium state for the potential $-\log f'(x)$, that is,
$$
0=P (-\log f'(x))= h_{\mu_{geo}} (f) -\int_{T} \log f'(x)
d\mu_{geo}
$$
$$
= h_{\mu_{geo}} (f) -\int_{T} \log f'(x) \rho (x) dx
$$
$$
=\sup \{ h_{\nu}(f) -\int_{T} \log f'(x) d\nu \; |\; \hbox{$\nu$
is an $f$-invariant propbability measure}\}
$$
$$
= h_{Leb} (g) -\int_{T} \log g'(y) dy,
$$
where $h_{\mu_{geo}}$ and $h_{Leb}(g)$ denote the
measure-theoretical entropies with respect to $\mu_{geo}$ and the
Lebesgue measure. The equilibrium state $\mu_{geo}$ is unique in
this case.

Now by considering the dual invariant probability measure
$\mu^{*}$ for this Gibbs measure $\mu$, we have that

\vspace*{10pt}
\begin{theorem}~\label{der1}
Suppose $f\in {\mathcal C}^{1+}$. Consider $\Sigma^{*}$ with the
metric $d(\cdot, \cdot)$ induced from $\mu^*$ on $\Sigma^{*}$.
Then the right shift $\sigma^{*}$ is $C^{1+}$ differentiable with
respect to $d(\cdot, \cdot)$ and its derivative is the dual
derivative $D^{*}(f)(w^{*})$ of $f$, i.e.,
$$
\frac{d \sigma^{*}}{dw^{*}} (w^{*}) = D^{*}(f)(w^{*}) , \quad
\forall \; w^{*}\in \Sigma^{*}.
$$
(Note that $\sigma^{*}$ is $C^{1+}$ differentiable means that it
is differentiable and the derivative is a H\"older continuous
function.)
\end{theorem}

\begin{proof}
Suppose $w^{*}=\cdots j_{n-1}\cdots j_{1}j_{0}$ is a point in
$\Sigma^{*}$. Let $w_{n}^{*}=j_{n-1}\cdots j_{1}j_{0}$ and
$v_{n-1}^{*} = j_{n-1}\cdots j_{1}$. Let $I_{w_{n}}$ and
$I_{v_{n-1}}$ be the corresponding intervals in the
$n^{th}$-partition $\eta_{n}$ and the $(n-1)^{th}$-partition
$\eta_{n-1}$.

From the definition,
$$
\mu^{*} ([w_{n}^{*}]) = \mu ([w_{n}]) =\mu_{geo} (I_{w_{n}})
$$
and
$$
\mu^{*} ([v_{n-1}^{*}]) = \mu ([v_{n-1}]) =\mu_{geo}
(I_{v_{n-1}}).
$$
Consider the ratio
$$
\frac{\mu^{*} ([v_{n-1}^{*}])}{\mu^{*} ([w_{n}^{*}])} =
\frac{\mu_{geo} (I_{v_{n-1}})}{\mu_{geo} (I_{w_{n}})} =
\frac{h'(\xi)}{h(\xi')} D^{*}(f)(w^{*}_{n}).
$$
Since the distribution function of $\mu_{geo}$ is a
$C^{1+\alpha}$-diffeomorphism, the ratio $h'(\xi)/h(\xi')$
converges to $1$ exponentially as $n\to \infty$. We also know that
$D^{*}(f)(w^{*}_{n})$ converges $D^{*}(f) (w^{*})$ exponentially
as $n\to \infty$. So there are two constants $C>0$ and $0<\tau<1$
such that
$$
\Big| \frac{\mu^{*} ([v_{n-1}^{*}])}{\mu^{*} ([w_{n}^{*}])}
-D^{*}(f)(w^{*})\Big| \leq C\tau^{n}, \quad \forall n>0.
$$
This implies that
$$
\frac{d \sigma^{*}}{dw^{*}} (w^{*}) =\lim_{n\to \infty}
\frac{\mu^{*} ([v_{n-1}^{*}])}{\mu^{*} ([w_{n}^{*}])}=
D^{*}(f)(w^{*}).
$$
So $\sigma^{*}$ is $C^{1+}$ smooth whose derivative is
$D^{*}(f)(w^{*})$. We proved the theorem.
\end{proof}

Since the convergence in the proof is exponential and
$D^{*}(f)(w^{*})$ is a strictly positive function and $\Sigma^{*}$
is a compact space, we have the Gibbs inequalities:
$$
C^{-1} \leq \frac{\mu^{*} ([w_{n}^{*}])}{\exp (\sum_{l=0}^{n-1}
-\log D^{*}(f) ((\sigma^{*})^{l}(w^{*})))} \leq C
$$
for any right cylinder $[w_{n}^{*}]$ and any $w^{*}$ in this
cylinder, where $C>0$ is a fixed constant.

\vspace*{10pt}
\begin{corollary}~\label{cder1}
The measure $\mu^{*}$ is the Gibbs measure for the potential
$-\log D^{*}(f)(w^{*})$.
\end{corollary}

Thus we call $\mu^{*}$ a dual geometric Gibbs measure for the
potential $-\log D^{*}(f)(w^{*})$ in this paper. Let
$h_{\mu^{*}}(\sigma^{*})$ be the measure-theoretic entropy of
$\sigma^{*}$ with respect to $\mu^{*}$. Since the Borel
$\sigma$-algebra of $\Sigma^{*}$ is generated by all right
cylinders, then $h_{\mu^{*}} (\sigma^{*}) $ can be calculated as
$$
h_{\mu^{*}} (\sigma^{*})  =\lim_{n\to \infty} \frac{1}{n}
\sum_{w_{n}^{*}} \Big( -\mu ([w_{n}^{*}]) \log \mu
([w_{n}^{*}])\Big)
$$
$$
= \lim_{n\to \infty} \sum_{w_{n}^{*}} \Big( -\mu ([w_{n}^{*}])
\log \frac{\mu( [w_{n}^{*}])}{\mu(\sigma([w_{n}^{*}]))}\Big),
$$
where $w_{n}^{*}$ runs over all words $w_{n}^{*}=j_{n-1}\cdots
j_{0}$ of $\{ 0, 1,\cdots, d-1\}$ of length $n$.

\vspace*{10pt}
\begin{corollary}~\label{cder2}
The dual geometric Gibbs measure $\mu^{*}$ for the potential
$-\log D^{*}(f)(w^{*})$ is a $g$-measure with respect to the
$g$-function $1/D^{*}(f)(w^{*})$ whose pressure
$$
P(-\log D^{*}(f)) =0.
$$
Moreover, the Radon-Nikod\'ym derivative
$$
D_{\mu^{*}} (w^{*}) =\frac{1}{D^{*}(f) (w^{*})}, \quad \hbox{for
\; $\tilde{\mu}^{*}$-a.e. $w^{*}$},
$$
and $\mu^{*}$ is a unique equilibrium state for the potential
$-\log D^{*}(f)(w^{*})$ in the sense that
$$
0= P (-\log D^{*}(f))= h_{\mu^{*}} (\sigma^{*}) -\int_{\Sigma^{*}}
\log D^{*}(f) (w^{*}) d\mu^{*}(w^{*})
$$
$$
=\sup \big\{ h_{\nu } (\sigma^{*}) -\int_{\Sigma^{*}} \log
D^{*}(f) (w^{*}) d\nu (w^{*})\; |\; \hbox{$\nu$ is a
$\sigma^{*}$-invariant measure}\big\}.
$$
\end{corollary}

Now following Theorem~\ref{der1} and Corollary~\ref{cder2}, we
conclude one of the main results in this paper, which is in some
sense similar to the measurable Riemann mapping theorem for smooth
Beltrami coefficients in the real one-dimensional case.

\vspace*{10pt}
\begin{theorem}~\label{gder1}
Suppose $\Psi^{*} (w^{*})\in {\mathcal T}{\mathcal C}^{1+}$. Then
there is a unique non-atomic measure $\mu^{*}$ whose support is
the whole $\Sigma^{*}$ such that consider the metric $d(\cdot,
\cdot)$ induced from $\mu^{*}$ on $\Sigma^{*}$, the right shift
$\sigma^{*}$ is $C^{1+}$ differentiable and $\Psi^{*}(w^{*})$ is
the derivative, that is,
$$
\frac{d \sigma^{*}}{dw^{*}} (w^{*})  = \Psi^{*} (w^{*}), \quad
\forall\; w^{*}
 \in \Sigma^{*}.
$$
Moreover, by considering the dynamical system
$$
\sigma^{*}: \Sigma^{*}\to \Sigma^{*}
$$
and the given potential $-\log \Psi (w^{*})$, the dual invariant
measure $\mu^{*}$ (or the induced metric $d(\cdot, \cdot)$) is an
equilibrium state for the potential $-\log \Psi (w^{*})$.
\end{theorem}

Suggested by Theorem~\ref{gder1}, we have the following
definition.

\begin{definition}
Suppose $\Psi^{*}(w^{*})$ is a positive continuous function
defined on $\Sigma^{*}$. A non-atomic probability measure
$\mu^{*}$ with support on the whole $\Sigma^{*}$ is a dual
geometric Gibbs type measure for the potential $-\log
\Psi^{*}(w^{*})$ if
$$
\frac{d \sigma^{*}}{dw^{*}} (w^{*}) = \Psi^{*}(w^{*}) , \quad
\forall \; w^{*}\in \Sigma^{*}.
$$
\end{definition}

In the last section, we will discuss the existence of a dual
geometric Gibbs type measure for a continuous potential $-\log
\Psi^{*} (w^{*})$.

\section{Dual geometric Gibbs type
measures for continuous potentials.}

A map $f\in {\mathcal U}{\mathcal S}$ may not be differentiable
everywhere (it may not even be absolutely continuous). There is no
suitable Gibbs theory to be used in the study of geometric
properties of a $\sigma$-invariant measure. We thus turn to the
dual symbolic dynamical system $(\Sigma^{*}, \sigma^{*})$ and
produce a similar dual geometric Gibbs type measure theory as we
did in the previous section.

Suppose $\mu$ is a probability measure on $T$. We call it a
symmetric measure if its distribution function $h(z) =\mu ([1,z])$
is a symmetric circle homeomorphism, where $[1,z]$ means the
oriented arc on $T$ from $1$ to $z$.

An $f$-invariant measure $\mu$ can be found as we did in the
previous section. Let $\mu_{0}$ be the Lebesgue measure. Consider
the push-forward measures $\mu_{n}=(f^{n})_{*}\mu_{0}$ and sum
them up to get
$$
\nu_{n} =\frac{1}{n} \sum_{i=0}^{n-1} \mu_{n}.
$$
Take a weak limit $\mu_{geo}$ of a subsequence of $\{\nu_{n}\}$.
Then $\mu_{geo}$ is an $f$-invariant probability measure. In the
following we will prove that $\mu_{geo}$ is a symmetric
$f$-invariant probability measure.

Actually we will prove that the sequence of the distribution
functions $\{ h_{n}(z)\}_{n=0}^{\infty}$ of $\{
\nu_{n}\}_{n=0}^{\infty}$ has a convergent subsequence. And every
convergent subsequence converges to the distribution function
$h(z)$ of $\mu_{geo}$ and $h(z)$ is symmetric.

Let $H_{n}(x)$ be the lift of $h_{n}(z)$ to the real line
${\mathbb R}$. Then
$$
H_{n}(x) = \frac{1}{n} \sum_{k=0}^{n-1} \sum_{l=0}^{d^{k}-1}
|F^{-k} ([l, l+x])|.
$$

\vspace*{10pt}
\begin{theorem}~\label{symmetricmeasure}
Suppose $f$ is a uniformly symmetric circle endomorphism. Then the
sequence $\{ H_{n}(x) \}_{n=0}^{\infty}$ has a convergent
subsequence in the maximal norm on ${\mathbb R}$. Every convergent
subsequence converges in the maximal norm on ${\mathbb R}$ to a
symmetric circle homeomorphism. Thus the sequence $\{
h_{n}(z)\}_{n=0}^{\infty}$ has a convergent subsequence in the
maximal norm on $T$. Every convergent subsequence converges to a
symmetric circle homeomorphism $h(z)$ and the corresponding
subsequence of probability measures $\{\mu_{n}\}_{n=0}^{\infty}$
converges in the weak topology to an $f$-invariant symmetric
probability measure $\mu_{geo}$ whose distribution function is
$h(z)$.
\end{theorem}

\begin{proof}
Since $f$ is uniformly symmetric, there is a bounded positive
function $\varepsilon (t)>0$ with $\varepsilon (t) \to 0$ as $t\to
0^{+}$ such that
$$
\frac{1}{1+\varepsilon (t)} \leq
\frac{|F^{-n}(x+t)-F^{-n}(x)|}{|F^{-n}(x) -F^{-n}(x-t)|} \leq
1+\varepsilon (t), \quad \forall x\in {\mathbb R}, t>0.
$$
Let $C>0$ be an upper bound of $\epsilon (t)$.

From the definition of $H_{n}$,
$$
H_{n}(\frac{1}{2}) =\frac{1}{n} \sum_{k=0}^{n-1}
\sum_{l=0}^{d^{k}-1} |F^{-k} ([l, l+\frac{1}{2}])|.
$$
Since $H_{n}(0)=0$ and $H_{n}(1)=1$,
$$
\frac{1}{1+C^{-1}}\leq \frac{|F^{-k} ([l,
l+\frac{1}{2}])|}{|F^{-k} ([l, l+1])|} \leq \frac{1}{1+C}.
$$
This implies that
$$
\frac{1}{1+C^{-1}}\leq H_{n}(\frac{1}{2}) \leq \frac{1}{1+C}.
$$
Similarly,
$$
\frac{1}{1+C^{-1}}\leq
\frac{H_{n}(\frac{1}{4})}{H_{n}(\frac{1}{2})} \leq \frac{1}{1+C}.
$$

Since $\{ H_{n}(x)\}_{n=0}^{\infty}$ is a sequence of
quasisymmetric circle homeomorphisms whose quasisymmetric
constants are bounded uniformly by $C$, and since the distances of
the images of any two points in $\{0, 1/4, 1/2, 1\}$ under $H_{n}$
are greater than a constant uniformly on $n$, $\{
H_{n}(x)\}_{n=0}^{\infty}$ is in a compact set in the space of all
quasisymmetric circle homeomorphisms. Thus $\{
H_{n}(x)\}_{n=0}^{\infty}$ has a convergent subsequence $\{
H_{n_{i}}(x)\}_{i=0}^{\infty}$ in the maximal norm on ${\mathbb
R}$ whose limiting function $H(x)$ is a circle homeomorphism.
Furthermore, since the sequence $\{ H_{n}\}_{n=0}^{\infty}$ is
uniformly symmetric, that is,
$$
\frac{1}{1+\varepsilon (t)} \leq
\frac{|H_{n}(x+t)-H_{n}(x)|}{|H_{n}(x) -H_{n}(x-t)|} \leq
1+\varepsilon (t), \quad \forall \; x\in {\mathbb R},\;\; t>0,
$$
the limiting circle homeomorphism $H(x)$ is also symmetric, and
$$
\frac{1}{1+\varepsilon (t)} \leq \frac{|H(x+t)-H(x)|}{|H(x)
-H(x-t)|} \leq 1+\varepsilon (t), \quad \forall \; x\in {\mathbb
R},\;\; t>0.
$$

Since $H_{n}(x)$ is the lift of $h_{n}$, $\{
h_{n_{i}}(x)\}_{i=0}^{\infty}$ is a convergent subsequence in the
maximal norm on $T$ whose limiting function $h(x)$ is a circle
homeomorphism whose lift is $H(x)$. Since $h_{n}(z)$ is the
distribution function of $\nu_{n}$, so $\{
\nu_{n_{i}}\}_{i=0}^{\infty}$ is a convergent subsequence in the
weak topology and converges to $\mu_{geo}$ whose distribution
function is $h(z)$. So $\mu_{geo}$ is a symmetric measure.
\end{proof}

We now lift $\mu_{geo}$ to $\Sigma$ to get a $\sigma$-invariant
measure $\mu$ as follows. For any finite word $w_{n}=i_{0}\cdots
i_{n-1}$, consider the left cylinder $[w_{n}]$. Define
$$
\mu ([w_{n}]) =\mu_{geo} (I_{w_{n}}),
$$
where $I_{w_{n}}$ is the interval in $\eta_{n}$ labeled by
$w_{n}$. One can check that it satisfies the finite additive law
and the continuity law. So it can be extended to a
$\sigma$-invariant probability measure $\mu$ on $\Sigma$ such that
$$
(\pi_{f})_{*} \mu=\mu_{geo}.
$$
For $\mu$, we can construct its
dual invariant measure $\mu^{*}$ on $\Sigma^{*}$ as we did in the
previous two sections. Then we have the following dual geometric
Gibbs type property as we had before in the smooth case:

\vspace*{10pt}
\begin{theorem}~\label{qsder1}
Suppose $f\in {\mathcal U}{\mathcal S}$. Consider $\Sigma^{*}$
with the metric $d(\cdot, \cdot)$ induced from $\mu^*$ on
$\Sigma^{*}$. Then the right shift $\sigma^{*}$ is $C^{1}$
differentiable with respect to $d(\cdot, \cdot)$ and its
derivative is the dual derivative $D^{*}(f)(w^{*})$ of $f$, i.e.,
$$
\frac{d \sigma^{*}}{dw^{*}} (w^{*}) = D^{*}(f)(w^{*}) , \quad
\forall\; w^{*}\in \Sigma^{*}.
$$
\end{theorem}

\begin{proof}
Suppose $w^{*}=\cdots j_{n-1}\cdots j_{1}j_{0}$ is a point in
$\Sigma^{*}$. Let $w_{n}^{*}=j_{n-1}\cdots j_{1}j_{0}$ and
$v_{n-1}^{*} = j_{n-1}\cdots j_{1}$. Let $I_{w_{n}}$ and
$I_{v_{n-1}}$ be the corresponding intervals in the
$n^{th}$-partition $\eta_{n}$ and the $(n-1)^{th}$-partition
$\eta_{n-1}$.

From the definition,
$$
\mu^{*} ([w_{n}^{*}]) = \mu ([w_{n}]) =\mu_{geo} (I_{w_{n}})
$$
and
$$
\mu^{*} ([v_{n-1}^{*}]) = \mu ([v_{n-1}]) =\mu_{geo}
(I_{v_{n-1}}).
$$
Consider the ratio
$$
\frac{\mu^{*} ([v_{n-1}^{*}])}{\mu^{*} ([w_{n}^{*}])} =
\frac{\mu_{geo} (I_{v_{n-1}})}{\mu_{geo} (I_{w_{n}})}.
$$
Since the distribution function $h(z)$ of $\mu_{geo}$ is
symmetric, from the quasisymmetric distortion property
(Lemma~\ref{qsdistortion}), the sequence
$$
\frac{\mu_{geo} (I_{v_{n-1}})}{\mu_{geo} (I_{w_{n}})} - \frac{
|I_{v_{n-1}}|}{|I_{w_{n}}|} =\frac{\mu_{geo}
(I_{v_{n-1}})}{\mu_{geo} (I_{w_{n}})}- D^{*}(f) (w^{*}_{n})
$$
converges to $0$  as $n\to \infty$ uniformly on $w^{*}$. This
implies that
$$
\frac{d \sigma^{*}}{dw^{*}} (w^{*}) =\lim_{n\to \infty}
\frac{\mu^{*} ([v_{n-1}^{*}])}{\mu^{*} ([w_{n}^{*}])}=
D^{*}(f)(w^{*}).
$$
So $\sigma^{*}$ is $C^{1}$ under the metric $d(\cdot, \cdot)$
induced from $\mu^{*}$ whose derivative is $D^{*}(f)(w^{*})$. We
proved the theorem.
\end{proof}

Finally, we conclude one of the main results in this paper, which
is in some sense similar to the measurable Riemann mapping theorem
for general Beltrami coefficients in the one-dimensional case.

\vspace*{10pt}
\begin{theorem}~\label{der2}
Suppose $\Psi^{*}(w^{*})\in {\mathcal T}{\mathcal U}{\mathcal S}$.
Then there is a dual geometric Gibbs type measure $\mu^{*}$ for
the continuous potential $-\log \Psi^{*}(w^{*})$. It is a
$g$-measure for the $g$-function $1/\Psi^{*}(f)(w^{*})$ whose
pressure
$$
P(-\log D^{*}(f)) =0.
$$
Moreover, the Radon-Nikod\'ym derivative
$$
D_{\mu^{*}} (w^{*}) =\frac{1}{\Psi^{*}(f)(w^{*})}, \quad
\hbox{for\; $\tilde{\mu}^{*}$-a.e. $w^{*}$}.
$$
And, furthermore, the dual invariant measure $\mu^{*}$ is an
equilibrium state for the continuous potential $-\log
\Psi^{*}(f)(w^{*})$ in the sense that
$$
0= P (-\log D^{*}(f))= h_{\mu^{*}} (\sigma^{*}) -\int_{\Sigma^{*}}
\log \Psi^{*}(f) (w^{*}) d\mu^{*}(w^{*})
$$
$$
=\sup \big\{ h_{\nu } (\sigma^{*}) -\int_{\Sigma^{*}} \log
\Psi^{*}(f) (w^{*}) d\nu (w^{*})\; |\; \hbox{$\nu$ is a
$\sigma^{*}$-invariant measure}\big\}.
$$
\end{theorem}

\section{Symmetric invariant measure and metric entropy}

Suppose $\tau\in {\mathcal T}{\mathcal U}{\mathcal S}$ and suppose $f\in \tau$. From Theorem~\ref{symmetricmeasure},
there is an $f$-invariant symmetric measure $\mu_{geo}$. Let $h(z)=\mu_{geo}([1,z])$ be the distribution function
of $\mu_{geo}$. Then
$$
\tilde{f} =h\circ f\circ h^{-1}\in \tau
$$
preserves the Lebesgue measure $Leb$ on $T$. This means that the Lebesgue measure $Leb$ is $\tilde{f}$-invariant.
Let $h_{\mu_{geo}}(f)$ be the metric entropy of $f$ with respect to $\mu_{geo}$. Then we have that
$$
h_{\mu_{geo}}(f) = h_{Leb} (\tilde{f})=h_{\mu^{*}}(\sigma^{*}).
$$
From Theorem~\ref{der2}, $h_{\mu_{geo}}(f)$ is a positive number.

If the topological degree of $f$ is $d\geq 2$. Then the topological entropy of $f$ is $\log d$, which is the maximum value of the metric entopy
$h_{\mu_{geo}}(f)$ over all $\tau\in {\mathcal T}{\mathcal U}{\mathcal S}$ and all $f\in \tau$ with their symmetric $f$-invariant
measures $\mu_{geo}$.

\vspace*{10pt}
\begin{theorem}~\label{infvalue}
The infimum of the metric entropy $h_{\mu_{geo}}(f)$ over all $\tau\in {\mathcal T}{\mathcal U}{\mathcal S}$ and all $f\in \tau$ with their symmetric $f$-invariant measures $\mu_{geo}$ is zero.
\end{theorem}

\begin{proof}
To prove this theorem, we construct a family $\{f_{s}\}_{0<s<1}$ of orientation-preserving
circle endomorphisms such that each of them is $C^{1+\alpha}$ expanding for some $0<\alpha\leq 1$ and preserves the
Lebesgue measure. Then the equivalent class $[f_{s}]$ is a point in
$\tau\in {\mathcal T}{\mathcal U}{\mathcal S}$.
Moreover, we prove that the metric entropy $h_{Leb}(f_{s})$ tends to $0$ as $s\to 1^{-}$.
Without loss of generality, we prove this theorem for $d=2$ as follows.

First let us consider the unit circle $T$ as ${\mathbb R}/{\mathbb Z}$. Let $[0,1]$ be a copy of $T$ such that $0=1$.
Consider a piecewise smooth expanding map for any $0<s<1$,
$$
L(x) =\left\{ \begin{array}{ll}
              \frac{x}{s}, & x\in [0, s];\cr
              \frac{x-s}{1-s}, & x\in [s,1]
              \end{array}
      \right.
$$
The Lebesgue measure $Leb$ on $[0,1]$ is the the unique smooth $L$-invariant
measure and the metric entropy
$$
h_{Leb} (L) = s\log s +(1-s) \log (1-s).
$$
Thus $h_{Leb}(L) \to 0$ as $s\to 1^{-}$. Next we will smooth $L$ such
that the resulting map is a $C^{1+\alpha}$ circle expanding endomorphism $f$ and
still preserves the Lebesgue measure $Leb$ on $T$.

Let $r=1-s$. Then
$$
s +r=1.
$$

Let $0<\alpha\leq 1$. Consider the interval $[-r, s]$ and construct a $\alpha$-H\"older
continuous function $\phi(x)$ on it such that
\begin{itemize}
\item[i)]
$$
\phi (x) =\left\{ \begin{array}{ll}
                  r^{-1}, & x\in [-r, -r/2]\cr
                  s^{-1}, & x\in [0,s]
                  \end{array}
          \right.
$$
\item[ii)] $\int_{-r}^{0} \phi (\xi) d\xi =1$,
\item[iii)] $r^{-1}\leq  \phi (x) \leq M r^{-1}, \quad \forall x\in [-r,-r^{2}]$,
\item[iv)] $s^{-1}\leq  \phi (x)\leq r^{-1}, \quad x\in [-r^{2},0]$.
\end{itemize}
Then
$$
\tilde{f}_{0} (x) =\int_{0}^{x}\phi (\xi)d\xi=\frac{x}{s}:
[0,s]\to [0,1]
$$
is a $C^{1+\alpha}$-diffeomorphism and
$$
\tilde{f}_{1}(x) =\int_{s}^{x} \phi (\xi-1)d\xi : [s, 1]\to [0,1]
$$
is a $C^{1+\alpha}$-diffeomorphism. Furthermore,
$$
\tilde{f}_{1}' (1-)= \phi (0-) =\tilde{f}_{0}'(0+) =\phi
(0+)=s^{-1}.
$$
So we define a circle endomorphism $\tilde{f}$ which is
$C^{1+\alpha}$ on $[0,1]\setminus \{s\}$. Moreover,
the derivative $\tilde{f}'(x)\geq \min\{s^{-1}, r^{-1}\}>1$ for any $x\in [0,1]\setminus \{s\}$.

Let $\tilde{g}_{0}$ and $\tilde{g}_{1}$ be the inverses of $\tilde{f}_{0}$ and $\tilde{f}_{1}$.
Consider the interval $I_{0}=[1/2, 1]$. Since $f_{1}'(x) =r^{-1}$ on $[s, (1+s)/2]$,
we have $f_{1}((1+s)/2) =1/2$. Therefore,
the preimage of $I_{0}$
under $\tilde{f}$ is the union of two intervals
$[s/2, s]=\tilde{g}_{0}(I_{0})$ and $[(1+s)/2, 1]=\tilde{g}_{1}(I_{0})$

Define $g_{1} (x)=  \tilde{g}_{1}(x)$ on $[0,1]$. Then its inverse $f_{1}=\tilde{f}_{1}: [s,1]\to [0,1]$
is a $C^{1+\alpha}$ diffeomorphism.
Denote
$$
\psi_{1} (x) = g_{1}'(x) = \frac{1}{\phi (g_{1}(x)-1)}, \quad x\in [0,1].
$$
Then
$$
g_{1} (x) =s+\int_{0}^{x} \psi_{1}(\xi)d\xi, \quad x\in [0,1].
$$

Define
$$
\psi_{0} (x) = 1- \psi_{1}(x) = 1-\frac{1}{\phi (g_{1}(x)-1)},
\quad x\in [0,1].
$$
Then $\psi_{0} (x) = s$ for $x\in [0,1/2]$ since in this case $g_{1}(x) \in [s, (1+s)/2]$ and $g_{1}(x)-1\in [-r, -r/2]$.
Define
$$
g_{0}(x) =\int_{0}^{x} \psi_{0} (\xi)d\xi, \quad x\in [0,1].
$$
It is clearly that $g_{0}(0)=0$ and $g_{0}(1)=s$. So
$$
g_{0}: [0,1]\to [0,s]
$$
is a $C^{1+\alpha}$-diffeomorphism. Let
$$
f_{0}(x): [0,s]\to [0,1]
$$
be the inverse of $g_{0}(x)$. Then it is a
$C^{1+\alpha}$-diffeomorphism. Furthermore,
$$
f_{0}'(s-)= \frac{1}{g_{0}'(1-)}=\frac{1}{1-\frac{1}{\phi(0-)}}=
r^{-1}=f_{1}'(s+).
$$
Thus
$$
f(x) =\left\{ \begin{array}{ll}
              f_{0}(x), & x\in [0, s];\cr
              f_{1}(x), & x\in [s,1]
              \end{array}
      \right.
$$
is a $C^{1+\alpha}$ expanding circle endomorphism.

For any $x\in T$, let $\{ x_{0}, x_{1}\}=f^{-1}(x)$ such that $x_{0}\in [0,s]$ and $x_{1}\in [s,1]$. Then we have that
$$
\frac{1}{f'(x_{0})} +\frac{1}{f'(x_{1})} = g_{0}'(x) +g_{1}'(x) = \psi_{0}(x) +\psi_{1}(x) =1.
$$
This implies that for any interval $J$ of $T$,
$$
Leb (f^{-1}(J)) =Leb (g_{0}(J)) + Leb(g_{1}(J)) =Leb (J).
$$
So $f$ preserves the Lebesgue
measure.

Now we prove that the metric entropy $h_{Leb}(f_{s})$ tends to $0$ as $s\to 1^{-}$.

Let $a=1-f(1-r^{2})$. That is $1-a=f(1-r^{2})$. Since $\phi (x) \leq r^{-1}$ for $x\in [-r^{2},0]$,
$$
a=f(1)-f(1-r^{2}) =\int_{1-r^{2}}^{1} \phi (\xi -1)d\xi \leq r^{2}r^{-1} =r.
$$

If $1-a\leq x\leq 1$, then $1-r^{2}\leq g_{1}(x) \leq 1$. This implies that
$\phi (g_{1}(x)-1)\leq r^{-1}$, that is,
$$
g_{1}'(x)= \psi_{1}(x) \geq r, \quad \forall x\in [1-a,1].
$$
Hence
$$
g_{0}'(x) =1-g_{1}(x) \leq 1-r=s, \quad \forall x\in [1-a,1].
$$
This implies that
$$
s-g_{0}(1-a) =g_{0}(1) -g_{0}(1-a)=\int_{1-a}^{1} g_{0}'(\xi) d\xi \leq sa.
$$
Furthermore,
$$
g_{0}(1-a) \geq s-sa\geq s-sr.
$$

On the other hand, if $0\leq x\leq 1-a$, then $s\leq g_{1}(x) \leq 1-r^{2}$ and $-r\leq g_{1}(x) \leq r^{-2}$.
This implies that
$$
\psi_{1}(x) =\frac{1}{\phi(g_{1}(x)-1)} \geq r, \quad \forall x\in [0,1-a].
$$
Hence
$$
g_{0}'(x) = \psi_{0}(x) =1-\psi_{1}(x) \geq 1-r, \quad \forall x\in [0,1-a].
$$
It follows that for $0\leq x\leq g_{0}(1-a)$,
$$
f'(x)= \frac{1}{g_{0}'(f_{0}(x))} \leq (1-r)^{-1}.
$$

Note that for all $x\in T$, $\phi(x) \leq Mr^{-1}$. It implies that $f'(x) \leq Mr^{-1}$ for all $x\in T$.

Since the Lebesgue measure $Leb$ is $f$-invariant, by Rohlin's formula
$$
h_{Leb} (f) =\int_{T}\log f'(\xi)d\xi
=\int_{0}^{g_{0}(1-a)} \log f'(\xi)d\xi + \int_{1}^{g_{0}(1-a)} \log f'(\xi)d\xi
$$
$$
\leq g_{0}(1-a) \log (1-r)^{-1} +\big( 1-g_{0}(1-a)\big) \log (Mr^{-1})
$$
$$
\leq
-\log s +(1-(s-sr)) \log (Mr^{-1}) \leq -\log s -r(1+s) \log (M^{-1}r).
$$
When $s\to 1^{-}$, $r\to 0^{+}$. This implies
$$
h_{Leb}(f_s) \to 0 \quad \hbox{as} \quad s\to 1^{-}.
$$
We have completed the proof.
\end{proof}

Let $\Psi^{*}(w^{*})$ be the function model of $\tau\in {\mathcal T}{\mathcal U}{\mathcal S}$.
Let $\mu^{*}$ be a dual geometric Gibbs type measure for
the continuous potential $-\log \Psi^{*}(w^{*})$. Let $\mu_{geo}$ be the corresponding symmetric measure
for an $f\in\tau$. Then $h_{\mu^{*}}(\sigma^{*}) =h_{\mu_{geo}}(f)$. Finally, as a consequence of Theorem~\ref{infvalue}, we have that

\begin{theorem}
The maximum value of the metric entropy $h_{\mu^{*}}(\sigma^{*})$ over all $\Psi^{*}(w^{*})\in {\mathcal T}{\mathcal U}{\mathcal S}$ with their dual geometric Gibbs type measures $\mu^{*}$ is $log d$. And the infimum of the metric entropy $h_{\mu^{*}}(\sigma^{*})$ over all $\Psi^{*}(w^{*})\in  {\mathcal T}{\mathcal U}{\mathcal S}$ with their dual geometric Gibbs type measures $\mu^{*}$ is $0$.
\end{theorem}

\begin{remark}
The infimum of the metric entropy for all Anosov diffeomorphisms of a smooth manifold with their SRB measures has been studied in a recent paper~\cite{HuJiangJiang}. The infimum of the metric entropy for all area-preserving Anosov diffeomorphisms of a smooth manifold with their SRB measures is still an open problem. Theorem~\ref{infvalue} answers this problem in the one-dimensional case too.
\end{remark}

\begin{remark} The family $\{f_{s}\}_{0<s<1}$ constructed in the proof of Theorem~\ref{infvalue} is totally degenerate. This means that $\tau_{s}=[f_{s}]$ tends to the boundary of the Teichm\"uler space but its limit can not be seen on the boundary. This is different from the family constructed in~\cite{HuJiangJiang} where the family tends to the boundary of the space of all Anosov diffeomorphisms with a limiting point on the boundary. The limiting point is an almost hyperbolic diffeomorhism. This is also an interesting problem in the Teichm\"uller theory of Riemann surfaces, that is, which curve is totally degenerate and which curve tends to a surface with a parabolic node. The construction of the families in the proof of Theorem~\ref{infvalue} and in~\cite{HuJiangJiang} provides some idea to study this problem.
\end{remark}


\vspace*{20pt}
\bibliographystyle{amsalpha}

\end{document}